\documentclass[reqno]{amsart}
\usepackage[utf8]{inputenc}
\usepackage[T1]{fontenc}
\usepackage{amsthm}
\usepackage{amsmath}
\usepackage{amssymb}
\usepackage{nicefrac}
\usepackage{color}
\usepackage{xcolor}
\usepackage{mathtools} 
\usepackage{enumitem}
\usepackage{hyperref}
\hypersetup{
	colorlinks,
	linkcolor={red!50!black},
	citecolor={blue!50!black},
	urlcolor={blue!20!black}
}
\usepackage[capitalise]{cleveref}

\newtheorem{theorem}{Theorem}

\newtheorem{lem}[theorem]{Lemma}
\newtheorem{kor}[theorem]{Corollary}
\newtheorem{prop}[theorem]{Proposition}

\theoremstyle{definition}
\newtheorem{rem}[theorem]{Remark}

\numberwithin{equation}{section}
\numberwithin{theorem}{section}
\numberwithin{definition}{section}

\DeclareMathOperator\supp{supp}
\DeclareMathOperator\sgn{sgn}
\newcommand{\rev}{\color{black}}
\newcommand{\revv}{\color{black}}
\newcommand{\revvv}{\color{black}}
\newcommand{\revtwo}{\color{black}}
\newcommand{\revvtwo}{\color{black}}

\begin{document}
	
	\title{\normalsize Weighted Sums and Berry-Esseen Type Estimates in Free Probability Theory}
	\author{Leonie Neufeld}
	\thanks{Fakult\"at f\"ur Mathematik,
		Universit\"at Bielefeld, 33501 Bielefeld, 
		Germany;  lneufeld@math.uni-bielefeld.de}
		\thanks{Funded by the Deutsche Forschungsgemeinschaft (DFG, German Research Foundation) -- IRTG 2235 -- 282638148.}
	
	\begin{abstract}
	We study weighted sums of free identically distributed self-adjoint random variables with weights chosen randomly from the unit sphere and show that the Kolmogorov distance between the distribution of such a weighted sum and Wigner's semicircle law is of order \rev $n^{-\nicefrac{1}{2}}$ \revv with high probability. Replacing the Kolmogorov distance by a weaker pseudometric, we obtain a rate of convergence of order  \revtwo $n^{-1}$, thus providing a free analog of the Klartag-Sodin result in classical probability theory. \revvtwo Moreover, we show  that our ideas generalize to the setting of sums of free non-identically distributed bounded self-adjoint random variables \revtwo leading to \revvtwo a new rate of convergence in the free central limit theorem.
	\end{abstract}
	
	\maketitle

	\section{Introduction and main results}
	\subsection{Weighted sums}
	Given a sequence $(X_i)_{i \in \mathbb{N}}$ of free identically distributed self-adjoint random variables, we consider so-called \textit{weighted sums}
	\begin{align*}
		S_\theta := \theta_1X_1 + \dots + \theta_nX_n, \qquad \theta_1^2 + \dots + \theta_n^2 = 1,
	\end{align*}
	where the vector $\theta=(\theta_1, \dots, \theta_n)$ is taken from the unit sphere $\mathbb{S}^{n-1}$ in $\mathbb{R}^n.$  The main goal of this paper is to estimate the rate of convergence of the distribution $\mu_\theta$ of $S_\theta$ to Wigner's semicircle law $\omega.$ \\

	The general motivation behind the analysis of weighted sums and the speed of convergence to their limits comes from classical probability theory. Klartag and Sodin \cite{Klartag2012} were able to prove that weighted sums of classical independent identically distributed (i.i.d.) random variables exhibit a rate of convergence to the standard normal distribution $\gamma$  of order $n^{-1}$  -- at least for most vectors $\theta \in \mathbb{S}^{n-1}$. More precisely, they have shown: Given a sequence of i.i.d.\@ random variables with mean zero, unit variance, and finite fourth moment $m_4$, let $S_\theta$ and $\mu_\theta$ be defined as above. Then, for any fixed $\rho \in (0,1)$, there exist a set $\mathcal{F} \subset \mathbb{S}^{n-1}$ with $\sigma_{n-1}(\mathcal{F}) \geq 1- \rho$ and a constant $C_\rho >0$ such that for all $\theta \in \mathcal{F}$ the Kolmogorov distance $\Delta(\mu_\theta, \gamma)$ between  $\mu_\theta$ and $\gamma$ admits the estimate 
	\begin{align*}
		\Delta(\mu_\theta, \gamma) \leq \frac{C_\rho m_4}{n}.
	\end{align*}
	Here, $\sigma_{n-1}$ denotes the \rev uniform probability measure \revv on the unit sphere.
	Recall that the classical Berry-Esseen theorem asserts that $\Delta(\mu_\theta, \gamma)$ is of order $\sum_{i=1}^n \vert \theta_i \vert^3.$ Since $\sum_{i=1}^n \vert \theta_i \vert^3 \geq n^{-\nicefrac{1}{2}}$ holds for all $\theta \in \mathbb{S}^{n-1}$, Klartag and Sodin's result provides an improved rate of convergence.  For more recent progress on weighted sums in classical probability theory, we refer to works by Bobkov, Chistyakov, and Götze; see \cite{Bobkov2020a} and the references mentioned therein. \\
	
	This paper is mainly concerned with establishing a free analog of the above-mentioned Klartag-Sodin result. In order to be able to evaluate our rates of convergence properly, we briefly outline already existing Berry-Esseen type estimates in free probability theory. Chistyakov and Götze \cite{Chistyakov2013, Chistyakov2008a} studied the rate of convergence in the free central limit theorem under the usual finite third absolute moment \revtwo assumption. \revvtwo More precisely, in \revtwo \cite[Corollary 2.2]{Chistyakov2013}, \revvtwo they established the standard rate of order $n^{-\nicefrac{1}{2}}$ for free identically distributed self-adjoint random variables. In the case of free non-identically distributed self-adjoint random variables, they obtained a rate of convergence given by the square root of the third Lyapunov fraction; compare to \cite[Theorem 2.6]{Chistyakov2008a}. \rev Recently, Maejima and Sakuma \cite{Maejima2022} provided Berry-Esseen type estimates under weaker moment assumptions. \revv Banna and Mai \cite{Banna2021} and Mai and Speicher \cite{Mai2013} considered the corresponding rates in the setting of operator-valued free probability theory.
	
\revtwo Let us now turn to weighted sums in free probability theory. Given $\theta \in \mathbb{S}^{n-1}$ and a sequence of free identically distributed self-adjoint random variables with distribution having mean zero, unit variance, and finite third absolute moment $\beta_3$, we can apply Chistyakov and Götze's result \cite[Theorem 2.6]{Chistyakov2008a} to the distribution $\mu_\theta$ of the corresponding weighted sum $S_\theta$ leading to
\begin{align*}
		\Delta(\mu_\theta, \omega) \leq c\left(\beta_3\sum_{i=1}^n \vert \theta_i \vert^3 \right)^{\nicefrac{1}{2}}
	\end{align*}
	for some absolute constant $c>0$. Together with $\smash{\sum_{i=1}^n \vert \theta_i \vert^3  \geq n^{-\nicefrac{1}{2}}}$, \revvtwo we obtain an a-priori rate for $\Delta(\mu_\theta, \omega)$ which is larger than $\smash{n^{-\nicefrac{1}{4}}}$. Our first and most general result improves this rate to \rev $\smash{n^{-\nicefrac{1}{2}}}$, at least for most choices of $\theta \revtwo \in \mathbb{S}^{n-1}$. 
	\begin{theorem} \label{main theorem} Let $(X_i)_{i \in \mathbb{N}}$ be a sequence of free identically distributed self-adjoint random variables with distribution $\mu$ for some probability measure $\mu$ on $\mathbb{R}$. Assume that $\mu$ has mean zero, unit variance, and finite fourth moment. 
		For $\theta \in \mathbb{S}^{n-1}$ denote the distribution of $S_\theta = \theta_1X_1 + \dots + \theta_nX_n$ by $\mu_\theta$. 
		Let $\rho \in (0,1)$. Then, \revtwo there exist a set $\mathcal{F} \subset \mathbb{S}^{n-1}$ with $\sigma_{n-1}(\mathcal{F}) \geq 1- \rho$ and constants $n_{\rho, \mu} \in \mathbb{N}$, $C_{\rho, \mu}>0$ such that for all $\theta \in \mathcal{F}$ and all $n \geq n_{\rho, \mu}$ we have \revvtwo
		\begin{align*}
			\Delta(\mu_\theta, \omega) \leq  \rev \frac{C_{\rho, \mu}}{\sqrt{n}} \revv.
		\end{align*}
		The constant $C_{\rho, \mu}$ depends on $\rho$ and \revtwo on $\mu$ through its third and fourth absolute moment. 
	\end{theorem}

	On average the Kolmogorov distance $\Delta(\mu_\theta, \omega)$ admits the same rate of convergence as established in the above theorem. We denote the expectation with respect to $\sigma_{n-1}$ by $\mathbb{E}_\theta$.
	\begin{kor} \label{Corollary unbounded} 
		Let $(X_i)_{i \in \mathbb{N}}$, $\mu$, and $\mu_\theta$ be as in Theorem \ref{main theorem}. Then, there exist constants $n_\mu \in \mathbb{N}$ and $C>0$ such that \revtwo for all $n \geq n_\mu$ we have 
		\end{kor}
		\begin{align*}
			\mathbb{E}_\theta \left(\Delta(\mu_\theta, \omega) \right) \leq \rev \frac{C\beta_3(\mu)^2m_4(\mu)}{\sqrt{n}} \revv.
		\end{align*}
	
	\rev The proofs of Theorem \ref{main theorem} and Corollary \ref{Corollary unbounded} are based on ideas introduced by Chistyakov and Götze in their works \cite{Chistyakov2013, Chistyakov2008a} and heavily rely on the method of subordination. 
	
\revvv 
Let us note that the rate of convergence of order $n^{-\nicefrac{1}{2}}$ in the last two statements does not seem to be improvable without imposing further conditions \revtwo on the underlying probability measure \revvtwo such as a vanishing third moment -- at least not with our approach.
We believe that this is due to the fact that the density of Wigner's semicircle law is non-differentiable at $-2$ and $2$, thus making a good approximation at these points more difficult. 
A replacement of Wigner's semicircle law by a non-linear modification of itself (belonging to the \revtwo class \revvtwo of so-called \textit{Meixner measures}) might lead to a rate of convergence going beyond $n^{-\nicefrac{1}{2}}$. \revtwo The density of this modification \revvtwo still has non-differentiable points, but these are shifted in accordance to the skewness of the approximating free additive convolution. We refer to Chistyakov and Götze \cite{Chistyakov2013} for an extensive analysis of approximations to Meixner measures. \\

\revtwo In view of the last observation, \cref{main theorem} does not yield a satisfying free analog of the Klartag-Sodin result. In order to achieve a more appropriate analog, we do not pursue the above-mentioned idea including the class of Meixner measures, but instead \revvtwo focus on deriving an improved rate of convergence of the distribution of a weighted sum to Wigner's semicircle law by replacing the Kolmogorov distance by a new distance-measuring quantity. For this purpose, fix $\varepsilon>0$ and define the pseudometric $\Delta_\varepsilon$ by 
\begin{align} \label{def Kolmogorov epsilon}
	\Delta_\varepsilon(\nu_1, \nu_2) :=	\sup_{x \in [-2+\varepsilon, 2-\varepsilon]} \big\vert \nu_1((-2+\varepsilon, x])  - \nu_2((-2+ \varepsilon, x]) \big\vert
\end{align}
for any two probability measures $\nu_1, \nu_2$ on $\mathbb{R}$. 
Non-rigorously speaking, $\Delta_\varepsilon$ \revtwo only takes into account the behavior of the considered measures inside $(-2,2)$, \revvtwo thus avoiding the points $-2$ and $2$.  \revtwo Finally, using the quantity $\Delta_\varepsilon$ instead of the Kolmogorov distance, we obtain the following theorem:  \revvtwo
\revv

\begin{theorem}  \label{main Kolmogorov epsilon unbounded} 
\revtwo	Let $(X_i)_{i \in \mathbb{N}}$ be a sequence of free identically distributed self-adjoint random variables with distribution $\mu$ for some probability measure $\mu$ on $\mathbb{R}$. Assume that $\mu$ has mean zero, unit variance, and finite sixth moment. \revtwo For $\theta \in \mathbb{S}^{n-1}, $ \revvtwo define $S_\theta$ and $\mu_\theta$ as in Theorem \ref{main theorem}. Let $\varepsilon, \rho \in (0,1).$ Then, there \revtwo exist a set $\mathcal{F} \subset \mathbb{S}^{n-1}$ with $\sigma_{n-1}(\mathcal{F}) \geq 1- \rho$  and constants $n_{\varepsilon, \rho, \mu} \in \mathbb{N}$, $C_{\varepsilon, \rho, \mu}>0$ \revvtwo such that for all $\theta \in \mathcal{F}$ and all $n \geq n_{\varepsilon, \rho, \mu}$ we have 
\begin{align*}
	\Delta_\varepsilon(\mu_\theta, \omega) \leq \revtwo \frac{C_{\varepsilon, \rho, \mu}}{n}.
\end{align*}
The constant $C_{\varepsilon, \rho, \mu}$ depends on $\varepsilon$, $\rho$, and on $\mu$ through its first six absolute moments. \revvtwo
\end{theorem} 

\revtwo We comment on the above moment assumptions later; see \cref{main Kolmogorov unbounded moment assumption}. Instead, let us now explain why Theorem \ref{main Kolmogorov epsilon unbounded}  can be interpreted as a free analog of the Klartag-Sodin result. If we choose $\theta \in \mathbb{S}^{n-1}$ to be the vector $\smash{\theta = (n^{-\nicefrac{1}{2}}, \dots,n^{-\nicefrac{1}{2}})}$ corresponding to the standard normalization, one can show that an upper bound of order $n^{-\nicefrac{1}{2}}$ is optimal for $\Delta_\varepsilon(\mu_\theta, \omega)$, $\varepsilon>0$, among all underlying probability measures with mean zero, unit variance, and finite \revtwo third absolute moment; \revvtwo compare to \cref{counterexample binomial}. Hence, we conclude that the randomization of the weights as in the above theorem has a \revtwo significant \revvtwo improving effect on the rate of convergence \revtwo to Wigner's semicircle law \revvtwo measured with respect to the pseudometric $\Delta_\varepsilon$. \\

\revtwo
Having found a free analog of Klartag and Sodin's result, let us come back to the analysis of the rate of convergence of the distribution of a weighted sum to Wigner's semicircle law in terms of the Kolmogorov distance. 
Below, we define the quantity $\tilde{\Delta}(\mu_\theta, \omega)$ and observe that it is closely related to $\Delta(\mu_\theta, \omega).$ We show that $\tilde{\Delta}(\mu_\theta, \omega)$  is of order $\smash{n^{-\nicefrac{3}{4}}\sqrt{\log n}}$, while the corresponding proof can be interpreted as half of the proof for a bound on $\Delta(\mu_\theta, \omega)$ which is better than the one established in \cref{main theorem}. 

We define \revvtwo
\begin{align} \label{def Delta tilde}
	\tilde{\Delta}(\mu_\theta, \omega) := \sup_{u \in [-2+\nicefrac{\varepsilon}{2}, 2-\nicefrac{\varepsilon}{2}]} \int_{a}^1 \!\! \vert G_{\theta}(u+iv) - G_{\omega}(u+iv) \vert dv + a + \varepsilon^{\nicefrac{3}{2}}
\end{align}
for $a \in (0,1)$ and $\varepsilon \geq 0$, both possibly depending on $n$. Here, $G_\theta$ and $G_\omega$ denote the \textit{Cauchy transforms} of $\mu_\theta$ and $\omega$; see \eqref{def CT} for the definition. According to a version of Bai's inequality, compare to Proposition \ref{Bai - Götze's Version}, the quantities $\smash{\tilde{\Delta}(\mu_\theta, \omega)}$ and $\Delta(\mu_\theta, \omega)$ are \revtwo  connected via \revvtwo
\begin{align} \label{Bai non rigorous}
		\Delta(\mu_\theta, \omega) \leq  C_1\int_{-\infty}^{\infty} \!\! \left\vert G_\theta(u+i) - G_\omega(u+i) \right\vert du   
		+ C_2\tilde{\Delta}(\mu_\theta, \omega),
\end{align}
where $C_1, C_2>0$ are numerical constants. \revtwo As indicated above, we have the following bound on $\tilde{\Delta}(\mu_\theta, \omega)$:\revvtwo

\begin{theorem} \label{main Delta tilde}
	Let $(X_i)_{i \in \mathbb{N}}$, $\mu$, and $\mu_\theta$ be as in \revtwo  Theorem \ref{main Kolmogorov epsilon unbounded}. Choose $\rho \in (0,1).$  \revvtwo Then, there \revtwo exist a set $\mathcal{F} \subset \mathbb{S}^{n-1}$ with $\sigma_{n-1}(\mathcal{F}) \geq 1 -\rho$ and constants $n_{\rho, \mu} \in \mathbb{N}$, $C_{\rho, \mu}>0$ \revvtwo such that for all $\theta \in \mathcal{F}$ and all $n \geq n_{\rho, \mu}$ we have 
	\begin{align*}
		\tilde{\Delta}(\mu_\theta, \omega)   \leq \revtwo C_{\rho, \mu}\frac{\sqrt{\log n}}{n^{\nicefrac{3}{4}}}.
	\end{align*}
	The constant $C_{\rho, \mu}$ depends on \revtwo $\rho$ and on $\mu$ through its first six absolute moments. \revvtwo
\end{theorem}

\revtwo For completeness, let us mention that the choice of the parameter $\varepsilon$ in \eqref{def Delta tilde}, i.e.\@ the distance to the points $-2$ and $2$, limits the final decay rate for $\tilde{\Delta}(\mu_\theta, \omega)$ to the claimed rate. Moreover, as we will see in \cref{counterexample binomial 2}, the order of the upper bound on $\tilde{\Delta}(\mu_\theta, \omega)$ for $\theta = (n^{-\nicefrac{1}{2}}, \dots, n^{-\nicefrac{1}{2}})$ cannot be better than $n^{-\nicefrac{1}{2}}$ in general.

Previously, we indicated that $\Delta(\mu_\theta, \omega)$ might decay faster than $n^{-\nicefrac{1}{2}}$ in the special case that the underlying probability measure has vanishing third moment. Indeed, by proving that the integral in \eqref{Bai non rigorous} is of order $n^{-1}$ in that case and combining this with  \cref{main Delta tilde}, we arrive at the following result:  \revvtwo

\begin{kor} \label{main Delta vanishing third moment} 
	Let $(X_i)_{i \in \mathbb{N}}$, $\mu$, and $\mu_\theta$ be as in \revtwo Theorem \ref{main Kolmogorov epsilon unbounded} and assume that $\mu$ has vanishing third moment. \revvtwo Let $\rho \in (0,1)$. Then, there \revtwo exist a set $\mathcal{F} \subset \mathbb{S}^{n-1}$ with $\sigma_{n-1}(\mathcal{F}) \geq 1 -\rho$ and constants $n_{\rho, \mu} \in \mathbb{N}$, $C_{\rho, \mu}>0$ \revvtwo such that for all $\theta \in \mathcal{F}$ and all $n \geq n_{\rho, \mu}$ we have 
	\begin{align*}
	\Delta(\mu_\theta, \omega)   \leq \revtwo C_{\rho, \mu}\frac{\sqrt{\log n}}{n^{\nicefrac{3}{4}}}.
	\end{align*}
	The constant $C_{\rho, \mu}$ depends on \revtwo $\rho$ and on $\mu$ through its first six absolute moments. \revvtwo
\end{kor}

 \revtwo In our last result on weighted sums, we restrict to the special case that the underlying probability measure has compact support. Our goal is to analyze the resulting bounded weighted sum in view of the so-called concept of \textit{superconvergence} \revvtwo  -- a type of convergence related to the free central limit theorem, which does not have a direct counterpart in classical probability theory. Given a sequence $(X_i)_{i \in \mathbb{N}}$ of free (not necessarily identically distributed) bounded self-adjoint random variables satisfying certain moment constraints, Bercovici and Voiculescu \cite{Bercovici1995} proved that the distribution $\mu_n$ of the sum $n^{-\nicefrac{1}{2}}(X_1 + \dots + X_n)$ has compact support in an interval $[a_n, b_n]$ with $\lim_{n \rightarrow \infty} a_n = -2$, $\lim_{n \rightarrow \infty} b_n = 2$ and is absolutely continuous for sufficiently large $n$. Moreover, the density $\nicefrac{d\mu_n}{dx}$ converges uniformly to the density of Wigner's semicircle law as $n \rightarrow \infty.$ Kargin \cite{Kargin2007} was able to estimate the rate of convergence of the support of $\mu_n$ to $[-2,2]$. Applying his result in the setting of \revtwo Theorem \ref{compact superconvergence result} (see below), \revvtwo we obtain 
	\begin{align*}
		\supp \mu_\theta \subset \left( -2 -5L^3\sum_{i=1}^n \vert \theta_i \vert^3, \, 2 + 5L^3\sum_{i=1}^n \vert \theta_i \vert^3 \right) \subset  \left( -2 - \frac{165L^3}{\sqrt{n}}, \, 2  + \frac{165L^3}{\sqrt{n}}\right)
	\end{align*}
	for sufficiently large $n$ and most $\theta \in \mathbb{S}^{n-1}$.  By careful modifications of Kargin's proof, this rate of convergence can be improved \revtwo to $n^{-1}$ \revvtwo as formulated in the next theorem:

	\begin{theorem} \label{compact superconvergence result}
		\revtwo 
			Let $(X_i)_{i \in \mathbb{N}}$ be a sequence of free identically distributed bounded self-adjoint random variables with distribution $\mu$ for some probability measure $\mu$ on $\mathbb{R}$. Assume that $\mu$ has mean zero, unit variance, and compact support in $[-L,L]$ for some $L>0.$ \revtwo For $\theta \in \mathbb{S}^{n-1}$, \revvtwo define $S_\theta$ and $\mu_\theta$ as in \cref{main theorem}. \revvtwo Let $\rho \in (0,1).$ Then, there \revtwo exist a set $\mathcal{F} \subset \mathbb{S}^{n-1}$ with $\sigma_{n-1}(\mathcal{F}) \geq 1 - \rho$ and constants $n_{\rho, \mu} \in \mathbb{N}$, $C_\rho>0$ \revvtwo such that for all $\theta \in \mathcal{F}$ and all $n \geq n_{\rho, \mu}$ we have
		\begin{align*}
			\supp \mu_\theta \subset \left( -2 - \frac{C_\rho \left(768L^4 + 6\vert m_3(\mu) \vert\right)}{n}, \, 2+  \frac{C_\rho\left(768L^4 + 6\vert m_3(\mu) \vert\right)}{n} \right).
		\end{align*}
		\revtwo Above, $m_3(\mu)$ denotes the third moment of $\mu.$ \revvtwo
	\end{theorem}
	
\rev
	For completeness, let us mention that it is possible to obtain decay rates holding on average (such as in Corollary \ref{Corollary unbounded}) in the settings of Theorems \ref{main Kolmogorov epsilon unbounded}, \ref{main Delta tilde}, and \ref{compact superconvergence result} and in the one of \cref{main Delta vanishing third moment}. Since the exact formulations are immediate, we omit them at this point. \revv

\subsection{Berry-Esseen type estimate in the free central limit theorem}
Before we end the introduction, let us leave the setting of weighted sums and consider sums of free non-identically distributed bounded \revtwo self-adjoint \revvtwo random variables. Our approach can be generalized to such sums providing \revtwo the following \revvtwo new Berry-Esseen type estimate:
	\begin{theorem} \label{Berry esseen bounded non id}
		Let $(X_i)_{i \in \mathbb{N}}$ be a sequence of free -- not necessarily identically distributed -- bounded self-adjoint random variables with distributions $(\nu_i)_{i \in \mathbb{N}}$ for probability measures $\nu_i$ on $\mathbb{R}$. \revtwo Assume that $\nu_i$ has mean zero and  variance $\smash{\sigma_i^2 \in (0, \infty)}$,  $i  \in \mathbb{N}$.  Moreover, suppose that $\supp \nu_i \subset [-T_i, T_i]$ holds for some $T_i >0$. \revvtwo Define
		\begin{align*}
			B_n:= \left(\sum_{i=1}^n \sigma_i^2\right)^{\nicefrac{1}{2}}, \qquad S_n := \frac{1}{B_n} \sum_{i=1}^{n} X_i, \qquad L_n := \frac{\sum_{i=1}^{n} T_i^3}{B_n^3},
		\end{align*}
		and let $\mu_{\boxplus n}$ denote the distribution of $S_n$. Then, we have
		\begin{align*}
			\Delta(\mu_{\boxplus n}, \omega) \leq cL_n
		\end{align*}
		for some absolute constant $c>0.$
	\end{theorem}
	Let us comment on the \revtwo above \revvtwo theorem: Firstly, recall that in the setting of Theorem \ref{Berry esseen bounded non id} the best result known so far was established by Chistyakov and Götze \revtwo \cite[Theorem 2.6]{Chistyakov2008a} \revvtwo providing a rate of convergence of order $\smash{(B_n^{-3}\sum_{i=1}^n \beta_3(\nu_i))^{\nicefrac{1}{2}}}$, where $\beta_3(\nu_i)$ denotes the third absolute moment of $\nu_i$. We were able to remove the square root at the cost of an increase in the \revtwo numerator \revvtwo from $\sum_{i=1}^n \beta_3(\nu_i)$ to $\sum_{i=1}^n T_i^3$. Note that in the case of free identically distributed bounded summands, Theorem \ref{Berry esseen bounded non id} yields the standard Berry-Esseen rate. 
	\rev Secondly, observe that Theorem \ref{Berry esseen bounded non id} implies Theorem \ref{main theorem} in the special case of bounded random variables; compare to Lemma \ref{5.3.3 Bobkov}. \revv
	

\subsection*{Organization} In Section \ref{Section: Preliminaries} we briefly recall the basics of free probability theory and collect some concentration inequalities on the sphere. Section \ref{Section: unbounded} is devoted to the proofs of Theorem \ref{main theorem} and Corollary \ref{Corollary unbounded}, whereas the \revtwo proof of \cref{main Kolmogorov epsilon unbounded} is \revvtwo carried out in Section \ref{Section: Delta epsilon}. \rev The results of Theorem \ref{main Delta tilde} and \cref{main Delta vanishing third moment} \revtwo are verified \revvtwo  in Section \ref{Section: Delta tilde}. Section \ref{Section: Superconvergence} is concerned with the proof of the superconvergence result in Theorem \ref{compact superconvergence result}. Lastly, the proof of Theorem \ref{Berry esseen bounded non id} is outlined in Section \ref{sec: Berry esseen bounded non id}.

\subsection*{Acknowledgments} I would like to thank Friedrich Götze for numerous discussions and helpful feedback on this project. Moreover, I thank the anonymous referee for a number of valuable suggestions, \revtwo which helped to improve some of the results of this work significantly. \revvtwo

	\section{Preliminaries} \label{Section: Preliminaries}
	Before we start with the basics of free probability theory, we briefly introduce some notation concerning the (complex) square root. 
	When writing $\sqrt{x}$ for some $x \in \mathbb{R}, x \geq 0$, we always refer to the principal (real) square root. In contrast to this, we will not work with the principal branch of the complex square root and instead place the branch cut on the non-negative real axis. In more detail, for any $z=re^{i \varphi}$ with $r \geq 0$, $\varphi = \arg z \in (0, 2\pi)$, we set $\sqrt{z} := \sqrt{r} e^{i\nicefrac{\varphi}{2}}$. A simple calculation involving the half angle formula shows that
	\begin{align} \label{square root real and im formula}
		\Re\sqrt{z} = \sgn(v)\sqrt{\frac{1}{2} \left(\sqrt{u^2+v^2} + u\right)}, \qquad \Im\sqrt{z} =  \sqrt{\frac{1}{2}\left( \sqrt{u^2+v^2} - u\right)}
	\end{align}
	hold for $z \in \mathbb{C}\setminus[0, \infty)$ with $z = u+iv$. Here, $\sgn(v)$ denotes the sign of $v$ under the convention $\sgn(0)=1$. Clearly, the imaginary part of $\sqrt{z}$ is always positive and $z \mapsto \sqrt{z}$ is holomorphic in $\mathbb{C} \setminus[0, \infty)$. Last but not least, we remark that the usual multiplication rule for square roots does not hold, i.e.\@ $\sqrt{z_1 \cdot z_2} \neq \sqrt{z_1}\sqrt{z_2}$ in general for $z_1, z_2 \in \mathbb{C}\setminus [0, \infty).$

	\subsection{\texorpdfstring{Basics of free probability theory}{Basics of free probability theory}}  \label{section: basics in FPT}
	In the statements of our main theorems, we chose the operator-theoretic approach to the free additive convolution, which non-rigorously is defined as follows: Let $X$ and $Y$ be free random variables with analytic distributions $\nu_1$ and $\nu_2$. Then, the analytic distribution of the sum $X+Y$ is called the free additive convolution and is denoted by $\nu_1 \boxplus \nu_2$; see for instance \cite{Bercovici1993}. However, in order to prove our results, we will need an analytic characterization of the free additive convolution. In the past, two analytic approaches turned out to be very helpful: the machinery of so-called $R$-transforms and the concept of subordination. Since both approaches will be relevant in our proofs, the main part of this section is devoted to a brief outline of each. After that, we continue by quoting a few results that are related to the Kolmogorov and L\'{e}vy distance. 
	
	The following facts can be found in \cite{Mingo2017}. Choose a probability measure $\nu$ on $\mathbb{R}$ and let
	\begin{align*}
		m_k(\nu) := \int_{\mathbb{R}} x^k \nu(dx), \qquad \beta_k(\nu) := \int_{\mathbb{R}} \vert x \vert ^k \nu(dx), \qquad k \in \mathbb{N},
	\end{align*}
	denote its $k$-th (absolute) moments.
	We define the Cauchy transform of $\nu$ by
	\begin{align} \label{def CT}
		G_\nu(z) := \int_{\mathbb{R}} \frac{1}{z-t} \nu(dt), \qquad z \in \mathbb{C}^+,
	\end{align}
	where $\mathbb{C}^+$ denotes the complex upper half-plane. The Cauchy transform $G_\omega$ of Wigner's semicircle law is given by 
	\begin{align*}
		G_\omega(z) = \frac{1}{2}\left( z - \sqrt{z^2-4} \right), \qquad z \in \mathbb{C}^+.
	\end{align*}
	The measure $\nu$ can be recovered from its Cauchy transform with the help of the Stieltjes-Perron inversion formula: For all $a,b \in \mathbb{R}$ with $a<b$, we have
	\begin{align*}
		\nu ((a,b)) + \frac{1}{2} \nu(\{a,b\}) = -\lim_{\varepsilon \downarrow 0}  \frac{1}{\pi} \int_{a}^b \Im G_\nu (x+i\varepsilon) dx. 
	\end{align*}
	If $\nu$ has compact support in $[-L,L]$ for some $L>0$, then $G_\nu$ is holomorphic in $\smash{\{ z \in \mathbb{C}: \vert z \vert >L \}}$ with power series expansion given by
	\begin{align*}
		G_\nu(z) = \sum_{k=0}^{\infty} \frac{m_k(\nu)}{z^{k+1}}, \qquad \vert z \vert >L.
	\end{align*}
	Using that expansion, it is easy to prove that  $G_{\nu}$ is univalent in $\{ z \in \mathbb{C}: \vert z \vert > 4L\}$. Moreover, we have 
	\begin{align*}
	\big	\{ z \in \mathbb{C}: 0< \vert z \vert < (6L)^{-1} \big\} \subset	G_\nu \big(\{  z \in \mathbb{C}: \vert z \vert > 4L\} \big).
	\end{align*}
	Hence, the functional inverse
	\begin{align*}
		K_\nu(z) := G_\nu^{-1}(z),
	\end{align*}
	also known as the $K$-transform of $\nu$, is well-defined in the punctured disk $\smash{\{ z \in \mathbb{C}: 0< \vert z \vert < (6L)^{-1}\}}$ and satisfies $G_\nu(K_\nu(z)) = z$ for all complex $z$ with $\smash{0<\vert z \vert < (6L)^{-1}}$. Additionally, one can prove that $K_\nu (G_\nu(z)) = z$ holds for all $z$ with $\vert z \vert > 7L$. The $R$-transform $R_\nu$ of $\nu$ is given by 
	\begin{align*}
		R_\nu(z) := K_\nu(z) - \frac{1}{z},
	\end{align*}
	whenever it is defined. It is known that $R_\nu$ is analytic in $\{ z \in \mathbb{C}: \vert z \vert < (6L)^{-1}\}$ with power series expansion given by
	\begin{align*}
		R_\nu(z) = \sum_{m=1}^\infty \kappa_m(\nu) z^{m-1}, \qquad  \vert z \vert < (6L)^{-1}.
	\end{align*}
	Here, $\kappa_m(\nu)$ denotes the $m$-th free cumulant of $\nu$. We obtain that the $K$-transform $K_\nu$ is analytic in $\{ z \in \mathbb{C}: 0 < \vert z \vert < (6L)^{-1}\}$ with Laurent series given by 
	\begin{align*}
		K_\nu(z) = \frac{1}{z} + \sum_{m=1}^\infty \kappa_m(\nu) z^{m-1}, \qquad  0<\vert z \vert < (6L)^{-1}.
	\end{align*}
	By making use of a version of Lagrange's inversion theorem, Kargin was able to establish the following bound on the cumulants
	\begin{align*}
		\vert \kappa_m(\nu) \vert \leq \frac{2L}{m-1} (4L)^{m-1}, \qquad m \geq 2;
	\end{align*}
	see \cite[Lemma 5, Lemma 6]{Kargin2007a} for the details. 
	A crucial benefit of the $R$- and $K$-transform is the fact that both linearize the free additive convolution in the following sense: Letting $\nu_1$ and $\nu_2$ denote two compactly supported probability measures on $\mathbb{R},$ we have 
	\begin{align*}
		R_{\nu_1 \boxplus \nu_2}(z) = R_{\nu_1}(z) + R_{\nu_2}(z), \qquad K_{\nu_1 \boxplus \nu_2}(z) = K_{\nu_1}(z) + K_{\nu_2}(z) - \frac{1}{z}
	\end{align*}
	on the intersection of the corresponding domains. Since cumulants and moments uniquely determine compactly supported probability measures, it is possible to define the free additive convolution of such measures by these linearization identities. Hence, the free additive convolution can be characterized not only as the distribution of a sum of free (bounded) random variables, but also as a binary operation on the space of all compactly supported probability measures. This change of perspective was extended to probability measures with finite variance by Maassen \cite{Maassen1992} and to arbitrary probability measures on the real line by Belinschi and Bercovici \cite{Belinschi2007} and Chistyakov and Götze \cite{Chistyakov2011} by complex analytic methods. In the case of arbitrary probability measures, one introduces a new transform, the reciprocal Cauchy transform, which is given by 
	\begin{align*}
		F_{\nu}(z) := \frac{1}{G_{\nu}(z)}, \qquad z \in \mathbb{C}^+.
	\end{align*}
	Now, the free additive convolution of arbitrary probability measures on $\mathbb{R}$  can be defined solely by the use of the corresponding reciprocal Cauchy transforms as shown in the next theorem. We refer to \cite[Theorem 2.1, Corollary 2.2]{Chistyakov2011} for a proof.
	
	\begin{theorem} \label{subordination functions}
		Let $\nu_1, \dots, \nu_n$ be probability measures on $\mathbb{R}$. There exist unique holomorphic functions $Z_1, \dots, Z_n: \mathbb{C}^+ \rightarrow \mathbb{C}^+$ such that for any $z \in \mathbb{C}^+$ the equations
		\begin{align*}
			Z_1(z) + Z_2(z) + \cdots + Z_n(z) - z = (n-1)F_{\nu_1}(Z_1(z)), \,\,\,\, F_{\nu_1}(Z_1(z)) = \cdots = F_{\nu_n}(Z_n(z))
		\end{align*}
		hold. The so-called subordination functions $Z_1, \dots, Z_n$ satisfy $\Im Z_i(z) \geq \Im z$ for all $z \in \mathbb{C}^+, i \in \{1, \dots, n\}.$  Moreover, there exists a probability measure $\nu$ such that $F_\nu(z) = F_{\nu_1}(Z_1(z))$ holds for all $z \in \mathbb{C}^+$. We define $\nu_1 \boxplus \cdots \boxplus \nu_n := \nu$.
	\end{theorem}
The proof of Theorem \ref{subordination functions} shows that the subordination functions $Z_1, \dots, Z_n$ are \textit{Nevanlinna} functions that satisfy $\nicefrac{Z_i(z)}{z} \rightarrow 1$ as $z \rightarrow \infty$ non-tangentially to $\mathbb{R}$. With that knowledge, the following lemma can be derived from the theory of Nevanlinna functions; see \cite[Chapter 3]{Chistyakov2011} and the references mentioned therein.
	\begin{lem} \label{reciprocal subordination function as Cauchy transform}
		Let $Z$ be a subordination function with respect to some free additive convolution. Then, there exists a probability measure $\sigma$ on $\mathbb{R}$ such that $G_{\sigma}(z) = \frac{1}{Z(z)}$ holds for all $z \in \mathbb{C}^+.$
	\end{lem}

	For the last results of this section, we need a few definitions: For two probability measures $\nu_1$ and $\nu_2$ on $\mathbb{R}$, the Kolmogorov distance $\Delta$ is defined by 
	\begin{align*}
		\Delta(\nu_1, \nu_2) := \sup_{x \in \mathbb{R}} \big \vert \nu_1((-\infty,x]) - \nu_2((-\infty,x]) \big\vert, 
	\end{align*}			
	whereas the L\'{e}vy distance $d_L$ is given by 
	\begin{align*}
		d_L(\nu_1, \nu_2) := \inf \left\{ s > 0: \nu_1((-\infty, x-s]) - s \leq \nu_2((-\infty, x]) \leq \nu_1((-\infty, x+s])+s \,\, \forall x \in \mathbb{R} \right\}.
	\end{align*}
	It is well-known that $d_L(\nu_1, \nu_2) \leq \Delta(\nu_1, \nu_2)$ holds for $\nu_1$ and $\nu_2$ as above.
	The dilation $D_c \nu$ of $\nu$ by the factor $c>0$ is given by $D_c \nu (A) = \nu(c^{-1}A)$ for any measurable set $A \subset \mathbb{R}$, where $c^{-1}A$ is defined by $c^{-1}A := \{ x \in \mathbb{R}:  cx \in A \}$. Clearly, we have $\Delta(D_c\nu_1, D_c\nu_2) = \Delta(\nu_1, \nu_2)$ for any $c>0$. The following lemma analyzes the L\'{e}vy distance of two different dilations of a compactly supported probability measure. 
	\begin{lem} \label{Diliatation Levy}
		Let $\nu$ be a probability measure with compact support in $[-L, L]$ for some $L>0$ and let $\varepsilon_1, \varepsilon_2 >0$. Then, we have
		\begin{align*}
			d_L\left(D_{\varepsilon_1}\nu, D_{\varepsilon_2}\nu\right) \leq L\vert \varepsilon_1 - \varepsilon_2 \vert. 
		\end{align*}
		
	\end{lem}
	Let $\omega_{c}$ denote Wigner's semicircle law with mean zero and variance $c$. We have $\omega_c = D_{\sqrt{c}} \omega$ as well as $\omega_{\nicefrac{1}{2}} \boxplus \omega_{\nicefrac{1}{2}} = \omega$. The following theorem is a special case of a result proven by Bao, Erd\H{o}s, and Schnelli \cite[Theorem 2.7]{Bao2016}. The proof of the theorem heavily relies on the uniqueness of subordination functions and the Newton-Kantorovich theorem. 
	
	\begin{theorem} \label{Thm 2.7 Bao, Erdos, Schnelli}
		Let $\mathcal{I} \subset (-2,2)$ be a compact non-empty interval and fix $\eta \in (0, \infty)$. Define 
		\begin{align*}
			S_{\mathcal{I}}(0, \eta) := \{ x +iy \in \mathbb{C}^+: x \in \mathcal{I}, y \in [0, \eta]\}. 
		\end{align*}
		Then, there exist constants $b=b(\omega_{\nicefrac{1}{2}}, \mathcal{I}, \eta)>0$ and $Z=Z(\omega_{\nicefrac{1}{2}}, \mathcal{I}, \eta)< \infty$ such that whenever two probability measures $\nu_1$ and $\nu_2$ on $\mathbb{R}$ satisfy the condition 
		\begin{align*}
			d_L(\omega_{\nicefrac{1}{2}}, \nu_1) + d_L(\omega_{\nicefrac{1}{2}}, \nu_2) \leq b, 
		\end{align*}
		we have 
		\begin{align*}
			\max_{z \in S_{\mathcal{I}}(0, \eta) } \left \vert G_{\omega}(z) - G_{\nu_1 \boxplus \nu_2}(z)\right \vert  \leq Z \left(	d_L(\omega_{\nicefrac{1}{2}}, \nu_1) + d_L(\omega_{\nicefrac{1}{2}}, \nu_2) \right). 
		\end{align*} 
	\end{theorem}
	
	We end this section with a well-known smoothing inequality provided by Bai \cite{Bai1993}. In its full generality, it establishes an upper bound on the Kolmogorov distance of two probability measures on $\mathbb{R}$ in terms of their Cauchy transforms. Since we are interested in the Kolmogorov distance between free additive convolutions and Wigner's semicircle law only, we use the following version of Bai's inequality proved in \cite[Corollary 2.3]{Goetze2003}.
	\begin{prop} \label{Bai - Götze's Version}
		Let $\nu$ be a probability measure on $\mathbb{R}$ with Cauchy transform $G_\nu.$ Let $G_\omega$ denote the corresponding analog of Wigner's semicircle law and assume that 
		\begin{align} \label{assumption Bai finite integral}
			\int_{-\infty}^\infty \left\vert \nu((-\infty, x]) - \omega((-\infty, x]) \right\vert dx < \infty 
		\end{align}
		holds. Choose $a \in (0,1)$ and $\varepsilon, \tau, \gamma >0$ in such a way that 
		\begin{align*}
		\rev	\gamma \revv:= \frac{1}{\pi} \int_{\vert x \vert < \tau} \frac{1}{x^2+1} dx \rev> \frac{1}{2} \revv \qquad \text{and} \qquad \varepsilon > 2a\tau
		\end{align*}
are satisfied. Define $I_\varepsilon :=  [-2 + \nicefrac{\varepsilon}{2},2 - \nicefrac{\varepsilon}{2}]$. Then, we have
		\begin{align*}
		\!	\Delta(\nu, \omega) \leq \rev C_{\gamma} \revv \left( \int_{-\infty}^{\infty} \!\!\!\!\! \left\vert G_\nu(u+i) - G_\omega(u+i) \right\vert du   +  \sup_{u \in I_\varepsilon} \int_{a}^1 \!\! \vert G_\nu(u+iv) - G_\omega(u+iv) \vert dv + \rev \frac{4\tau^2a}{\pi} + \gamma \varepsilon^{\nicefrac{3}{2}} \revv \right),
		\end{align*}
		\rev where $C_\gamma>0$ is given by $C_\gamma := (\pi (2\gamma -1))^{-1}$.
	\end{prop}

	\subsection{Concentration inequalities on the sphere}  \label{section: concentration inequalities}

	In this section we formulate some concentration inequalities on the sphere which will help to define the set $\mathcal{F} \subset \mathbb{S}^{n-1}$ in our results. 
	
	\begin{lem} \label{concentration inequalities F_3}
		Let $n \geq 4$ and $A \geq 4.$ We have
		\begin{align*}
			\sigma_{n-1}\left(  \max_{i \in \{1, \dots, n\}}  \vert \theta_i \vert > A\sqrt{ \frac{ \log n }{n}} \right) \leq \frac{8}{A\sqrt{2\pi}} \frac{1}{n}.
		\end{align*}
		\begin{proof} 
			The proof heavily relies on identities shown in  \cite[Section 2]{Bobkov2020a}. The distribution of the normalized first component $\sqrt{n} \theta_1$ with respect to $\sigma_{n-1}$ has a symmetric density given by 
			\begin{align*}
				f_n(x) := c_n'\left(1-\frac{x^2}{n}\right)_+^{\frac{n-3}{2}}, \qquad c_n' := \frac{\Gamma\left( \frac{n}{2}\right)}{\sqrt{\pi n} \Gamma \left( \frac{n-1}{2} \right)}.
			\end{align*}
			Note that we have $c_n' <(2\pi)^{-\nicefrac{1}{2}}$ for all $n \geq 2.$ Together with the inequality $\log (1-x)<x$, $x \in (0,1)$, it is easy to see that 
			\begin{align*}
				f_n(x) \leq c_n'e^{-\nicefrac{x^2}{8}}
			\end{align*}
			holds true for all $x \in \mathbb{R}$ and $n \geq 4.$ Combining the last observations with the well-known tail inequality for the standard normal distribution, we arrive at
			\begin{align*}
				\sigma_{n-1}\left( \vert \theta_1 \vert > A \sqrt{\frac{\log n }{n}} \right) = 2\sigma_{n-1}\left( \sqrt{n}\theta_1  > A \sqrt{\log n} \right)    \leq 2\int_{A \sqrt{\log n}}^\infty \frac{e^{-\nicefrac{x^2}{8}}}{\sqrt{2 \pi}}  dx \leq  \frac{8}{A\sqrt{2\pi}} n^{-\frac{A^2}{8}}.
			\end{align*}
			Since the components $\theta_1, \dots, \theta_n$ viewed as random variables are equally distributed with respect to $\sigma_{n-1}$, we immediately obtain the claimed inequality by
			\begin{align*}
				\sigma_{n-1}\left(  \max_{i \in \{1, \dots, n\}}  \vert \theta_i \vert >  A \sqrt{\frac{\log n }{n}}\right) \leq n \sigma_{n-1}\left( \vert \theta_1 \vert  > A \sqrt{\frac{\log n }{n}} \right) \leq  \frac{8}{A\sqrt{2\pi}} n^{-\frac{A^2}{8}+1} \leq  \frac{8}{A\sqrt{2\pi}} n^{-1}.
			\end{align*}
		\end{proof}
	\end{lem}
	
	The following lemma was proven in \cite[Lemma 5.3.3]{Bobkov2020b}.
	\begin{lem} \label{5.3.3 Bobkov}
		For all real $r \geq 1$ and all $k \in \mathbb{N}, k>2$, we have 
		\begin{align*}
			\sigma_{n-1}\left( \sum_{i=1}^n \vert \theta_i \vert^k \geq \frac{B_kr}{n^{\nicefrac{(k-2)}{2}}} \right) \leq \exp\left( -(rn)^{\nicefrac{2}{k}}\right)
		\end{align*}
		for $B_3 = 33, B_4 = 121$, and  $B_k = (\sqrt{k} +2)^{k}$ for $k>4.$
	\end{lem}
	
\revtwo	The next lemma will be helpful in the proofs of the improved rates in \cref{main Kolmogorov epsilon unbounded,main Delta tilde}. \revvtwo
	\begin{lem} \label{concentration for I_5}
		For any $t > 0$ we have
		\begin{align*}
			\sigma_{n-1}\left( \left \vert \sum_{i=1}^n \theta_i^3 \right \vert  \geq \frac{t}{n}\right) \leq 2\exp\left( -\frac{1}{23}t^{\nicefrac{2}{3}}\right).
		\end{align*}
		Moreover, we have 
		\begin{align*}
			\sigma_{n-1}\left( \left \vert \sum_{i=1}^{n} \theta_i^3 \right \vert \geq \frac{10}{\sqrt{n} \log n} \right) < \frac{2}{\sqrt{n}}
		\end{align*}
		for any $n \geq 2.$
		\begin{proof}
			The first claim is proven in \cite[Lemma 5.3.2]{Bobkov2020b}. The second claim follows by easy modifications of that proof. We briefly sketch the idea: Define $f: \mathbb{S}^{n-1} \rightarrow \mathbb{R}$ by $f(\theta_1, \dots, \theta_n) := n \sum_{i=1}^n \theta_i^3$. By symmetry, we have  $\mathbb{E}_\theta(f) = 0$. The logarithmic Sobolev inequality on the unit sphere yields
			\begin{align*}
				\int_{\mathbb{S}^{n-1}} \vert f\vert^2 d\sigma_{n-1}  \leq  \frac{1}{n-1} \sum_{i=1}^{n}	\int_{\mathbb{S}^{n-1}} \left \vert \frac{\partial  f}{\partial \theta_i}  \right\vert^2 d\sigma_{n-1}  \leq 18n \sum_{i=1}^n  \int_{\mathbb{S}^{n-1}} \theta_i^4 d\sigma_{n-1} < 54;
			\end{align*}
			compare to \cite[equation (5.3.2)]{Bobkov2020b}. Note that the last equation is true due to 
			\begin{align*}
			\mathbb{E}_\theta \theta_1^4 < \frac{4\Gamma\left(\frac{5}{2}\right)}{\sqrt{\pi}n^2}  = \frac{3}{n^2}.
			\end{align*}
			Now, the Markov inequality implies 
			\begin{align*}
				\sigma_{n-1}\left( \left \vert \sum_{i=1}^{n} \theta_i^3 \right \vert \geq \frac{10}{\sqrt{n} \log n} \right) = 	\sigma_{n-1}\left( \left \vert n \sum_{i=1}^{n} \theta_i^3 \right \vert \geq \frac{10\sqrt{n}}{\log n} \right) < \frac{2}{\sqrt{n}}.
			\end{align*}
		\end{proof}
	\end{lem}

	\section{\texorpdfstring{Rate of convergence with respect to $\Delta$: Proofs of Theorem \ref{main theorem} and Corollary \ref{Corollary unbounded}}{Proofs of Theorem 1.1 and Corollary 1.2}} \label{Section: unbounded}
	
In this section we prove Theorem \ref{main theorem} and Corollary \ref{Corollary unbounded} \revtwo establishing rates of convergence of a weighted sum to Wigner's semicircle law measured with respect to the (expectation of the) Kolmogorov distance. \revvtwo Both proofs are based on an idea introduced by Chistyakov and Götze in  \cite[Theorem 2.6]{Chistyakov2008a} and \cite[Theorem 2.1]{Chistyakov2013}.  We begin with Theorem \ref{main theorem}.

\subsection{\texorpdfstring{\rev}{} Proof of Theorem \texorpdfstring{\ref{main theorem}}{1.1}}
For better comprehensibility, the proof is divided into nine parts, each of which is handled in a separate section for later reference. We start with the construction of the set $\mathcal{F}$ in \cref{part 1 - construction of F}. After that, in \cref{part 2 - 3rd order functional eq}, we derive a cubic functional equation for one of the subordination functions of the convolution $\mu_\theta$. Solving that equation, we obtain a concrete formula for the subordination function in \cref{part 3 - roots 3rd oder}. \Cref{part 4 - 2nd order functional eq,part 5 - roots of 2nd} are concerned with establishing and solving a quadratic functional equation for the same subordination function. In \cref{part 6 - application Bai} we explain how Proposition \ref{Bai - Götze's Version} can be applied to our setting. We see that it remains to bound two integrals in order to finish the proof. The results obtained from the cubic functional equation can be used to bound one of these integrals in \cref{part 7 - bounding integral wrt im }, whereas the other integral can be handled with the help of the quadratic functional equation in \cref{part 8 - bounding integrals wrt re}. We end the proof of \cref{main theorem} in \cref{part 9 - final conclusion}.
		
		\subsubsection{Construction of \texorpdfstring{$\mathcal{F}$}{F}}  \label{part 1 - construction of F}

		Fix $\rho \in (0,1)$ and let $n \geq 4.$ To ease the notation, we write $[n] := \{1, \dots, n\}$.
		Define 
		\begin{align*}
			\mathcal{F}_1 := \left\{ \theta \in \mathbb{S}^{n-1}: \max_{i \in [n]} \vert \theta_i \vert \leq A_\rho\sqrt{\frac{\log n}{n}} \right\}  
		\end{align*}
		for $A_\rho := \frac{12}{\sqrt{2\pi} \rho} \geq 4$ and note that $\sigma_{n-1}\left( \mathcal{F}_1^{\mathsf{c}}\right)  \leq \nicefrac{\rho}{2}$ holds according to Lemma \ref{concentration inequalities F_3}. Moreover, let 
		\begin{align*}
			\mathcal{F}_2 := \left\{ \theta \in \mathbb{S}^{n-1}: \sum_{i=1}^n \vert \theta_i \vert^{7} \leq \frac{B_{7, \rho}}{n^{\nicefrac{5}{2}}} \right\}
		\end{align*}
		for $B_{7, \rho} := B_7c_{7}\left( \log \nicefrac{2}{\rho} \right)^{\nicefrac{7}{2}} $, where $B_7$ is as in Lemma \ref{5.3.3 Bobkov} and $c_{7} := (\log 2 )^{-\nicefrac{7}{2}} >1$. Lemma \ref{5.3.3 Bobkov} yields 
		$\sigma_{n-1}\left( \mathcal{F}_2^{\mathsf{c}}\right) \leq \nicefrac{\rho}{2}$. Using Hölder's inequality, we obtain  
		\begin{align*}
			\mathcal{F}_2 \subset \left\{ \theta \in \mathbb{S}^{n-1}: \sum_{i=1}^n \vert \theta_i \vert^{k} \leq \frac{B_{7, \rho}}{n^{\frac{k-2}{2}}}  \right\}
		\end{align*}
		for all $k \in \mathbb{N}$ with $2 < k  \leq 7$. 
		Finally, let 
		\begin{align} \label{definition F}
			\mathcal{F} := \mathcal{F}_1 \cap \mathcal{F}_2.
		\end{align}
		It is clear that $\sigma_{n-1}\left( \mathcal{F}^{\mathsf{c}} \right) \leq \rho$ holds. 

		\subsubsection{Cubic functional equation for \texorpdfstring{\textbf{$Z_1$}}{Z1}} \label{part 2 - 3rd order functional eq}  
		Let us start with introducing some notation which we will use repeatedly in the course of this paper. Define $\mu_i := D_{\theta_i} \mu $, $i \in [n]$, and let $\mu_\theta = \mu_1 \boxplus \dots \boxplus \mu_n$. It is clear that $\mu_\theta$ is the analytic distribution of the weighted sum $S_\theta$ defined in Theorem \ref{main theorem}. Let $G_i$ denote the Cauchy transform of $\mu_i$ and let $F_i$ be the corresponding reciprocal Cauchy transform. The Cauchy transform of $\mu_\theta$ will be denoted by $G_\theta.$ Lastly, let the subordination functions belonging to $\mu_\theta$ be given by $Z_1, \dots, Z_n$. 
		
		For the rest of this proof, fix some arbitrarily chosen $\theta \in \mathcal{F}$. Without loss of generality, assume that $\theta_1^2 = \min_{i \in [n]} \theta_i^2$ holds. The aim of this section is to establish a cubic functional equation for $Z_1$.
		
		According to Theorem \ref{subordination functions}, we have
		\begin{align*}
			Z_1(z) - z = \sum_{i=2}^n F_i(Z_i(z)) - Z_i(z), \qquad z \in \mathbb{C}^+,
		\end{align*}
		from which we obtain 
		\begin{align} \label{$Z_1(z) - z + (Z_1(z))^{-1}$}
			Z_1(z) - z + \frac{1}{Z_1(z)} = Z_1(z) - z + \frac{\sum_{i=1}^n \theta_i^2}{Z_1(z)}  = \sum_{k=1}^5 J_k(z),
		\end{align}
		where 
		\begin{align*}
			J_1(z) := \sum_{i=2}^n F_i(Z_i(z)) - Z_i(z) + \frac{\theta_i^2}{Z_i(z)}  + \frac{\theta_i^3m_3(\mu)}{Z_i^2(z)}, \qquad	\quad J_2(z) :=  \sum_{i=2}^n \frac{\theta_i^2}{Z_1(z)} - \frac{\theta_i^2}{Z_i(z)}, \\
			J_3(z) :=   \frac{\theta_1^2}{Z_1(z)}, \qquad \quad 
			J_4(z) :=  m_3(\mu)\sum_{i=2}^n \frac{\theta_i^3}{Z_1^2(z)} - \frac{\theta_i^3}{Z_i^2(z)}, \qquad \quad 
			J_5(z) := -\frac{m_3(\mu)\sum_{i=2}^n \theta_i^3}{Z_1^2(z)}.
		\end{align*}
		With $I_3 := \theta_1^2$, $I_5 := - m_3(\mu)\sum_{i=2}^n \theta_i^3 $, and $I_k(z) := Z_1^2(z)J_k(z)$ for $k \in \{1,2,4\}$, we can write
		\begin{align} \label{functional eq Z_1 }
			Z_1(z) - z + \frac{1}{Z_1(z)} = \frac{I_1(z) + I_2(z) + I_4(z) + I_5}{Z_1^2(z)} + \frac{I_3}{Z_1(z)}, \qquad  z \in \mathbb{C}^+.
		\end{align}
		Define $r(z):= I_1(z) + I_2(z) + I_4(z) + I_5$ as well as 
		\begin{align} \label{def P}
			P(z, \omega) := \omega^3 - z\omega^2 +(1 - I_3)\omega - r(z)
		\end{align}
	for $z$ as above. 
		Then, the equation in \eqref{functional eq Z_1 } is equivalent to $P(z, Z_1(z)) = 0$ for all $z \in \mathbb{C}^+$. 
		
		Later, we need to find the roots of $P$. In order to be able to do that, we have to bound the coefficients $I_3$ and $r(z)$. We start with estimating $I_1(z)$ appearing in $r(z)$. A simple calculation shows that
		\begin{align} \label{subordination times cauchy expansion}
			\begin{split}
			Z_i(z)G_i(Z_i(z)) & = \rev 1 + \int_{\mathbb{R}} \frac{u}{Z_i(z) - u} \mu_i(du) \\ \revv  &= 1 + \frac{1}{Z_i(z)} \int_{\mathbb{R}} \frac{u^2}{Z_i(z) - u} \mu_i(du) \\
			& =  1 + \frac{\theta_i^2}{Z_i^2(z)} + \frac{1}{Z_i^2(z)} \int_{\mathbb{R}} \frac{u^3}{Z_i(z) - u} \mu_i(du)   \\
			&  = 1 +  \frac{\theta_i^2}{Z_i^2(z)} + \frac{\theta_i^3m_3(\mu)}{Z_i^3(z)} +  \frac{1}{Z_i^3(z)}  \int_{\mathbb{R}} \frac{u^4}{Z_i(z) - u}\mu_i(du) 
			\end{split}
		\end{align}
		holds for all $i \in [n]$, $z \in \mathbb{C}^+.$ For later reference and under the additional assumption that $\mu$ has finite absolute moments up to the sixth order, we can continue the expansion above by 
		\begin{align} \label{subordination times cauchy expansion more moments}
			\begin{split}
				Z_i(z)G_i(Z_i(z)) 
				&=1 +  \frac{\theta_i^2}{Z_i^2(z)} + \frac{\theta_i^3m_3(\mu)}{Z_i^3(z)} +  \frac{\theta_i^4m_4(\mu)}{Z_i^4(z)} + \frac{1}{Z_i^4(z)}\int_{\mathbb{R}} \frac{u^5}{Z_i(z) - u}\mu_i(du) \\
				& = 1 +  \frac{\theta_i^2}{Z_i^2(z)} + \frac{\theta_i^3m_3(\mu)}{Z_i^3(z)} +  \frac{\theta_i^4m_4(\mu)}{Z_i^4(z)} + \frac{\theta_i^5m_5(\mu)}{Z_i^5(z)} + \frac{1}{Z_i^5(z)}\int_{\mathbb{R}} \frac{u^6}{Z_i(z) - u}\mu_i(du).
			\end{split}
		\end{align}
		Note that the integrals appearing in \eqref{subordination times cauchy expansion} and \eqref{subordination times cauchy expansion more moments} can be bounded by 
		\begin{align*}
			\left \vert \int_{\mathbb{R}} \frac{u^k}{Z_i(z) - u} \mu_i(du) \right \vert \leq \frac{\beta_k(\mu_i)}{\Im Z_i(z)} \leq \frac{ \vert \theta_i \vert ^k \beta_k(\mu)}{\Im z}, \qquad k \in [6].
		\end{align*}
		Now, let us return to our original assumption $m_4(\mu) < \infty$, i.e.\@ we forget about \eqref{subordination times cauchy expansion more moments}. We can write
		\begin{align} \label{relation F_i - Id and $r_{n,i}$}
			F_i(Z_i(z)) - Z_i(z) = \frac{1}{G_i(Z_i(z))} - Z_i(z) = \frac{1 - G_i(Z_i(z))Z_i(z)}{G_i(Z_i(z))Z_i(z)}Z_i(z) = -\frac{r_{n,i}(z)}{1 + r_{n,i}(z)}Z_i(z)
		\end{align}
		with $r_{n,i}(z) := G_i(Z_i(z))Z_i(z) -1$, $z \in \mathbb{C}^+.$
		\rev We continue by bounding $\vert r_{n,i}(z) \vert$ for certain $z$. Making use of \eqref{subordination times cauchy expansion} and Cauchy's inequality, we obtain 
		\begin{align*}
			\vert r_{n,i}(z) \vert \leq \int_{\mathbb{R}} \left \vert \frac{u}{Z_i(z)- u} \right \vert \mu_i(du)  \leq \left( \int_{\mathbb{R}} u^2 \mu_i(du) \right)^{\nicefrac{1}{2}}\left( \int_{\mathbb{R}} \frac{1}{\vert Z_i(z) - u \vert^2} \mu_i(du) \right)^{\nicefrac{1}{2}}
		\end{align*}
		for all $z \in \mathbb{C}^+$ and $i \in [n]$. Observe that
		\begin{align*}
		\Im \left( \frac{1}{Z_i(z) - u} \right) = \frac{- \Im Z_i(z)}{\vert Z_i(z) - u \vert^2} <0
		\end{align*}
		holds for any $u \in \mathbb{R}$ and $z, i$ as above. Together with \cref{subordination functions}, it follows 
		\begin{align} \label{bound $r_n,i$ mit Cauchy und integration by parts}
				\vert r_{n,i}(z) \vert \leq  \frac{\vert \theta_i \vert}{\sqrt{\Im z}} \left( \int_{\mathbb{R}} -\Im \left( \frac{1}{Z_i(z) - u} \right) \mu_i(du) \right)^{\nicefrac{1}{2}} \leq \frac{\vert \theta_i \vert}{\sqrt{\Im z}} \vert G_{i}(Z_i(z)) \vert^{\nicefrac{1}{2}} = \frac{\vert \theta_i \vert}{\sqrt{\Im z}} \vert G_{\theta}(z) \vert^{\nicefrac{1}{2}}
		\end{align}
		for $z \in \mathbb{C}^+$, $i \in [n]$. By integration by parts, \cite[Theorem 2.6]{Chistyakov2008a}, and $\theta \in \mathcal{F}_2$, we get
		\begin{align*}
			\vert G_\theta(z) - G_\omega(z) \vert \leq \frac{\pi \Delta(\mu_\theta, \omega)}{\Im z} \leq \frac{\pi c_0 \left(\beta_3(\mu)\sum_{i=1}^n \vert \theta_i \vert^3 \right)^{\nicefrac{1}{2}}}{\Im z} \leq \frac{\pi c_0 (\beta_3(\mu)B_{7, \rho})^{\nicefrac{1}{2}}}{n^{\nicefrac{1}{4}}} \cdot \frac{1}{\Im z}
		\end{align*}
	for all $z \in \mathbb{C}^+$ with $c_0 > 0$ being the numerical constant from \cite[Theorem 2.6]{Chistyakov2008a}. Combining the last two estimates with $\theta \in \mathcal{F}_1$ and $\vert G_\omega \vert \leq 1$ holding in $\mathbb{C^+}$, see \cite[Lemma 8]{Kargin2007a}, we arrive at 
	\begin{align*}
		\vert r_{n,i}(z) \vert & \leq  A_\rho\sqrt{\frac{\log n}{n}}\frac{1}{\sqrt{\Im z}}\left( 1 + \frac{\pi c_0 (\beta_3(\mu)B_{7, \rho})^{\nicefrac{1}{2}}}{n^{\nicefrac{1}{4}}} \cdot \frac{1}{\Im z}  \right)^{\nicefrac{1}{2}} \\ &  \leq A_\rho\sqrt{\frac{\log n}{n}}\frac{1}{\sqrt{\Im z}} + A_\rho (\pi c_0)^{\nicefrac{1}{2}}(\beta_3(\mu)B_{7, \rho})^{\nicefrac{1}{4}} \frac{\sqrt{\log n}}{n^{\nicefrac{5}{8}}} \frac{1}{\Im z } \\
	& 	\leq A_\rho \frac{\sqrt{\log n}}{n^{\nicefrac{1}{4}}} + A_\rho (\pi c_0)^{\nicefrac{1}{2}}(\beta_3(\mu)B_{7, \rho})^{\nicefrac{1}{4}} \frac{\sqrt{\log n}}{n^{\nicefrac{1}{8}}}
	\end{align*}
		for all $z \in \mathbb{C}^+$ with $\Im z \geq n^{-\nicefrac{1}{2}}$ and $i \in [n]$. Hence, there exists a threshold $n_1 \geq 4$ -- depending on $\mu$ and $\rho$ -- such that the inequality 
		\begin{align}  \label{estimate r_{n,i}}
			\vert r_{n,i}(z) \vert < \frac{1}{10}
		\end{align}
		is valid for all choices of $z \in \mathbb{C}^+$ satisfying $\Im z \geq n^{-\nicefrac{1}{2}}$ and all $i \in [n], n \geq n_1.$
		From now on, assume that $n \geq n_1$ holds. Define \begin{align} \label{D_1}
			D_1 :=\left \{ z \in \mathbb{C}^+ : \Im z \geq  \rev \frac{1}{\sqrt{n}} \right\}
		\end{align}
	and observe 
		\begin{align*}
			\vert 1 + r_{n,i}(z) \vert > \frac{9}{10}, \qquad z \in D_1, i \in [n].
		\end{align*} \revv
Together with \eqref{relation F_i - Id and $r_{n,i}$}, we get 
		\begin{align} \label{first estimate I_1}
			\!\!\!\!\! \left \vert  F_i(Z_i(z)) \!- \! Z_i(z) \!  + \! \frac{\theta_i^2}{Z_i(z)} \!+\! \frac{\theta_i^3m_3(\mu)}{Z_i^2(z)}  \right \vert \leq  \frac{10}{9}
			\left \vert r_{n,i}(z)Z_i(z) \!-\! \frac{\theta_i^2(1 + r_{n,i}(z))}{Z_i(z)} \!-\! \frac{\theta_i^3 m_3(\mu)(1 + r_{n,i}(z))}{Z_i^2(z)} \right \vert
		\end{align}
		for all $z \in D_1, i \in [n].$
		Using the definition of $r_{n,i}(z)$ and the identities in \eqref{subordination times cauchy expansion}, we calculate
		\begin{align} \label{r_{n,i}(z)Z_i(z)  - ...}
			r_{n,i}(z)Z_i(z) - \frac{\theta_i^2(1 + r_{n,i}(z))}{Z_i(z)} - \frac{\theta_i^3 m_3(\mu)(1 + r_{n,i}(z))}{Z_i^2(z)}  = \frac{s_i(z)}{Z_i^2(z)},  \qquad z \in \mathbb{C}^+,
		\end{align}
		where $s_i(z)$ is defined by
		\begin{align} \label{def s_i(z)}
			s_i(z) :=  \left(1-\frac{\theta_i^2}{Z_i^2(z)}-\frac{\theta_i^3m_3(\mu)}{Z_i^3(z)}\right)\int_{\mathbb{R}} \frac{u^4}{Z_i(z) - u} \mu_i(du) -\frac{\theta_i^4}{Z_i(z)} - \frac{2\theta_i^5m_3(\mu)}{Z_i^2(z)} - \frac{\theta_i^6m_3(\mu)^2}{Z_i^3(z)}.
		\end{align}
		With the help of \cref{subordination functions}, we obtain\revv 
		\begin{align*}
			\frac{Z_1(z)}{Z_i(z)} = \frac{Z_1(z)G_{1}(Z_1(z))}{Z_i(z)G_i(Z_i(z))} = \frac{1 + r_{n,1}(z)}{1 + r_{n,i}(z)}, \qquad  z \in \mathbb{C}^+,i \in [n].
		\end{align*}
		The inequality in \eqref{estimate r_{n,i}} yields 
		\begin{align} \label{estimate $Z_1/Z_i -1$ by constant}
			\left \vert \frac{Z_1(z)}{Z_i(z)} - 1 \right \vert  \leq \frac{10}{9}\left( \vert r_{n,1}(z) \vert  + \vert r_{n,i}(z) \vert \right)
			< \frac{2}{9}
		\end{align}
		for all $z \in D_1$.  Combining the last observation with the assumption $\theta_1^2 = \min_{i \in [n]} \theta_i^2$, we derive
		\begin{align} \label{estimate $Z_1/Z_i -1$}
			\left \vert \frac{Z_1(z)}{Z_i(z)} - 1 \right \vert	\leq \frac{10}{9} \frac{1}{\Im z}\left( \frac{\theta_1^2}{\vert Z_1(z) \vert} + \frac{\theta_i^2}{\vert Z_i(z) \vert}\right) \leq \frac{10}{9} \frac{\theta_i^2}{\Im z  \vert Z_1(z) \vert} \left( 1 + \left \vert \frac{Z_1(z) }{ Z_i (z)} \right\vert \right) \leq 3 \frac{\theta_i^2}{\Im z \vert Z_1(z) \vert}
		\end{align}
for $z \in D_1$ and $i \in [n]$.
		Similarly, we deduce
		\begin{align*}
			\left \vert \frac{Z_1^2(z)}{Z_i^2(z)} - 1 \right \vert  &\leq \frac{100}{81}\left( 2\vert r_{n,1}(z) \vert + 2 \vert r_{n,i}(z) \vert + \left\vert r^2_{n,1}(z) \right\vert + \left\vert r_{n,i}^2(z) \right\vert \right)  
			< \frac{3}{5}
		\end{align*}
		and arrive at
		\begin{align} \label{estimate $Z_1^2/Z_i^2 -1$}
			\begin{split}
				\left \vert \frac{Z_1^2(z)}{Z_i^2(z)} - 1 \right \vert & \leq  \frac{100}{81} \left(  \frac{2\theta_1^2}{\Im z \vert Z_1(z) \vert} + \frac{2\theta_i^2}{\Im z \vert Z_i(z) \vert } + \frac{\theta_1^4}{ (\Im z)^2 \vert Z_1^2(z) \vert} + \frac{\theta_i^4}{ (\Im z)^2\vert Z_i^2(z) \vert}  \right) \\ &\leq
				\frac{100}{81} \frac{2\theta_i^2}{\Im z \vert Z_1(z) \vert} \left( \left \vert \frac{Z_1(z)}{ Z_i(z)} \right\vert  + 1\right) +  \frac{100}{81} \frac{\theta_i^4}{(\Im z)^2 \vert Z_1^2(z) \vert} \left( \left \vert \frac{Z_1^2(z)}{ Z_i^2(z)} \right \vert+ 1 \right)   \\
				&\leq  6\left( \frac{\theta_i^2}{\Im z \vert Z_1(z) \vert} +   \frac{\theta_i^4}{(\Im z)^2 \vert Z_1^2(z) \vert}\right) 
			\end{split}
		\end{align}
		for all $z \in D_1, i \in [n].$
Together with \eqref{first estimate I_1} and \eqref{r_{n,i}(z)Z_i(z)  - ...}, it follows 
		\begin{align} \label{second estimate I_1 unbounded main}
			\left \vert Z_1^2(z) \left( F_i(Z_i(z)) - Z_i(z)  + \frac{\theta_i^2}{Z_i(z)} + \frac{\theta_i^3m_3(\mu)}{Z_i^2(z)}  \right)\right \vert \leq \frac{10}{9} \left\vert \frac{Z_1^2(z)}{Z_i^2(z)}s_i(z) \right \vert  \leq \frac{20}{9} \vert s_i(z) \vert
		\end{align}
	for $z$ and $i$ as before. 
	From \eqref{def s_i(z)}, we directly obtain 
		\begin{align*}
			\vert s_i(z) \vert \leq  \frac{\theta_i^4m_4(\mu) + \theta_i^4}{\Im z}  + \frac{2\vert \theta_i \vert^5 \beta_3(\mu)}{(\Im z)^2} + \frac{\theta_i^6 m_4(\mu)+ \theta_i^6 \beta_3(\mu)^2}{(\Im z)^3} + \frac{\vert\theta_i\vert^{7}\beta_3(\mu)m_4(\mu)}{(\Im z)^4}
		\end{align*}
		for any $z \in \mathbb{C}^+.$
		Now, note that each summand of the upper bound for $\sum_{i=2}^n \vert s_i(z)\vert$ is up to constants of the form $(\Im z)^{-l}\sum_{i=2}^n \vert \theta_i \vert^k$ 
		with $k,l \in \mathbb{N}, k - l = 3$, $1 \leq l \leq 4 \leq k \leq 7$. Together with the fact that $\theta \in \mathcal{F}_2$ holds, we can estimate
	\rev	\begin{align} \label{estimate difference more than 3}
			\frac{\sum_{i=2}^n \vert \theta_i \vert^k}{(\Im z)^l} \leq \frac{B_{7,\rho}}{n^{\frac{k-2}{2}}} n^{\frac{l}{2}} = \frac{B_{7,\rho}}{\sqrt{n}} 
		\end{align}
	\revv	for all $z \in D_1$ and $k,l$ as above. Due to  $m_2(\mu) = 1$, we have $\beta_3(\mu), m_4(\mu) \geq 1$. By \eqref{second estimate I_1 unbounded main} and \eqref{estimate difference more than 3}, we conclude 
		\begin{align} \label{estimate I_1}
			\begin{split}
				\vert I_1(z) \vert &\leq \sum_{i=2}^n \left \vert Z_1^2(z) \left( F_i(Z_i(z)) - Z_i(z)  + \frac{\theta_i^2}{Z_i(z)} + \frac{\theta_i^3m_3(\mu)}{Z_i^2(z)}  \right)\right \vert \leq \frac{20}{9} \sum_{i=2}^n\vert s_i(z) \vert  \\ & \leq \rev  \frac{16B_{7, \rho}\beta_3(\mu)^2m_4(\mu)}{\sqrt{n}}  \revv
			\end{split}
		\end{align}
	for all $z \in D_1$. Now, let us bound $I_2(z).$ Making use of \eqref{estimate $Z_1/Z_i -1$}, we immediately obtain
		\begin{align*}
			\vert J_2 (z)\vert \leq  \sum_{i=2}^n \left \vert  \frac{\theta_i^2}{Z_1(z)} \right \vert  \left \vert \frac{Z_1(z)}{Z_i(z)} -1\right \vert \leq 3 \sum_{i=2}^n \frac{\theta_i^4}{\vert Z_1^2(z) \vert \Im z}, \qquad z \in D_1.
		\end{align*}
		Applying \eqref{estimate difference more than 3} once again, we arrive at
		\begin{align*}
			\vert I_2(z) \vert = \big \vert Z_1^2(z) J_2(z) \big \vert \leq \frac{3}{\Im z} \sum_{i=1}^n \theta_i^4 \leq \rev \frac{3B_{7,\rho}}{\sqrt{n}}, \revv \qquad z \in D_1. 
		\end{align*}
		Similarly, we can handle $I_4(z)$. By \eqref{estimate $Z_1^2/Z_i^2 -1$} and \eqref{estimate difference more than 3}, we have
		\begin{align*}
			\vert I_4(z) \vert \leq \vert m_3(\mu) \vert \sum_{i=2}^n \left \vert \theta_i \right \vert^3 \left \vert \frac{Z_1^2(z)}{Z_i^2(z)} -1\right \vert \leq  6\vert m_3(\mu) \vert \sum_{i=2}^n \left( \frac{\vert \theta_i \vert^5}{\Im z \vert Z_1(z) \vert} +   \frac{\vert \theta_i \vert^7}{(\Im z)^2 \vert Z_1^2(z) \vert}\right) \leq \rev  \frac{12B_{7, \rho}\beta_3(\mu)}{\sqrt{n}} \revv
		\end{align*}
		for all $z \in D_1.$
		Last but not least, we can bound $I_5$ by
		\begin{align*}
			\vert I_5\vert = \left \vert m_3(\mu) \sum_{i=2}^n \theta_i^3 \right \vert \leq \beta_3(\mu)\sum_{i=2}^n \vert \theta_i \vert^3 \leq \frac{B_{7, \rho}\beta_3(\mu)}{\sqrt{n}}.
		\end{align*}
		It follows
		\begin{align*}
			\vert r(z) \vert \leq \rev \frac{16B_{7,\rho}\beta_3(\mu)^2m_4(\mu) + 3B_{7,\rho} + 13B_{7,\rho}\beta_3(\mu)}{\sqrt{n}} \leq
			\frac{32B_{7,\rho} \beta_3(\mu)^2 m_4(\mu)}{\sqrt{n}}  \leq \frac{C_{0, \rho, \mu}}{\sqrt{n}}
		\end{align*}
		for all $z \in D_1, n \geq n_1$, and $\rev C_{0,\rho, \mu} := \max\{A_\rho^2, 32B_{7,\rho} \beta_3(\mu)^2 m_4(\mu)\}$.  \revv The estimate for $I_3$ is obvious: 
		\begin{align*}
			\vert I_3\vert = \big \vert \theta_1^2 \big \vert \leq A_\rho^2\frac{\log n}{n} \leq \frac{A^2_\rho}{\sqrt{n}}.
		\end{align*}

		\subsubsection{Analyzing the roots of \texorpdfstring{$P$}{P}} \label{part 3 - roots 3rd oder} 
		\rev The aim of this section is to find the roots of $P(z, \omega)$ for $z$ contained in some suitably chosen subset of $D_1$ and to prove that $Z_1$ is equal to one of these roots in that set. 
		
		Define  \revv
		\begin{align} \label{D_2} \rev
			D_2 := \left \{ z \in \mathbb{C}^+ : \vert \Re z \vert \leq 2-\varepsilon_n, 3 \geq \Im z \geq  \frac{50C_{0, \rho, \mu}}{\sqrt{n}} \right\} \subset D_1, \qquad \varepsilon_n := 50C_{0, \rho, \mu}\sqrt{\frac{\log n}{n}}.
		\end{align} 
		\rev By increasing $n$ to $n \geq n_2 \geq n_1$ for some sufficiently large $n_2 \in \mathbb{N}$, we can assume that $\varepsilon_n < 1$ is valid. Clearly, this implies $\smash{50C_{0, \rho, \mu}n^{-\nicefrac{1}{2}} < 1}$ for those $n$.
		
		\revv  Let $\omega_j = \omega_j(z)$, $j=1,2,3$, denote the roots of $P(z, \omega)$. We firstly prove that for any $z \in D_2$ and sufficiently large $n$ the polynomial $P(z,\omega)$ has a root, \rev say $\omega_1 = \omega_1(z)$, \revv such that 
		\begin{align} \label{estimate omega_1}
			\vert \omega_1 \vert < \frac{12C_{0,\rho, \mu}}{\sqrt{n}}, \qquad \vert \omega_j \vert \geq \frac{12C_{0,\rho, \mu}}{\sqrt{n}}, \qquad j=2,3,
		\end{align}
		hold.
		Let $P_1(z, \omega) := \omega^3 -z\omega^2$, $P_2(z, \omega) := (1- I_3)\omega - r(z) = (1-I_3)(\omega - r_1(z))$
		with 
		\begin{align*}
			r_1(z) := \frac{r(z)}{1-I_3}.
		\end{align*}
	\rev	Note that we have $\vert I_3 \vert \leq \nicefrac{1}{10}$ as well as $\vert r_1(z) \vert \leq \frac{10}{9} \vert r(z) \vert$ for all $n \geq  n_2$. \revv Define \rev $\smash{n_3 := \max\left\{n_2, 75^2C_{0,\rho, \mu}^2\right\}}$ and, from now on, assume that $n \geq n_3$ holds. \revv
		On the circle $\smash{\vert \omega \vert = 12C_{0,\rho, \mu}n^{-\nicefrac{1}{2}}}$, we can estimate
		\begin{align*}
			\vert P_1(z, \omega) \vert = \vert \omega \vert^2 \vert \omega - z \vert	\leq \frac{144C^2_{0, \rho, \mu}}{n}\left( \frac{12C_{0,\rho, \mu}}{\sqrt{n}} + \sqrt{13} \right) \leq \frac{8C_{0,\rho,\mu}}{\sqrt{n}}
		\end{align*}
		as well as
		\begin{align*}
			\vert P_2(z, \omega) \vert \geq \left( 1 - \vert I_3 \vert \right) \vert \omega - r_1(z)\vert  \geq \frac{9}{10}\left(\frac{12C_{0,\rho, \mu}}{\sqrt{n}} -  \vert r_1(z) \vert \right) \geq \frac{9}{10}\left(\frac{12C_{0,\rho, \mu}}{\sqrt{n}} -  \frac{10}{9}\frac{C_{0,\rho, \mu}}{\sqrt{n}} \right) >  \frac{9C_{0,\rho, \mu}}{\sqrt{n}}
		\end{align*}
		for any $z \in D_2.$
		Hence, we have $\vert P_1(z,\omega) \vert < \vert P_2(z,\omega) \vert$ on $\vert \omega \vert = 12C_{0,\rho, \mu}n^{-\nicefrac{1}{2}}$. Rouch\'{e}'s theorem implies that $P(z,w) = P_1(z,w) + P_2(z,w)$ has as many roots as $P_2(z, w)$ in the disk $\vert \omega \vert < 12C_{0,\rho, \mu}n^{-\nicefrac{1}{2}}$. Due to $\vert r_1(z) \vert \leq 2C_{0,\rho, \mu} n^{-\nicefrac{1}{2}}$, we conclude that $P(z,w)$ has exactly one root -- say $\omega_1$ -- in that disk. It is immediate that $\vert \omega_j \vert \geq 12C_{0,\rho, \mu}n^{-\nicefrac{1}{2}}$, $j=2,3$, holds. 
		
		It is easy to see that we have $Z_1(z) \neq \omega_1(z)$ for all $z \in D_2$ and   \rev $n \geq n_3$. \revv Assuming the contrary, we obtain the following contradiction:
	\rev	\begin{align*}
			\frac{50C_{0, \rho, \mu}}{\sqrt{n}} \leq \Im z \leq \Im Z_1(z) = \Im \omega_1(z) \leq \vert \omega_1(z) \vert < \frac{12C_{0,\rho,\mu}}{\sqrt{n}}.
		\end{align*}  \revv
		
		Before we continue with the analysis of the remaining roots $\omega_2$ and $\omega_3$, we will prove that the (restricted) function $\omega_1 : D_2 \rightarrow \mathbb{C}, z \mapsto \omega_1(z)$ is continuous for $n \geq \rev n_3 \revv$. Fix $z_0 \in D_2$ and let $\varepsilon>0, n \geq \rev n_3 \revv$. Without loss of generality, we can assume that 
		\begin{align*}
			\varepsilon < \min\left\{\frac{M}{2}, \frac{12C_{0, \rho, \mu}}{\sqrt{n}} - \vert \omega_1(z_0) \vert \right\}
		\end{align*}
		holds, where $M$ is given by $M :=  \min\left \{ \left\vert \omega_j(z_0) - \omega_1(z_0) \right\vert: j=2,3 \right\} >0$. We obtain  $P(z_0, \omega) \neq 0$ on the circle $\left \vert \omega - \omega_1(z_0) \right \vert = \varepsilon$ due to 
		\begin{align} \label{aux w1 continuous}
		\rev	\vert \omega_j(z_0) - \omega_1(z_0) \vert \geq M > \varepsilon, \qquad j=2,3.
		\end{align}
		Hence, we find $m>0$ with $m \leq  \inf \left \{ \left\vert P(z_0, \omega)\right \vert :\left \vert \omega - \omega_1(z_0) \right \vert = \varepsilon \right\}$.
		Since $z \mapsto r(z)$ is holomorphic and thus continuous in $\mathbb{C}^+$, there exists $\delta_0>0$ such that $\vert r(z_0) - r(z) \vert < \nicefrac{m}{2}$ holds true for all $z \in \mathbb{C}^+$ satisfying $\vert z_0 - z \vert < \delta_0.$ Define $\delta_1 := \min\{ \delta_0, \frac{m}{2}(\varepsilon + \vert \omega_1(z_0) \vert)^{-2}\}>0$. For all $z \in \mathbb{C}^+$ with $\vert z_0 - z \vert < \delta_1$ and for all $\omega$ on the circle $\left\vert \omega -  \omega_1(z_0) \right\vert = \varepsilon$, we conclude
		\begin{align*}
			\vert P(z, \omega) - P(z_0, \omega) \vert = \left\vert (\rev z_0 - z \revv)\omega^2 - r(z) + r(z_0)\right \vert < \delta_1 \left( \varepsilon + \vert \omega_1(z_0) \vert\right)^2 + \frac{m}{2} \leq \vert P(z_0, \omega) \vert.
		\end{align*}
		Making use of Rouch\'{e}'s theorem once again, we obtain that for each $z \in \mathbb{C}^+$ with $\vert z_0 - z \vert < \delta_1$  the polynomials $P(z, \omega)$ and $P(z_0, \omega)$ have the same number of roots in the disk $\vert \omega - \omega_1(z_0) \vert < \varepsilon.$ \rev Because of \eqref{aux w1 continuous}, $P(z, \omega)$ has exactly one root in the disk $\vert \omega - \omega_1(z_0) \vert < \varepsilon.$ \revv
		For $z \in D_2 \cap \{ s \in \mathbb{C}^+: \vert z_0- s \vert < \delta_1 \}$ this root must be $\omega_1(z)$ due to the following observations: Any $\omega \in \mathbb{C}$ with $\vert \omega - \omega_1(z_0) \vert < \varepsilon$ satisfies
		\begin{align*}
			\vert \omega \vert \leq \vert \omega - \omega_1(z_0) \vert + \vert \omega_1(z_0) \vert < \varepsilon + \vert \omega_1(z_0) \vert < \frac{12C_{0, \rho, \mu}}{\sqrt{n}}.
		\end{align*}
		Applying \eqref{estimate omega_1} to the polynomial $P(z,\omega)$ with $z$ as before, we obtain that the corresponding root is $\omega_1(z)$ as claimed above. Finally, we conclude $\vert \omega_1(z) - \omega_1(z_0) \vert < \varepsilon$ for all $z \in D_2$ with $\vert z_0- z \vert < \delta_1$. 
		
		The next step consists of proving that $\omega_2 \neq \omega_3$ holds in $ D_2$ for sufficiently large $n$. First of all, we write $P(z, \omega) = (\omega - \omega_1)P_3(z, \omega)$ with 
		\begin{align*}
			P_3(z, \omega) := \omega^2 - (z - \omega_1)\omega +1 - I_3 - \omega_1(z-\omega_1).
		\end{align*}
		Assume that $\omega_2(z) = \omega_3(z)$ holds for some $z \in D_2$. Then, solving the equation $P_3(z, \omega) = 0$, we obtain $(z-\omega_1)^2 = 4\left( 1-I_3 - \omega_1(z-\omega_1)\right)$, which leads to
		\begin{align*}
			z = -\omega_1 \pm 2\sqrt{1-I_3 + \omega_1^2}.
		\end{align*}
		Observe that
		\begin{align*}
			\Re \big(1-I_3 + \omega_1^2\big)  \geq 1-I_3 - \big\vert \omega_1^2 \big\vert >0
		\end{align*}
		holds for all $n \geq \rev n_3. \revv$
		For a (complex) number $u+iv$ with $u \geq 0$ we have $\left \vert \Re \sqrt{u+iv} \right \vert \geq \sqrt{u}.$
		Combining the last two observations with $\smash{\sqrt{1-x} \geq 1- \sqrt{x}}$, $x \in [0,1]$, and  $\smash{\sqrt{x+y} \leq \sqrt{x} + \sqrt{y}}$, $x,y \geq 0$, we get
		\begin{align*}
			\left \vert \Re \sqrt{1-I_3 + \omega_1^2} \right \vert \geq \sqrt{1-I_3 - \big\vert \omega_1^2 \big\vert}  \geq 1- \left( A_\rho \sqrt{\frac{\log n}{n}} + \frac{12C_{0, \rho, \mu}}{\sqrt{n}} \right)  \geq 1 - 13C_{0, \rho, \mu}\sqrt{\frac{\log n}{n}} > 0
		\end{align*}
		for all $n \geq \rev n_4 := \max\{n_3, 13^4C_{0, \rho, \mu}^4\}$. \revv	It follows 
		\begin{align*}
			\vert \Re z \vert \geq 	2\left \vert \Re \sqrt{1-I_3 + \omega_1^2} \right \vert  - \vert \omega_1\vert \geq 2 - 26C_{0, \rho, \mu}\sqrt{\frac{\log n}{n}} - \frac{12C_{0, \rho, \mu}}{\sqrt{n}} \geq 2 - 38C_{0, \rho, \mu} \sqrt{\frac{\log n}{n}} > 2-\varepsilon_n,
		\end{align*}
		which contradicts the fact that we chose $z$ to be in $D_2.$ Hence, we have $\omega_2 \neq \omega_3$ in $D_2$ for $n \geq \rev n_4.$
		
		Now, we can calculate $\omega_2$ and $\omega_3$ explicitly: Solving the equation $P_3(z, \omega) = 0$, we obtain 
		\begin{align*}
			\omega_j= \frac{1}{2} \left(z + (-1)^{j-1}\sqrt{g(z)} \right) - \frac{\omega_1}{2}, \qquad j=2,3
		\end{align*}
		with 
		\begin{align}  \label{g(z) neq 0}
			g(z) := (z-\omega_1)^2 - 4 +4I_3 + 4\omega_1(z-\omega_1) \neq 0, \qquad z \in D_2, n \geq \rev n_4.
		\end{align}
		Letting $r_2(z):= 4I_3+ (2z - 3\omega_1)\omega_1$, we can write
		\begin{align} \label{w_2, w_3}
			\omega_j = \frac{1}{2}\left( z + (-1)^{j-1}\sqrt{z^2-4+r_2(z)}  \right) - \frac{\omega_1}{2}.
		\end{align}
		For later use, we note that
		\begin{align} \label{estimate r2}
			\vert r_2(z) \vert \leq 4\vert I_3\vert + 2\sqrt{13}\vert \omega_1 \vert + 3\big\vert \omega_1^2 \big \vert \leq \frac{4A_\rho^2\log n}{n} + \frac{24\sqrt{13}C_{0,\rho, \mu}}{\sqrt{n}} + \frac{432C_{0,\rho, \mu}^2}{n} < \frac{94C_{0,\rho, \mu}}{\sqrt{n}} < \frac{3}{5}
		\end{align}
		holds for all $z \in D_2, n \geq \rev n_4. \revv$
		
		The last part of the analysis of the roots of $P$ is devoted to the proof that $Z_1(z) = \omega_3(z)$ holds for all $z \in D_2$ and sufficiently large $n$. We firstly show that $Z_1$ cannot jump between the solutions $\omega_2$ and $\omega_3$ in a sufficiently small neighborhood of any point in $D_2$ for  large $n$. Let $n \geq \rev n_5$ for some threshold $ \rev n_5$ which will be defined later. We prove the following: 
		\begin{align} \label{no jumps}
			\forall z_0 \in D_2 : \exists r_0 = r_0(z_0)>0: \exists! j \in \{2,3\}: \forall z \in D_2 \cap \{\revtwo s \revvtwo \in \mathbb{C} :  \vert \revtwo s \revvtwo - z_0 \vert < r_0 \}: Z_1(z) =  \omega_j(z).
		\end{align}
		Again, we argue by assuming the contrary: Then, there exists $z_0 \in D_2$ such that for all $r_0>0$ we find $z', z'' \in D_2 \cap \{ \revtwo s \revvtwo \in \mathbb{C}: \vert \revtwo s \revvtwo - z_0 \vert < r_0 \}$ with $Z_1(z') = \omega_2(z')$ and $Z_1(z'') = \omega_3(z'')$. Hence, we can construct a sequence $(z_k)_{k \in \mathbb{N}} \subset D_2$ satisfying $\lim_{k \rightarrow \infty} z_k = z_0$ and $Z_1(z_k) = \omega_3(z_k)$. Assume that $Z_1(z_0)=\omega_2(z_0)$ holds; otherwise we just switch the roles of $\omega_2$ and $\omega_3$. Together with the fact that $Z_1$ is continuous, we arrive at $\omega_2(z_0) = Z_1(z_0) = \lim_{k \rightarrow \infty}\omega_3(z_k)$. Using \eqref{w_2, w_3}, this translates into
		\begin{align*}
			\frac{1}{2}\left( z_0 - \sqrt{z_0^2-4 + r_2(z_0)} \right) - \frac{ \omega_1(z_0)}{2} = \lim_{k \rightarrow \infty} 	\frac{1}{2}\left( z_k + \sqrt{z_k^2-4 + r_2(z_k)} \right) - \frac{\omega_1(z_k)}{2}.
		\end{align*}
		Since we will use a continuity argument below, we need to prove that the complex square root function is continuous at complex numbers of the form $\smash{z^2 - 4 + r_2(z)}$ for $z \in D_2.$ We observe
		\begin{align*}
			\Re\big(z^2 - 4 + r_2(z)\big)& \leq (2-\varepsilon_n)^2 -4 -(\Im z)^2 + \vert r_2(z) \vert < -4\varepsilon_n + \varepsilon_n^2 + \frac{94C_{0, \rho, \mu}}{\sqrt{n}}  \\ &= \rev  -200C_{0, \rho, \mu} \sqrt{\frac{\log n}{n}} + 50^2C_{0, \rho, \mu}^2\frac{\log n}{n}  \revv + \frac{94C_{0, \rho, \mu}}{\sqrt{n}} < 0
		\end{align*}
		for all $z \in D_2$ and $n \geq \rev n_5 :=\max \{n_4, N_1\}$ with
	\rev $
			N_1 := \exp(( 12.5C_{0, \rho, \mu} + 0.47)^2). $
	 \revv
		Thus, it follows $z^2 -4 + r_2(z) \in \mathbb{C} \setminus [0, \infty)$  for $z \in D_2$, $n \geq \rev n_5$. We have already proven that $\omega_1\vert_{D_2}$ is continuous. By definition, the same conclusion holds for $r_2\vert_{D_2}$. It follows
		\begin{align*}
			\lim_{k \rightarrow \infty}	\frac{1}{2}\left( z_k + \sqrt{z_k^2-4 + r_2(z_k)} \right) - \frac{\omega_1(z_k)}{2} = 	\frac{1}{2}\left( z_0 + \sqrt{z_0^2-4 + r_2(z_0)}\right) - \frac{\omega_1(z_0)}{2}.
		\end{align*}
		Hence, we must have 
		$\sqrt{z_0^2-4 + r_2(z_0)} = 0.$ Recalling that $z_0^2-4 + r_2(z_0) = g(z_0) \neq 0 $ holds by \eqref{g(z) neq 0},  we arrive at the promised contradiction for all $n \geq \rev n_5.$
		
		The statement in \eqref{no jumps} yields that either $Z_1 = \omega_2$ or $Z_1 = \omega_3$ holds in $D_2$ for $n \geq \rev n_5.$ This can be seen as follows: By compactness of $D_2$, there exist points $z_1, \dots, z_m \in D_2$ and radii $r_1, \dots, r_m$ chosen as in \eqref{no jumps} \rev for some $m \in \mathbb{N}$, \revv such that 
		\begin{align*}
			D_2 \subset \bigcup_{i \in [m]} B_{r_i}(z_i)
		\end{align*}
		holds. In order to ease the notation, write $B_i = B_{r_i}(z_i)$, $i \in [m]$. Now, assume that there exist $x,y \in D_2$ with $Z_1(x) = \omega_2(x)$ and $Z_1(y) = \omega_3(y)$. Then, we find $k,l \in [m]$, $k \neq l$, such that $x \in B_k, y \in B_l$ hold true. We must have $B_k \cap B_l = \emptyset$. It is clear that there exist indices $i_1, \dots, i_j \in [m] \setminus \{k,l\}$, $j \leq m-2$, such that $B_{i_h} \cap B_{i_{h-1}} \neq \emptyset$ holds for all $2\leq h \leq j$ as well as $B_k \cap B_{i_1} \neq \emptyset, B_l \cap B_{i_j} \neq \emptyset$. Due to the choice of the radii, we must have either $Z_1 = \omega_2$ or $Z_1 = \omega_3$ in $B_k \cup B_{i_1} \cup \dots \cup B_{i_j} \cup B_l.$ Clearly, this provides a contradiction and we obtain that $Z_1$ equals either $\omega_2$ or $\omega_3$ in $D_2$ for $n \geq \rev n_5.$

		It remains to exclude the case $Z_1 = \omega_2$ in $D_2$. We will do this by proving that $\vert \omega_2(2i) \vert<  1$ holds for sufficiently large $n$. Together with $\vert Z_1(2i) \vert \geq \Im Z_1(2i) \geq 2$, this yields $Z_1(2i) \neq \omega_2(2i)$. 
		Observe that 
		\begin{align*}
			\vert \omega_2(2i) \vert \leq \frac{1}{2} \left \vert 2i - \sqrt{-8 + r_2(2i)} \right \vert + \frac{1}{2} \vert \omega_1(2i) \vert
		\end{align*}
		as well as  $\vert 2i- \sqrt{-8}  \vert = \vert 2i  - \sqrt{8}i  \vert < 1$ hold. Using the  bound on $r_2$ given in \eqref{estimate r2}, it is easy to see that we have $-8 + r_2(2i) \in \mathbb{C} \setminus [0, \infty)$ for $n \geq \rev n_5$. The continuity of the square root on $ \mathbb{C} \setminus [0, \infty)$ implies the existence of a threshold $\rev N_2 \in \mathbb{N}$ such that for all $\rev n \geq n_6:=\max\{n_5, N_2\}$ we obtain
		\begin{align*}
			\left \vert 2i - \sqrt{-8 + r_2(2i)} - \left(2i - \sqrt{-8}\right) \right \vert \leq \frac{1}{2}.
		\end{align*}
		Combining this with \eqref{estimate omega_1}, we conclude $\vert \omega_2(2i) \vert < 1$ as claimed. It follows $Z_1 = \omega_3$ in $D_2$ for $n \geq \rev n_6$.

		\subsubsection{Quadratic functional equation for \texorpdfstring{\textbf{$Z_1$}}{Z1}} \label{part 4 - 2nd order functional eq}

		In this part we derive a quadratic functional equation for $Z_1$. Since the procedure will be similar to that for the cubic functional equation, we omit some of the details. 
		
		As in \eqref{$Z_1(z) - z + (Z_1(z))^{-1}$}, we can write
		\begin{align} \label{quadratic functional eq.}
			Z_1(z) - z + \frac{1}{Z_1(z)} = \frac{M_1(z) + M_2(z) + M_3}{Z_1(z)}, \qquad z \in \mathbb{C}^+,
		\end{align}
		with 
		\begin{align*}
			M_1(z) :=  Z_1(z) \left( \sum_{i=2}^n F_i(Z_i(z))\,\, - \right. & \left.  Z_i(z) + \frac{\theta_i^2}{Z_i(z)}\right), \qquad
			M_2(z) :=  Z_1(z) \left( \sum_{i=2}^n \frac{\theta_i^2}{Z_1(z)} - \frac{\theta_i^2}{Z_i(z)}\right), \\
			M_3 := &\, Z_1(z) \left( \frac{\sum_{i=1}^n \theta_i^2}{Z_1(z)} - \frac{\sum_{i=2}^n \theta_i^2}{Z_1(z)} \right) = \theta_1^2.
		\end{align*}
		Define 
		\begin{align} \label{def Q}
			q(z) := M_1(z) + M_2(z) + M_3, \qquad Q(z, \omega) := \omega^2 - z\omega +1-q(z)
		\end{align}
		and note that $Q(z, Z_1(z)) = 0$ holds for all $z \in \mathbb{C}^+. $ 
		
	We continue by bounding the term $q(z)$. Using  \eqref{subordination times cauchy expansion}, \eqref{relation F_i - Id and $r_{n,i}$}, \eqref{estimate r_{n,i}}, and \eqref{estimate $Z_1/Z_i -1$ by constant}, we obtain
		\begin{align*}
			\vert M_1(z) \vert &\leq \frac{10}{9}\left \vert Z_1(z)  \right \vert   \sum_{i=2}^n\left \vert r_{n,i}(z)Z_i(z)- \frac{\theta_i^2(1+r_{n,i}(z))}{Z_i(z)} \right \vert  \\ &=
			\frac{10}{9}\left \vert Z_1(z)  \right \vert  \sum_{i=2}^n \left \vert  \frac{1}{Z_i(z)}\int_{\mathbb{R}} \frac{u^3}{Z_i(z) - u} \mu_i(du) - \frac{\theta_i^4}{Z_i^3(z)} - \frac{\theta_i^2}{Z_i^3(z)}\int_{\mathbb{R}} \frac{u^3}{Z_i(z)-u} \mu_i(du)   \right \vert   \\ 
			&\leq   \frac{10}{9} \sum_{i=2}^n \left \vert \frac{Z_1(z)}{Z_i(z)} \right \vert   \vert t_i(z)\vert  \leq \frac{110}{81} \sum_{i=2}^n \vert t_i(z) \vert 
		\end{align*}
		with 
		\begin{align*}
			t_i(z):= \left( 1 - \frac{\theta_i^2}{Z_i^2(z)} \right) \int_{\mathbb{R}} \frac{u^3}{Z_i(z)-u} \mu_i(du) - \frac{\theta_i^4}{Z_i^2(z)} 
		\end{align*}
		for all $z \in D_1$ -- see \eqref{D_1} for the definition of $D_1$ -- \rev and all $n \geq n_6 \geq n_1.$ \revv  The inequality in \eqref{estimate difference more than 3} yields
		\begin{align*} 
			\sum_{i=2}^n\vert t_i(z)\vert \leq  \frac{1}{\Im z}\sum_{i=1}^n\left( \vert \theta_i \vert^3\beta_3(\mu) + \frac{\theta_i^4}{\Im z} + \frac{\vert \theta_i\vert^5\beta_3(\mu)}{(\Im z)^2}\right) \leq
			\frac{1}{\Im z} \frac{3B_{7, \rho}\beta_3(\mu)}{\sqrt{n}}, \qquad z \in D_1,
		\end{align*}
		which results in 
		\begin{align*}
			\vert M_1(z) \vert \leq \frac{1}{\Im z} \frac{5B_{7, \rho}\beta_3(\mu)}{\sqrt{n}}
		\end{align*}
		for all $z \in D_1$ and $n \geq n_6.$
		Making use of \eqref{estimate $Z_1/Z_i -1$} and \eqref{estimate difference more than 3}, we get
		\begin{align*}
			\left \vert M_2(z) \right \vert = \left \vert \sum_{i=2}^n \theta_i^2 \left( 1 - \frac{Z_1(z)}{Z_i(z)}\right)  \right \vert \leq 3\sum_{i=2}^n\frac{\theta_i^4}{\Im z \vert Z_1(z) \vert} \leq \frac{1}{\Im z} \frac{\rev 3 \revv B_{7, \rho}}{\sqrt{n}}, \qquad z \in D_1, n \geq n_6.
		\end{align*}
		Define $D_3 := \left\{ z \in D_1:   \Im z \leq 3\right\}.$
		 Together with
		\begin{align*}
			\vert M_3 \vert \leq A_\rho^2\frac{\log n}{n} \leq \frac{1}{\Im z} \frac{3A_\rho^2}{\sqrt{n}}
		\end{align*}
		for all $z \in D_3$, we finally arrive at \rev
		\begin{align} \label{estimate q}
			\vert q(z) \vert \leq \frac{1}{\Im z} \frac{5B_{7,\rho}\beta_3(\mu)  + 3B_{7,\rho} + 3A_p^2}{\sqrt{n}} \leq \frac{1}{\Im z}\frac{4C_{0, \rho, \mu}}{\sqrt{n}},  \qquad z \in D_3, n \geq n_6, 
		\end{align} 
		where $C_{0, \rho, \mu}$ is defined at the end of Section \ref{part 2 - 3rd order functional eq}. Note that $\vert q(u+i) \vert < \nicefrac{1}{10}$ is valid for all $u \in \mathbb{R}$ and $n$ as before. \revv

		\subsubsection{Analyzing the roots of \texorpdfstring{$Q$}{Q}} \label{part 5 - roots of 2nd}
		The aim of this section consists of calculating the roots of $Q$ and proving that $Z_1$ corresponds to exactly one of these roots \rev in $\mathbb{C}_1 := \{ u+i : u \in \mathbb{R}\}$. \revv
		
		Let $\tilde{\omega}_1 = \tilde{\omega}_1(z)$ and $\tilde{\omega}_2 = \tilde{\omega}_2(z)$ denote the roots of $Q(z, \omega).$ Then, we have 
		\begin{align} \label{tilde omega_1 and tilde omega_2 - roots 2nd order}
			\tilde{\omega}_j = \tilde{\omega}_j(z) = \frac{1}{2} \left( z + (-1)^{j} \sqrt{q_1(z)} \right), \qquad q_1(z) := z^2 - 4 + 4q(z) 
		\end{align}
		for $z \in \mathbb{C}^+.$
		Assume that $q_1(z) = 0$ holds for some $z \in \mathbb{C}^+$. Combining this with \eqref{quadratic functional eq.}, we arrive at $Z_1^2(z) - zZ_1(z) = -1 + q(z) = \nicefrac{-z^2}{4}$, which is equivalent to $Z_1(z) = \nicefrac{z}{2}$. Hence, we obtain the contradiction $\nicefrac{\Im z }{2}= \Im Z_1(z) \geq \Im z$. Thus, we conclude $q_1 \neq 0$ and $\tilde{\omega}_1 \neq \tilde{\omega}_2$ in $\mathbb{C}^+$. 
		
		From now on, we restrict the analysis of the roots to the set $\mathbb{C}_1\subset D_3$.
		We claim that $Z_1$ equals $\tilde{\omega}_2$ in $\mathbb{C}_1$. In order to prove this, assume that $Z_1(u+i) = \tilde{\omega}_1(u+i)$ holds for some $u \in \mathbb{R}.$ Then, we get $\Im \tilde{\omega}_1(u+i) = \Im Z_1(u+i) \geq 1$. However, we have $\Im \tilde{\omega}_1(u+i) \leq \nicefrac{1}{2}$, which can be seen as follows: Suppose that $\Im ((u+i)^2 - 4 + 4q(u+i)) = 0$ holds. Then, we must have $2u + 4\Im q(u+i) = 0$. \rev Together with $\vert q(u+i) \vert < \nicefrac{1}{10}$, this implies \revv
		\begin{align*}
			\Re \left( (u+i)^2 - 4 + 4q(u+i) \right) =  \left( - 2\Im q(u+i)\right)^2  -5 + 4\Re q(u+i) < \frac{4}{100} - 5 + \frac{4}{10} < 0
		\end{align*}
		for all $n \geq n_6.$ Thus, we have $(u+i)^2 -4 + 4q(u+i)\in \mathbb{C}\setminus [0, \infty)$ for all $u \in \mathbb{R}, n \geq n_6.$ In particular, with the help of \eqref{square root real and im formula}, we obtain 
		\begin{align*}
			\Im  \sqrt{(u+i)^2 -4 + 4q(u+i)} \geq 0
		\end{align*}
		and conclude
		\begin{align*}
			\Im \tilde{\omega}_1(u+i) = \frac{1}{2}\Im \left( (u+i) - \sqrt{(u+i)^2 -4 + 4q(u+i)}\right) \leq  \frac{1}{2}.
		\end{align*}
		Hence, it follows $Z_1 = \tilde{\omega}_2$ in $\mathbb{C}_1$ for $n \geq n_6$ as claimed.

		\subsubsection{Application of Proposition  \texorpdfstring{\ref{Bai - Götze's Version}}{2.5}} \label{part 6 - application Bai}

		Let $n \geq n_6$. Later, we will apply Proposition \ref{Bai - Götze's Version} with the following parameters:
		\begin{align} \label{choices Bai}
		\rev 	\tau \in (1,3], \qquad  \rev \gamma < 1, \qquad a := \frac{50C_{0, \rho, \mu}}{\sqrt{n}} \in (0,1), \qquad \varepsilon := 2\varepsilon_n = 100C_{0, \rho, \mu} \sqrt{\frac{\log n}{n}} > 6a.
		\end{align}
		\rev 	Note that $a \in (0,1)$ and $\varepsilon > 6a$ are true because of $\smash{n_6 \geq  \max\{n_2, N_1\} \geq \max\{50^2C_{0, \rho, \mu}^2, e^9\}}$. \revv Moreover, the condition in \eqref{assumption Bai finite integral} is satisfied, which can be seen as follows: We know that $m_2(\mu_\theta) = 1$ holds. The layer cake representation, see \cite[Theorem 1.13]{Lieb2001}, yields 
		\begin{align*}
			1= m_2(\mu_\theta) = 2 \int_0^{\infty} x \mu_\theta([x, \infty)) dx = -2 \int_{-\infty}^0 x \mu_\theta((-\infty, x])dx.
		\end{align*}
		Clearly, the equations above remain valid if we replace $\mu_\theta$ by $\omega$.
		We conclude 
		\begin{align*}
			&\int_{-\infty}^{\infty} \left \vert \mu_\theta((-\infty, x]) - \omega((-\infty, x])  \right \vert dx  \\ \leq & \int_{-1}^1 2dx + \int_{1}^\infty x  \left \vert \mu_\theta((-\infty, x]) - \omega((-\infty, x])  \right \vert dx + \int_{-\infty}^{-1} -x  \left \vert \mu_\theta((-\infty, x]) - \omega((-\infty, x])  \right \vert dx \\  
			\leq &\, 4 + \int_{0}^\infty x  \left \vert \mu_\theta((x, \infty)) - \omega((x, \infty))  \right \vert dx + \int_{-\infty}^{0} -x  \left \vert \mu_\theta((-\infty, x]) - \omega((-\infty, x])  \right \vert dx  \leq 6 < \infty.
		\end{align*}
		Let $S(z) := \frac{1}{2}\left( z + \sqrt{z^2-4} \right)$, $z \in \mathbb{C}^+$, and note that $\nicefrac{1}{S(z)} = G_\omega(z)$  holds for all  $z \in \mathbb{C}^+$. Due to Lemma \ref{reciprocal subordination function as Cauchy transform}, we can find a probability measure $\nu$ such that $G_\nu = \nicefrac{1}{Z_1}$ holds in $\mathbb{C}^+$. It remains to estimate the integrals 
		\begin{align} \label{integrals wrt re}
		\!\!\!\!	\int_{-\infty}^{\infty} \!\! \! \vert G_\theta(u+i) - G_\omega(u+i) \vert du \leq \!\! \int_{-\infty}^{\infty} \!\! \! \vert G_\theta(u+i) - G_\nu(u+i) \vert du + \! \int_{-\infty}^{\infty} \!\! \! \vert G_{\nu}(u+i) - G_\omega(u+i) \vert du 
		\end{align}
		and
		\begin{align} \label{integrals wrt im}
			\!\! \int_a^1\! \vert G_\theta(u\! +\! iv) - G_\omega(u+\! iv) \vert dv \leq \! \! 	\int_a^1 \!   \vert G_{\theta}(u\! +\! iv) - G_\nu(u\! +\! iv) \vert dv + \! 	\int_a^1 \! \vert G_\nu(u\! +\! iv) - G_\omega(u\! +\! iv) \vert  dv 
		\end{align}
		for all $u \in I_\varepsilon := [-2 + \nicefrac{\varepsilon}{2}, 2 - \nicefrac{\varepsilon}{2}] $.  For the integrals in \eqref{integrals wrt re}, we need the quadratic functional equation for $Z_1$ from the last two sections, whereas the integrals in \eqref{integrals wrt im} can be bounded with the help of the cubic functional equation from Sections \ref{part 2 - 3rd order functional eq} and \ref{part 3 - roots 3rd oder}.

		\subsubsection{Bounding the integrals in \texorpdfstring{\eqref{integrals wrt im}}{(3.34)}}  \label{part 7 - bounding integral wrt im }

		Let $n \geq \rev n_6$. 	Knowing that $Z_1 = \omega_3$ holds in $D_2$ (see \eqref{w_2, w_3} for the formula for $\omega_3$), we can calculate
		\begin{align} \label{difference 1/Z_1 - 1/S, cubic}
			\frac{1}{Z_1(z)}  - \frac{1}{S(z)} = \frac{1}{2S(z)Z_1(z)}\left( \omega_1(z) - \frac{r_2(z)}{\sqrt{z^2-4} + \sqrt{z^2-4+r_2(z)}}\right), \qquad z \in D_2.
		\end{align}
		
		\rev Let us analyze the sum of the square roots appearing above in more detail. Together with \eqref{square root real and im formula}, we obtain
		\begin{align} \label{sign real parts}
			\sgn \left(\Re \sqrt{z^2-4}\right) = \sgn(2\Re z \Im z) = \sgn(\Re z), \,\,\,\,\, \sgn\left(\Re \sqrt{z^2-4 + r_2(z)}\right) = \sgn(2\Re z \Im z + \Im r_2(z))
		\end{align}
		for any $z \in D_2.$
		We differentiate two cases: Firstly, choose $z \in D_2$ with $\vert \Re z \vert \geq \rev 1$. \revv 
	Due to $\eqref{estimate r2}$, $\Im z \geq a$, and our assumption on $\vert \Re z \vert$, we conclude that both real parts have the same sign. In more detail, we have
		\begin{align*}
			2\Re z \Im z + \Im r_2(z) \geq \rev 2a \revv  - \vert r_2(z) \vert \geq \frac{100C_{0, \rho, \mu}}{\sqrt{n}} - \frac{94C_{0, \rho, \mu}}{\sqrt{n}} >0
		\end{align*}
		if $\Re z \geq \rev 1 $ and similarly, $	2\Re z \Im z + \Im r_2(z)  <0$	if $\Re z \rev \leq -1.$
		Moreover, we already know that both square roots have positive imaginary part. Hence, if $\vert \Re z \vert \geq \rev 1$, we get
		\begin{align*}
			\left 	\vert \sqrt{z^2-4} + \sqrt{z^2-4+r_2(z)}  \right \vert \geq \left \vert \sqrt{z^2-4} \right \vert 
		\end{align*}
		and thus arrive at 
		\begin{align} \label{estimate difference 1/Z_1 - 1/S for big real parts}
			\left \vert 	\frac{1}{Z_1(z)}  - \frac{1}{S(z)}  \right \vert \leq \frac{1}{2\vert S(z) \vert \vert Z_1(z) \vert}\left( \left \vert \omega_1(z) \right \vert + \left \vert  \frac{r_2(z)}{\sqrt{z^2-4}} \right \vert \right), \qquad z \in D_2, \vert \Re z \vert \geq  \rev 1, \revv n \geq n_6.
		\end{align}
		Now, let us consider $z \in D_2$ with $\vert \Re z \vert < \rev 1.$ It is clear that $u^2-v^2-2  \leq 0$ holds for all $u, v \in \mathbb{R}$ with $\vert u\vert \leq \sqrt{2}$. From this and \eqref{square root real and im formula}, we obtain 
		\begin{align} \label{im sq z2-4 larger 1}
			\left(\Im \sqrt{(u+iv)^2-4} \right)^2 = \frac{1}{2}\left( \sqrt{(u^2 -v^2 -4)^2 + 4u^2v^2} -\left(u^2 -v^2 -4 \right)\right) \geq 1
		\end{align}
		for all $u,v$ as above. In particular, it follows 
		\begin{align*} 
			\left 	\vert \sqrt{z^2-4} + \sqrt{z^2-4+r_2(z)}  \right \vert \geq  \Im \left( \sqrt{z^2-4} + \sqrt{z^2-4+r_2(z)}\right) \geq  \Im \sqrt{z^2-4} \geq 1
		\end{align*}
		for any $z \in D_2$ with $\vert \Re z \vert < 1$ and $n \geq n_6.$ \revv
		Consequently, we arrive at 
		\begin{align} \label{estimate difference 1/Z_1 - 1/S for small real parts}
			\left \vert 	\frac{1}{Z_1(z)}  - \frac{1}{S(z)}  \right \vert \leq \frac{ \left \vert \omega_1(z) \right \vert + \left \vert  r_2(z) \right \vert}{2\vert S(z) \vert \vert Z_1(z) \vert}, \qquad z \in D_2, \vert \Re z \vert < \rev 1 \revv,  n \geq \rev n_6.		\end{align}
			
	\rev 	We continue by proving that 
	\begin{align}  \label{bounds Z_1 by 10 and 1/10}
	\vert Z_1(z) \vert > \frac{1}{10}
	\end{align}
	holds for $z \in D_2$ and $n \geq n_{6}$.  For this purpose, assume that we find $z \in D_2$ with $\vert Z_1(z) \vert \leq \nicefrac{1}{10}$. Then, it follows $\vert Z_1(z) - z \vert \leq \nicefrac{1}{10} + \sqrt{13} < 4$. With the help of the identity $P(z, Z_1(z)) = 0$, we arrive at  
		\begin{align*}
			\frac{1}{10} \geq \vert Z_1(z) \vert > \frac{1}{4}\vert Z_1(z)\vert \vert Z_1(z) - z\vert = \frac{1}{4} \left \vert 1 - I_3 - \frac{r(z)}{Z_1(z)}\right \vert \geq \frac{1}{4} - \frac{1}{4} \left\vert I_3 + \frac{r(z)}{Z_1(z)} \right \vert.
		\end{align*}
		Using the upper bounds on $\vert r(z)\vert$ and $I_3$, one obtains
		\begin{align*}
			\left \vert I_3 + \frac{r(z)}{Z_1(z)} \right \vert \leq A_\rho^2 \frac{\log n}{n} + \frac{1}{\Im z} \cdot \frac{C_{0, \rho, \mu}}{\sqrt{n}} \leq \frac{2}{50}
		\end{align*}
		for $n \geq n_{6}.$ Clearly, this provides the needed contradiction. 
		\revv
		
				Note that we have $\vert \nicefrac{1}{S(z)} \vert = \vert G_\omega(z) \vert \leq 1$ for all $z \in \mathbb{C}^+$. Moreover, a simple calculation shows 
		\begin{align} \label{estimate z^2 -4}
			\left \vert z^2 -4 \right \vert \geq \max \left\{\Im z, \big((\Re z)^2 -5\big)_+\right\}
		\end{align}
		for all $z \in \mathbb{C}^+$.
		Making use of the estimates in \eqref{estimate difference 1/Z_1 - 1/S for big real parts}, \eqref{bounds Z_1 by 10 and 1/10}, \eqref{estimate z^2 -4}, \eqref{estimate omega_1}, and \eqref{estimate r2}, it follows
		\begin{align*}
			\int_a^1 \vert G_\nu(\rev u+iv \revv) - G_\omega(\rev u+iv \revv) \vert \rev dv \revv & = \int_a^1 	\left \vert 	\frac{1}{Z_1(u+iv)}  - \frac{1}{S(u+iv)}  \right \vert dv \\ &  \leq 5\left( \int_a^1 \vert \omega_1(u+iv) \vert dv + \int_a^1  \left \vert  \frac{r_2(u+iv)}{\sqrt{(u+iv)^2-4}} \right \vert dv \right) \\ & \leq  5\left( (1-a)\frac{12C_{0,\rho, \mu}}{\sqrt{n}} + \int_{a}^1 \frac{\left \vert r_2(u+iv) \right \vert }{\sqrt{v}} dv \right) \\ &\leq 5\left( \frac{12C_{0,\rho, \mu}}{\sqrt{n}}+ \left(2-2\sqrt{a} \right)\frac{94C_{0,\rho, \mu}}{\sqrt{n}} \right) \\  & \leq  \frac{1000C_{0,\rho, \mu}}{\sqrt{n}}
		\end{align*}
		for all $u \in I_\varepsilon = [-2 + \nicefrac{\varepsilon}{2}, 2 - \nicefrac{\varepsilon}{2}]$ with $\vert u \vert \geq \rev 1$ and $n \geq \rev n_6$.
		Similarly, together with \eqref{estimate difference 1/Z_1 - 1/S for small real parts}, we obtain 
		\begin{align*}
			\int_a^1 \vert G_\nu(u+iv) - G_\omega(u+iv) \vert dv  \leq 5\left( \int_a^1 \vert \omega_1(u+iv) \vert dv + \int_a^1   \vert r_2(u+iv)  \vert dv \right) \leq  \frac{530C_{0,\rho, \mu}}{\sqrt{n}}
		\end{align*}
		for all $u \in I_\varepsilon$ with $\vert u \vert < \rev 1$ and $n \geq \rev n_6$.
		We conclude 
		\begin{align} \label{integral 1}
			\sup_{u \in I_\varepsilon} \int_a^1 \vert G_\nu(u+iv) - G_\omega(u+iv) \vert dv \leq \frac{1000C_{0,\rho, \mu}}{\sqrt{n}}, \qquad n\geq \rev n_6.
		\end{align}
		
		The other integral appearing on the right-hand side in \eqref{integrals wrt im} can be handled in an easier way: We write
		\begin{align} \label{difference G_mu_theta - G_nu, cubic}
			G_\theta(z) - G_\nu(z) = G_{1}(Z_1(z)) - \frac{1}{Z_1(z)} = \frac{G_{1}(Z_1(z))Z_1(z) - 1 }{Z_1(z)} = \frac{r_{n,1}(z)}{Z_1(z)}
		\end{align}
		for all $z \in \mathbb{C}^+.$
		Using \eqref{subordination times cauchy expansion}, \eqref{bounds Z_1 by 10 and 1/10}, \rev and the inequality 
		\begin{align*}
			\vert \! \log a \vert \leq  \log (50C_{0, \rho, \mu})  + \frac{\log n}{2} \leq \log n
		\end{align*}
	holding for $n \geq \rev n_6 \geq 50^2C_{0, \rho, \mu}^2$, \revv
		 we arrive at 
		\begin{align} \label{integral 2}
			\begin{split}
				\sup_{u \in I_\varepsilon}\int_{a}^1 \left \vert 	G_\theta(u+iv) - G_\nu(u+iv) \right \vert dv & \leq 	\sup_{u \in I_\varepsilon}\int_{a}^1 \frac{1}{\vert Z_1(u+iv) \vert^3} \left( \theta_1^2 + \frac{\vert \theta_1\vert^3\beta_3(\mu)}{v} \right) dv  \\ & \leq 1000 \left((1-a)\theta_1^2 + \vert \theta_1 \vert^3\beta_3(\mu)  \vert \! \log a \vert\right) \\
				& \leq 1000  \left( A_\rho^2 \frac{\log n}{n} + \rev  A_\rho^{3}\beta_3(\mu) \frac{(\log n)^{\nicefrac{5}{2}}}{n^{\nicefrac{3}{2}}}\revv  \right) \\  & \leq 2000 \rev A_\rho^{3} \revv \beta_3(\mu) \frac{\log n}{n}
			\end{split}
		\end{align}
		for all $n \geq \rev n_6$. By \eqref{integral 1} and \eqref{integral 2}, we conclude
		\begin{align} \label{integral 5}
			\sup_{u \in I_\varepsilon}\int_{a}^1 \left \vert 	G_\theta(u+iv) - G_\omega(u+iv) \right \vert dv \leq  \frac{1000C_{0,\rho, \mu}}{\sqrt{n}} +  2000\rev  A_\rho^{3}\beta_3(\mu) \frac{\log n}{n} \leq \frac{C_{1, \rho, \mu}}{\sqrt{n}}
		\end{align}
		for all $n \geq \rev n_6$, where $\rev C_{1, \rho, \mu}$ can be defined by  $\rev C_{1, \rho, \mu} := 1000C_{0,\rho, \mu} + 2000  A_\rho^{3}\beta_3(\mu).$

		\subsubsection{Bounding the integrals in \texorpdfstring{\eqref{integrals wrt re}}{(3.33)}} \label{part 8 - bounding integrals wrt re}

		Let $n \geq \rev n_6$. \rev According to the results in Section \ref{part 5 - roots of 2nd}, we have $Z_1 = \tilde{\omega}_2$ in $\mathbb{C}_1 = \{u+i: u \in \mathbb{R}\}$. As in \eqref{difference 1/Z_1 - 1/S, cubic} \rev and with the help of \eqref{tilde omega_1 and tilde omega_2 - roots 2nd order}, \revv we derive
		\begin{align}  \label{estimate 1/Z_1 - 1/S, quadratic case }
			\frac{1}{Z_1(\rev u+i \revv)} - \frac{1}{S(u+i)} = \frac{1}{Z_1(u+i)S(u+i)} \frac{-2q(u+i)}{\sqrt{(u+i)^2-4} + \sqrt{(u+i)^2-4+4q(u+i)}}
		\end{align}
		for all $u \in \mathbb{R}.$
		
		Similarly to the last section, we can prove
		\begin{align*}
			\left \vert \sqrt{(u+i)^2-4} + \sqrt{(u+i)^2-4+4q(u+i)} \right \vert \geq \left \vert \sqrt{(u+i)^2-4} \right \vert, 
		\end{align*}
		whenever $\vert u \vert \geq \nicefrac{1}{5}$ holds. In more detail, for any $u \in \mathbb{R}$, we have 
		\begin{align*}
			\sgn \left(\Re \sqrt{(u+i)^2-4}\right) = \sgn u, \qquad \sgn \left( \Re\sqrt{(u+i)^2-4 + 4q(u+i)}\right) = \sgn(2u  +4\Im q(u+i)).
		\end{align*}
		Using the inequality $\vert q \vert < \nicefrac{1}{10}$ holding in $\mathbb{C}_1$, compare to \eqref{estimate q}, we see that $\sgn u = \sgn(2u +4\Im q( u+i))$ holds if $u$ is as above. We conclude
		\begin{align} \label{estimate 1/Z_1 - 1/S for big real parts, quadratic case }
			\left \vert \frac{1}{Z_1(u+i)} - \frac{1}{S(u+i)} \right \vert \leq  \frac{1}{\vert Z_1(u+i) \vert \vert S(u+i) \vert} \left \vert \frac{ 2q(u+i)}{\sqrt{(u+i)^2-4}} \right \vert, \qquad \vert u\vert \geq \frac{1}{5}, n \geq n_6.		\end{align}
		
		Moreover, we have
		\begin{align} \label{estimate $G_omega(u+i)$}
			\left \vert \frac{1}{S(u+i)}\right \vert =	\left \vert G_\omega(u+i) \right \vert = \left \vert \frac{1}{2}\left((u+i) - \sqrt{(u+i)^2 -4} \right) \right \vert \leq \frac{2}{ \sqrt{1 + ((\vert u \vert -4)_+)^2}}
		\end{align}
		for all $u \in \mathbb{R},$ which can be proven as follows: Note that the claimed inequality is equivalent to 
		\begin{align*}
			\left \vert \left( (u+i) + \sqrt{(u+i)^2 - 4} \right) \right \vert^2 \geq 
			\begin{cases} 1+(\vert u \vert -4)^2  & \vert  u \vert  \geq 4 \\ 1 & \vert  u \vert <4
			\end{cases}.
		\end{align*}
		The identities in \eqref{square root real and im formula} yield
		\begin{align*}
			\left \vert (u+i) + \sqrt{(u+i)^2 - 4} \right \vert^2 &=\left (u + \Re \sqrt{(u+i)^2-4} \right)^2 + \left(1+ \Im \sqrt{(u+i)^2 -4}\right)^2
			\\ & \geq   1+ \left(\Re \sqrt{(u+i)^2 -4}\right)^2 +  \left(\Im \sqrt{(u+i)^2 -4}\right)^2 \\ & = 1 + \sqrt{(u^2 -5)^2 + 4u^2}  \geq 1
		\end{align*}
		for all $u \in \mathbb{R}$. Note that the first inequality above is true due to  $\smash{\sgn\Re \sqrt{(u+i)^2 -4}= \sgn(u)}$, which implies
		$\smash{u \Re \sqrt{(u+i)^2-4} \geq 0}$. 
		Now, noting that $(u^2-5)^2 + 4u^2 \geq (u-4)^4$ holds for all $u \geq 4$, whereas $(u^2-5)^2 + 4u^2 \geq (-u-4)^4$ is valid for all $u \leq -4$, we get the claim for $\vert u \vert  \geq 4.$  
		
		Combining the inequalities in \eqref{estimate z^2 -4}, \eqref{estimate 1/Z_1 - 1/S for big real parts, quadratic case }, and \eqref{estimate $G_omega(u+i)$} with the fact that $\vert Z_1(u+i) \vert \geq 1$ holds for all $u \in \mathbb{R}$, it follows
		\begin{align} \label{estimate 1/Z_1 - 1/S for big real, quadratic, final}
			\left \vert \frac{1}{Z_1(u+i)} - \frac{1}{S(u+i)} \right \vert \leq \frac{4 \vert q(u+i) \vert}{ \sqrt{1 + ((\vert u \vert -4)_+)^2} \sqrt{\max\{1, (u^2-5)_+\}}}, \qquad \vert u \vert \geq \frac{1}{5}, n \geq n_6.
		\end{align}
		Since we have
		\begin{align} \label{estimate denominator quadratic case 0.01(1+u^2)^2}
			\left(1+( \vert u \vert-4)^2\right)\left(u^2 -5\right)\geq \frac{1}{100}\left(1+u^2\right)^2
		\end{align}
		for all $u \in \mathbb{R}$ with $\vert u \vert \geq 4$, the estimates in \eqref{estimate q} and \eqref{estimate 1/Z_1 - 1/S for big real, quadratic, final} yield
		\begin{align*}
			\int_4^\infty \left \vert G_\nu(u+i) - G_\omega(u+i) \right \vert du
			\leq \rev 40\int_4^\infty \frac{\vert q(u+i) \vert}{1+u^2} du \leq  \frac{\rev 40C_{0, \rho, \mu}}{\sqrt{n}}
		\end{align*}
		and 
		\begin{align*}
			\int_{-\infty}^{-4} \left \vert G_\nu(u+i) - G_\omega(u+i) \right \vert du \leq  \frac{\rev 40C_{0, \rho, \mu}}{\sqrt{n}}.
		\end{align*}
		In the same manner, we derive 
		\begin{align*}
			\int_{\sqrt{6}}^{4} \left \vert G_\nu(u+i) -  G_\omega(u+i) \right \vert &du \leq \rev 4\int_{\sqrt{6}}^{4}  \frac{\vert q(u+i) \vert}{\sqrt{u^2-5}} du \leq  \frac{\rev 16C_{0, \rho, \mu}}{\sqrt{n}}, \\
			\int_{-4}^{-\sqrt{6}} \left \vert G_\nu(u+i) \,\, -  \right. & \left. G_\omega(u+i) \right \vert du \leq  \frac{\rev 16C_{0, \rho, \mu}}{\sqrt{n}}
		\end{align*}
		as well as
		\begin{align*}
			\int_{\nicefrac{1}{5}}^{\sqrt{6}} \left \vert G_\nu(u+i) - G_\omega(u+i) \right \vert & du \leq \rev 4 \int_{\nicefrac{1}{5}}^{\sqrt{6}}  \vert q(u+i) \vert du \leq  \frac{\rev 36 C_{0, \rho, \mu}}{\sqrt{n}}, \\
			\int_{-\sqrt{6}}^{-\nicefrac{1}{5}}\left \vert G_\nu(u+i) \,\, - \right. & \left. G_\omega(u+i) \right \vert du \leq \frac{\rev 36C_{0, \rho, \mu}}{\sqrt{n}}. 
		\end{align*}
		It remains to fill the gap in the integration from $-\nicefrac{1}{5}$ to $\nicefrac{1}{5}$. For this purpose, note that we have $\smash{\Im \sqrt{(u+i)^2 -4} \geq 1}$ for all $u \in \mathbb{R}$. Hence, it follows
		\begin{align*}
			\left \vert \sqrt{(u+i)^2 -4} + \sqrt{(u+i)^2 - 4 +4q(u+i)}   \right \vert \geq \Im  \sqrt{(u+i)^2 -4}  \geq 1, \qquad u \in \mathbb{R}.
		\end{align*}
	Combining this with $\vert Z_1(u+i) \vert \geq 1$, $ \vert S(u+i) \vert \geq 1$ holding for all $u \in \mathbb{R}$, \eqref{estimate q},  and \eqref{estimate 1/Z_1 - 1/S, quadratic case }, we arrive at
		\begin{align*}
			\int_{-\nicefrac{1}{5}}^{\nicefrac{1}{5}} \left \vert G_\nu(u+i) - G_\omega(u+i) \right \vert du \leq 	\int_{-\nicefrac{1}{5}}^{\nicefrac{1}{5}}  2\vert q(u+i) \vert du \leq 	 \frac{\rev 4C_{0, \rho, \mu}}{\sqrt{n}}.
		\end{align*}
		As a conclusion, we obtain
		\begin{align} \label{integral 3}
			\int_{-\infty}^{\infty} \left \vert G_\nu(u+i) - G_\omega(u+i) \right \vert du \leq  \frac{\rev 188C_{0, \rho, \mu}}{\sqrt{n}}, \qquad \qquad n \geq \rev n_6.
		\end{align}
		
		We end this section by bounding the other integral on the right-hand side in \eqref{integrals wrt re}. Using  \eqref{difference G_mu_theta - G_nu, cubic} as well as \eqref{subordination times cauchy expansion}, we get 
		\begin{align*} \rev
			\int_{-4}^{4} \vert G_\theta(u+i) - G_\nu(u+i) \vert du = 	\int_{-4}^{4} \left \vert \frac{r_{n,1}(u+i)}{Z_1(u+i)} \right \vert du \leq 8\left( \theta_1^2 + \vert \theta_1 \vert^3 \beta_3(\mu) \right) \leq  16A^3_\rho \beta_3(\mu)\frac{\log n}{n}.
		\end{align*}
		For the integral with domain of integration outside of $[-4,4]$, we need the estimate
		\begin{align} \label{bounds for abs Z_1 for integral wrt Re of 1/Z_1 - 1/S}
			\vert Z_1(u+i) \vert \geq 	\frac{1}{5}\left( 1 + (\vert u \vert -4)_+\right) 
		\end{align}
		holding	for all $u \in \mathbb{R}$ with $\vert u \vert \geq  \nicefrac{1}{5}$ and $n \geq n_6$. We can prove \eqref{bounds for abs Z_1 for integral wrt Re of 1/Z_1 - 1/S} as follows: From \eqref{estimate q}, \eqref{estimate z^2 -4}, \eqref{estimate 1/Z_1 - 1/S for big real parts, quadratic case }, and \eqref{estimate $G_omega(u+i)$}, we derive
		\begin{align*}
			\vert Z_1(u+i) \vert \geq 	\frac{\sqrt{1+((\vert u \vert -4)_+)^2}}{2}- \frac{\rev 8C_{0, \rho, \mu}}{\sqrt{n}}   \frac{1}{\sqrt{\max\{1, (u^2 -5)_+\}}}
		\end{align*}
		for all $u$ with $\vert u \vert \geq \nicefrac{1}{5}$. We claim that 
		\begin{align*}
			\frac{\sqrt{1+((\vert u \vert -4)_+)^2} }{2}-\! \frac{\rev 8C_{0, \rho, \mu}}{\sqrt{n}}   \frac{1}{\sqrt{\max\{1, (u^2 -5)_+\}}} \geq 	\frac{1}{5}\left( 1 + (\vert u \vert -4)_+\right)
		\end{align*}
holds for all $u \in \mathbb{R}$. \rev Note that we have $\rev 8C_{0, \rho, \mu}n^{-\nicefrac{1}{2}}< \nicefrac{1}{5}$ and that this immediately implies the inequality above in the case $\vert u \vert <4$. Now, \revv consider $u \in \mathbb{R}$ with $\vert u \vert  \geq  4$. Then, we have
		\begin{align*}
			\left \vert \frac{\rev 8C_{0, \rho, \mu} \revv}{\sqrt{n}}   \frac{1}{\sqrt{u^2 -5}} \right \vert \leq \frac{\rev 8}{3}\frac{\rev C_{0, \rho, \mu}}{\sqrt{n}} < \frac{1}{10}, \qquad n \geq \rev n_6.
		\end{align*}
		Together with $\sqrt{1+ (\vert u \vert -4)^2} \geq 2^{-\nicefrac{1}{2}} \left( 1 + (\vert u \vert -4) \right)$ for $\vert u \vert \geq 4$, we conclude 
		\begin{align*}
			\frac{\sqrt{1+(\vert u \vert -4)^2}}{2}  -\! \frac{\rev 8C_{0, \rho, \mu}}{\sqrt{n}}  \frac{1}{\sqrt{u^2 -5}} \geq \frac{1}{5}(1 + (\vert u \vert -4)).
		\end{align*}
	Now, we can use \eqref{bounds for abs Z_1 for integral wrt Re of 1/Z_1 - 1/S} and obtain
		\begin{align*}
			\int_{4}^\infty  \vert G_\theta(u+i) - G_\nu(u+i) \vert du & =  \int_{4}^\infty \left \vert \frac{r_{n,1}(u+i)}{Z_1(u+i)}  \right\vert du \leq  \int_{4}^\infty  \frac{\theta_1^2 +\vert  \theta_1 \vert^3 \beta_3(\mu)}{\vert Z_1(u+i) \vert^2}  du 
			\\ & \leq    25\left(\theta_1^2 + \vert \theta_1 \vert^3\beta_3(\mu)\right)\int_{4}^\infty \frac{1}{(u-3)^2}du \leq 50A^3_\rho \beta_3(\mu)\frac{\log n}{n}
		\end{align*}
		as well as
		\begin{align*}
			\int_{-\infty}^{-4}  \vert G_\theta(u+i) - G_\nu(u+i) \vert du \leq 50A^3_\rho \beta_3(\mu)\frac{\log n}{n}
		\end{align*}
		for all $n \geq \rev n_{6}$. Hence, we derive
		\begin{align*} 
			\int_{-\infty}^{\infty}  \vert G_\theta(u+i) - G_\nu(u+i) \vert du \leq \rev 116 \revv A^3_\rho \beta_3(\mu)\frac{\log n}{n}.
		\end{align*}
		Combining this with \eqref{integral 3}, it follows
		\begin{align} \label{integral 6}
			\int_{-\infty}^{\infty} \vert G_\theta(u + i) - G_\omega(u+i) \vert du  \leq \frac{\rev 188C_{0, \rho, \mu}}{\sqrt{n}} +   \rev 116 \revv A^3_\rho \beta_3(\mu)\frac{\log n}{n} \leq \frac{\rev C_{2,\rho, \mu}}{\sqrt{n}}
		\end{align}
		for all $n \geq \rev n_{6}$ with $\rev C_{2, \rho, \mu}:= 188C_{0, \rho, \mu} + 116A^3_\rho \beta_3(\mu).$	
		
		\subsubsection{Final conclusion} \label{part 9 - final conclusion}
		In this section we end the proof of Theorem \ref{main theorem}. For completeness, let us add a short summary:  Let $\rho \in (0,1)$ and define $\mathcal{F} \subset \mathbb{S}^{n-1}$ as in \eqref{definition F}. Recall that $\sigma_{n-1}(\mathcal{F}) \geq 1- \rho$ holds. Set $n_{\rho, \mu} := \lceil \rev n_{6} \rceil$, where $\lceil \cdot \rceil$ denotes the ceiling function, i.e.\@ $\lceil x \rceil := \min\{ k \in \mathbb{N}: k \geq x\}$ for $x \in \mathbb{R}$. Choose $\theta \in \mathcal{F}$ arbitrarily and let $n \geq n_{\rho, \mu}$. 
		Combining Proposition \ref{Bai - Götze's Version} applied to the choices made in \eqref{choices Bai} with the estimates in \eqref{integral 5} and \eqref{integral 6}, we obtain 
		\begin{align*} \rev
			\Delta(\mu_\theta, \omega) \leq C_\gamma \left( \frac{C_{1, \rho, \mu} + C_{2, \rho, \mu}}{\sqrt{n}} +    \frac{600C_{0, \rho, \mu}}{\sqrt{n}}  +   \left(100C_{0, \rho, \mu} \sqrt{\frac{\log n}{n}} 	\right)^{\nicefrac{3}{2}} \right) \leq \frac{C_{\rho, \mu}}{\sqrt{n}}
		\end{align*}
		for \rev $C_\gamma$ as in Proposition \ref{Bai - Götze's Version} \revv and a suitably chosen constant $C_{\rho, \mu}>0$. Lastly, we note that $C_{\rho, \mu}$ depends solely on $\rho$ through $A_\rho$ and $B_{7, \rho}$ and on $\mu$ via $\beta_3(\mu)$ and $m_4(\mu)$. \\
	
	We would like to comment on some aspects of the proof of Theorem \ref{main theorem} as well as on possible improvements in the next remark.
	\begin{rem}  \phantomsection \label{remark main theorem}
		\begin{enumerate}
			\item[(a)] Our choices of $C_{\rho, \mu}$ and $n_{\rho, \mu}$ are not optimal.
			\item[(b)] By increasing the constant $C_{\rho, \mu}$ if necessary, it is possible to derive a bound on the Kolmogorov distance as proven above holding for all $n$ and not only for sufficiently large $n$.
			\rev\item[(c)]  The use of the cubic functional equation offers a certain degree of flexibility in the sense that small adjustments to the calculations done in Section \ref{part 2 - 3rd order functional eq} lead to an improved upper bound on the coefficient $r(z)$ and thus to a better decay rate for the integrals in \eqref{integrals wrt im}. We refer to the proofs of \revtwo Theorems \ref{main Kolmogorov epsilon unbounded} and \ref{main Delta tilde} \revvtwo for concrete examples of such adjustments. 
			\rev \item[(d)] In contrast to the last statement, the use of the quadratic functional equation leaves no room for improvements. This is due to the fact that for $z \in \mathbb{C}^+$ with $\Im z \in (0,1]$ the order of the term $M_1(z)$ appearing in the polynomial $Q$ is not smaller than $n^{-\nicefrac{1}{2}}$ in general. Unfortunately, in the process of bounding the integrals in \eqref{integrals wrt re}, we cannot replace the quadratic polynomial $Q$ by the cubic polynomial $P$ since we are not able to control the roots of $P$ in $\mathbb{C}_1$, but only in $D_2$; see \eqref{D_2} for the definition of $D_2.$ Hence, we are forced to work with $Q$ and, consequently, the final rate of convergence for the integrals in \eqref{integrals wrt re} and thus for $\Delta(\mu_\theta, \omega)$ is limited to the order $n^{-\nicefrac{1}{2}}$. For completeness, let us mention that the situation changes whenever we assume that the underlying probability measure has \revtwo finite sixth moment \revvtwo and vanishing third moment; compare to the proof of \cref{main Delta vanishing third moment} for more details.
			\item[(e)] Another approach of deriving rates of convergence of the distribution of sums of free bounded random variables to Wigner's semicircle law was introduced by Kargin \cite{Kargin2007a}. However, the application of his method did not improve the previously established rate of convergence for $\Delta(\mu_\theta, \omega)$ since we were not able to show that the Taylor expansion in \cite[equation (10)]{Kargin2007a} holds up to the point $t=1$.
		\end{enumerate}
	\end{rem}		
	
	\subsection{Proof of Corollary \ref{Corollary unbounded}}
 \rev
 We will mainly argue as done in the proof of \cref{main theorem} with some minor changes which do not affect any decay rates but only constants. Hence, the steps of the proof below roughly coincide with the ones of \cref{main theorem}; a few parts of the proof of the theorem are simply summarized into a single step. We use the notation introduced in the last section.  \revv
 
\subsubsection*{Step 1: Construction of $\mathcal{F}$}
		Let $n \geq 4.$ Similarly to \cref{part 1 - construction of F}, we define 
		\begin{align*}
			\mathcal{F}_1 :=  \left\{ \theta \in \mathbb{S}^{n-1}: \max_{i \in [n]} \vert \theta_i \vert \leq 4\sqrt{\frac{\log n}{n}} \right\}, \,\,\,\,
			\mathcal{F}_2 := \left\{ \theta \in \mathbb{S}^{n-1}: \sum_{i=1}^n \vert \theta_i \vert^{7} \leq \frac{B_7}{n^{\nicefrac{5}{2}}} \right\}, \,\,\,\, \mathcal{F} := \mathcal{F}_1 \cap \mathcal{F}_2.
		\end{align*}
		Then, we have
		\begin{align*}
			\sigma_{n-1}\left( \mathcal{F}_1^{\mathsf{c}}\right) \leq \sqrt{\frac{2}{\pi}} \frac{1}{n},  \qquad  	\sigma_{n-1}\left( \mathcal{F}_2^{\mathsf{c}}\right) \leq \exp\left(-n^{\nicefrac{2}{7}}\right) \leq \frac{1}{\sqrt{n}}, \qquad 	\sigma_{n-1}\left( \mathcal{F}^{\mathsf{c}}\right) \leq \frac{2}{\sqrt{n}}.
		\end{align*}
		As before, we observe that \rev $\sum_{i=1}^n \vert \theta_i \vert^{k} \leq B_7 n^{-\nicefrac{(k-2)}{2}}$ holds for any $\theta \in \mathcal{F}_2$ and \revv all $k \in \mathbb{N}$ with $2 < k  \leq 7$. From now on, fix $\theta \in \mathcal{F}$ and assume that $\theta_1^2 = \min_{i \in [n]} \theta_i^2$ holds. 
		
		\subsubsection*{\rev Steps 2+3: Analysis of the cubic functional equation for $Z_1$}
		\rev We proceed analogously to Sections \ref{part 2 - 3rd order functional eq} and \ref{part 3 - roots 3rd oder}. \revv Recall that $P(z, Z_1(z)) =0$ holds for any $z \in \mathbb{C}^+.$ We have to adapt our estimates on the coefficients $I_3$ and $r(z)$ appearing in $P$. Define 
		\begin{align*}
			T_1 := \left \{ z \in \mathbb{C}^+:  \Im z \geq \rev \frac{1}{\sqrt{n}} \revv \right\}.
		\end{align*}
	\rev 	Then, by Cauchy's inequality, Theorem \ref{subordination functions}, and \cite[Theorem 2.6]{Chistyakov2008a}, we have
		\begin{align*}
			\vert r_{n,i}(z) \vert  \leq 4 \frac{\sqrt{\log n}}{n^{\nicefrac{1}{4}}} + 4(\pi c_0)^{\nicefrac{1}{2}} (\beta_3(\mu) B_{7})^{\nicefrac{1}{4}} \frac{\sqrt{\log n}}{n^{\nicefrac{1}{8}}} <\frac{1}{10}, \qquad z \in T_1, i \in [n], n \geq n_1
 		\end{align*}
		for $c_0>0$ taken from \cite[Theorem 2.6]{Chistyakov2008a} and some sufficiently large threshold $n_1 \geq 4$ depending on $\mu$ via $\beta_3(\mu)$. \revv We observe that the statements in \rev \eqref{relation F_i - Id and $r_{n,i}$}, \eqref{first estimate I_1}, and \eqref{estimate $Z_1/Z_i -1$ by constant} -- \eqref{second estimate I_1 unbounded main} \revv remain valid for all $z \in T_1$ \rev and $n \geq n_1$. \revv Arguing as in \eqref{estimate I_1} and using that 
		\begin{align*}
			\frac{\sum_{i=2}^n \vert \theta_i \vert^k}{(\Im z)^l} \leq \frac{B_7}{\rev \sqrt{n}}
		\end{align*}
	holds for all $k,l \in \mathbb{N}$, $k-l=3$, $1 \leq l \leq 4 \leq k \leq 7$, we arrive at 
		\begin{align*}
			\vert I_1(z) \vert \leq \frac{20}{9}\sum_{i=2}^n \vert s_i(z) \vert \leq \rev \frac{16B_7\beta_3(\mu)^2m_4(\mu)}{\sqrt{n}}, \qquad z \in T_1, n \geq n_1.
		\end{align*}
		 \revv Here, $s_i(z)$ is defined as in \eqref{def s_i(z)}. In the same manner, we derive
		\begin{align*}
			\vert I_2(z) \vert \leq \frac{\rev 3B_7}{\rev \sqrt{n}}, \qquad \vert I_3 \vert \leq 16\frac{ \log n}{n}, \qquad \vert I_4(z) \vert \leq \frac{\rev 12B_7\beta_3(\mu)}{\rev \sqrt{n}}, \qquad \vert I_5\vert \leq \frac{\rev B_7\beta_3(\mu)}{\sqrt{n}}
		\end{align*}
		and conclude $	\vert r(z) \vert \leq \rev C_{0, \mu}n^{-\nicefrac{1}{2}}$
		for all $z \in T_1$ and $\rev n \geq n_1$ with $\rev C_{0, \mu} := 32B_7\beta_3(\mu)^2 m_4(\mu) \geq 16.$ 
		
Now, let us analyze the roots of $P$. Define
		\begin{align*}
			T_2 := \left\{ z \in \mathbb{C}^+: \vert \Re z \vert \leq 2 - \varepsilon_n, 3 \geq \Im z \geq \rev \frac{50C_{0, \mu}}{\sqrt{n}} \revv \right\}, \qquad \rev \varepsilon_n:= 50C_{0, \mu}\sqrt{\frac{\log n}{n}}
		\end{align*}
		\rev and let $n_2 \geq n_1$ be chosen in such a way that $\varepsilon_n <1$ holds for all $n \geq n_2.$ \revv
	As before, denote the roots of $P(z,\omega)$ by $\omega_j=\omega_j(z), j=1,2,3.$ Making use of Rouch\'{e}'s theorem, we can prove that for any $z \in T_2$ and all $n \geq \rev n_3:= \max\{n_2, 75^2C_{0, \mu}^2\}$ there exists a root, \rev say $\omega_1(z)$, \revv with
		\begin{align*}
			\vert \omega_1(z) \vert < \frac{12C_{0, \mu}}{\sqrt{n}}, \qquad \vert \omega_j(z) \vert \geq \frac{12C_{0, \mu}}{\sqrt{n}}, \qquad j=2,3.
		\end{align*}
		\rev Moreover, we obtain $Z_1 \neq \omega_1$ in $T_2$ for all $n \geq n_3$ and observe that $\omega_1|_{T_2}$ is continuous. By a contradiction argument, we can show that $\omega_2 \neq \omega_3$ holds in $T_2$ for all $n \geq n_4 := \max \{n_3, 13^4C_{0, \mu}^4\}.$ \revv Now, the remaining (non-equal) roots $\omega_j$, $j=2,3,$ must be of the form as given in \eqref{w_2, w_3}, where the error term $\rev r_2(z)$ admits the estimate 
		\begin{align*}
		\vert r_2(z) \vert \leq \rev 4 \vert I_3 \vert + 2\sqrt{13} \vert \omega_1 \vert + 3\left \vert \omega_1^2 \right \vert \revv \leq  \frac{\rev 94 \revv C_{0, \mu}}{\sqrt{n}} < \frac{3}{5}, \qquad z \in T_2, n \rev \geq n_4.
		\end{align*}
		Making use of the same arguments as before, we can prove that either $Z_1= \omega_2$ or $Z_1= \omega_3$ holds in $T_2$ for all $n \geq \rev n_5$, where $\rev n_5$ is given by
		$\smash{\rev n_5 := \max\{n_4, \exp( ( 12.5C_{0, \mu} + 0.47)^2)\}}$. Last but not least, by possibly increasing $n \geq \rev n_6$ for some sufficiently large \rev $n_6$ arising from a continuity argument, \revv we can show that $Z_1 = \omega_3$ holds in $T_2$ \rev for those $n$. \revv
		
		\subsubsection*{\rev Steps 4+5: Analysis of the quadratic functional equation for $Z_1$} We imitate the arguments in Sections \ref{part 4 - 2nd order functional eq} and \ref{part 5 - roots of 2nd}. Remember that  $Q(z, Z_1(z)) =0$ holds for all $z \in \mathbb{C}^+$. We can estimate
		\begin{align*}
			\vert M_1(z) \vert \leq \frac{1}{\Im z} \frac{\rev 5B_7\beta_3(\mu)}{\sqrt{n}},  \qquad \vert M_2(z) \vert \leq \frac{1}{\Im z}\frac{\rev 3B_7}{\sqrt{n}}
		\end{align*}
		for all $z \in T_1$, $\rev n \geq n_6,$ as well as 
		\begin{align*}
			\vert M_3\vert \leq \frac{1}{\Im z} \frac{48}{\sqrt{n}}, \qquad \vert q(z) \vert \leq \frac{1}{\Im  z}\frac{\rev 5B_7 \beta_3(\mu) + 3B_7 + 48}{\sqrt{n}} \leq  \frac{1}{\Im  z}\frac{\rev 4C_{0, \mu}}{\sqrt{n}}
		\end{align*}
		for $z \in T_3:=\left \{ z \in T_1: \Im z \leq 3 \right\}$ and $n \geq n_6.$ \rev Clearly, we have $\vert q \vert < \nicefrac{1}{10}$ in $\mathbb{C}_1 = \{ u +i: u \in \mathbb{R}\}$ for $n$ as before. 
		
		Let  $\tilde{\omega}_1(z)$ and $ \tilde{\omega}_2(z)$ denote the roots of $Q(z, \omega)$. Then, we have $\tilde{\omega}_1(z) \neq \tilde{\omega}_2(z)$ for all $z \in \mathbb{C}^+$ and the roots are of the form as given in \eqref{tilde omega_1 and tilde omega_2 - roots 2nd order}. \rev As before, we can prove that $Z_1 = \tilde{\omega}_2$ holds in $\mathbb{C}_1$ for $n \geq n_6.$ \revv
		
\subsubsection*{\rev Steps 6-9: Bounding the integrals in Proposition \ref{Bai - Götze's Version}} In this step, we argue as in Section \ref{part 6 - application Bai} up to \cref{part 9 - final conclusion}.
		From now on, assume that $n \geq n_6$ holds. We apply Proposition \ref{Bai - Götze's Version} with the following parameters:
		\begin{align*}
		\rev	\tau \in (1,3],\qquad  \gamma <1, \qquad a := \rev \frac{50C_{0, \mu}}{\sqrt{n}} \in (0,1),  \qquad \varepsilon := 100C_{0, \mu}\sqrt{\frac{\log n}{n}} > 6a.
		\end{align*}
	 It remains to bound the integrals in \eqref{integrals wrt re} and \eqref{integrals wrt im}.
	 
	 We start with the ones in \eqref{integrals wrt im}. \rev Together with $\vert Z_1  \vert> \nicefrac{1}{10}$ holding in $T_2$, it follows \revv 
		\begin{align*} \rev
			\int_a^1 \left\vert \frac{1}{Z_1(u+iv)} - G_\omega(u+iv) \right\vert dv &\leq  \int_a^1 \frac{1}{2\vert S(u+iv) \vert \vert Z_1(u+iv) \vert}\left( \left \vert \omega_1(u+iv) \right \vert + \left \vert  \frac{r_2(u+iv)}{\sqrt{(u+iv)^2-4}} \right \vert \right) dv \\ &\leq 
			\frac{\rev 1000 \revv C_{0, \mu}}{\sqrt{n}}
		\end{align*}
		for $u \in I_\varepsilon := [-2+\nicefrac{\varepsilon}{2}, 2-\nicefrac{\varepsilon}{2}]$ with $\vert u \vert \geq \rev 1$ as well as 
		\begin{align*}
		\rev	\int_a^1 \left\vert \frac{1}{Z_1(u+iv)} - G_\omega(u+iv) \right\vert dv \revv &\leq \int_a^1 \frac{ \left \vert \omega_1(u+iv) \right \vert +  \vert r_2(u+iv) \vert}{2\vert S(u+iv) \vert \vert Z_1(u+iv) \vert}dv \leq
			\frac{\rev 530 \revv C_{0, \mu}}{\sqrt{n}}
		\end{align*}
		for $u \in I_\varepsilon$ with  $\vert u \vert \rev < 1$. Hence, we get
		\begin{align*} 
			\sup_{u \in I_\varepsilon} 	\int_a^1 \left\vert \rev \frac{1}{Z_1(u+iv)} \revv- G_\omega(u+iv) \right \vert dv \leq \frac{\rev 1000 \revv C_{0, \mu}}{\sqrt{n}}.
		\end{align*}
		In analogy to \eqref{integral 2}, we deduce
		\begin{align*}
			\sup_{u \in I_\varepsilon}\int_{a}^1 \left \vert 	G_\theta(u+iv) - \rev \frac{1}{Z_1(u+iv)} \revv \right \vert dv  \leq \rev 80000 \revv \beta_3(\mu) \frac{\log n}{n}.
		\end{align*}
		
		Similarly, we can estimate the integrals in \eqref{integrals wrt re}. We have 
		\begin{align*} 
			\int_{-\infty}^{\infty} \left\vert \rev \frac{1}{Z_1(u+i)} \revv - G_\omega(u+i) \right\vert du \leq \frac{\rev 188 \revv C_{0, \mu}}{\sqrt{n}}, \qquad
			\int_{-\infty}^{\infty}  \left\vert G_\theta(u+i) - \rev \frac{1}{Z_1(u+i)} \revv \right\vert du \leq \rev 4640 \revv \beta_3(\mu)\frac{\log n}{n}. 
		\end{align*}
		 Finally, together with Proposition \ref{Bai - Götze's Version}, we obtain 
		\begin{align*} \rev
			\!\!\!\Delta(\mu_\theta, \omega)  \leq C_\gamma \left( \frac{1188 C_{0, \mu}}{\sqrt{n}} + 84640\beta_3(\mu) \frac{\log n}{n} + \frac{600C_{0, \mu}}{\sqrt{n}} + \left(100C_{0, \mu} \sqrt{\frac{\log n}{n}}\right)^{\nicefrac{3}{2}} \right)	 \leq \rev \frac{C_1\beta_3(\mu)^2m_4(\mu)}{\sqrt{n}}
		\end{align*}
		for \rev a numerical constant \rev $C_1 >0$, \rev $C_\gamma$ as in Proposition \ref{Bai - Götze's Version}, and all $n \geq n_{6}$. \revv We end the proof by observing that 
		\begin{align*}
			\mathbb{E}_\theta (\Delta(\mu_\theta, \omega))\leq \int_{\mathcal{F}} \Delta(\mu_\theta, \omega) d\sigma_{n-1}+ \rev \sigma_{n-1}\left(\mathcal{F}^{\mathsf{c}}\right)\revv \leq \rev \frac{C_1\beta_3(\mu)^2m_4(\mu)}{\sqrt{n}} +\frac{2}{\sqrt{n}}
		\end{align*}
		 holds for all $n \geq n_\mu := \lceil \rev n_6 \rceil.$

		\section{\texorpdfstring{Rate of convergence with respect to $\Delta_\varepsilon$: \revtwo Proof of Theorem \ref{main Kolmogorov epsilon unbounded}}{Proof of Theorem 1.3}} \label {Section: Delta epsilon}
		
 \revtwo The aim of this section is to prove \cref{main Kolmogorov epsilon unbounded} providing an appropriate free analog of the Klartag-Sodin result. In contrast to the previous section, we measure the rate of convergence of the distribution of a weighted sum to Wigner's semicircle law with respect to
 the pseudometric $\Delta_\varepsilon$ defined in \eqref{def Kolmogorov epsilon} and assume that the underlying probability measure has finite sixth moment. \revvtwo

\subsection*{Proof of \texorpdfstring{\cref{main Kolmogorov epsilon unbounded}}{Theorem 1.3}}
The proof combines a few new arguments (see \revtwo Steps 2 and 3 \revvtwo below) arising from the replacement of the Kolmogorov distance by $\Delta_\varepsilon$ with some ideas used in the proof of \cref{main theorem}. 

\revtwo
Let us briefly sketch the main steps: We start with the definition of the set $\mathcal{F}$. Then, in the second step, we bound $\Delta_\varepsilon$ in terms of three integrals by using the Stieltjes-Perron inversion formula and Cauchy's integral theorem. In the third step, we prove that one of these integrals can be handled by means of Theorem \ref{Thm 2.7 Bao, Erdos, Schnelli}, whereas the other two integrals can be controlled with the help of the cubic functional equation for the subordination functions as shown in the fourth and fifth step. The final conclusion of the proof is carried out in the sixth step. 
\revvtwo

\subsubsection*{Step 1: Construction of \texorpdfstring{$\mathcal{F}$}{F}}
		Fix $\rho \in (0,1)$, $n \geq 4$. For $A_\rho$ as before and \revtwo $B_{9, \rho} := \nolinebreak B_9(\log \nicefrac{3}{\rho})^{\nicefrac{9}{2}}$ with $B_9$ as in \cref{5.3.3 Bobkov}, \revvtwo we set
		\begin{align*}
			\mathcal{F}_1 &:= \Bigg\{ \theta \in \mathbb{S}^{n-1}:   \max_{i \in [n]} \vert \theta_i \vert \leq A_\rho\sqrt{\frac{\log n}{n}} \Bigg\}, \qquad
			\mathcal{F}_2 := \left\{ \theta \in \mathbb{S}^{n-1}: \revtwo \sum_{i=1}^n \vert \theta_i \vert^{9} \leq \frac{{B}_{9, \rho}}{n^{\nicefrac{7}{2}}} \revvtwo \right\},\\
			\mathcal{F}_3 &:=  \left \{ \theta \in \mathbb{S}^{n-1}: \left \vert \sum_{i=1}^n\theta_i^3 \right \vert \leq \frac{\left(23\log\frac{6}{\rho}\right)^{\nicefrac{3}{2}}}{n}   \right\}, \qquad \mathcal{F} := \mathcal{F}_1 \cap \mathcal{F}_2 \cap \mathcal{F}_3.
		\end{align*} \revvtwo
	Using the results presented in \cref{section: concentration inequalities}, we obtain  $\sigma_{n-1}(\mathcal{F}^{\mathsf{c}}) \leq \rho$. As before, choose $\theta \in \mathcal{F}$ and without loss of generality assume that $\smash{\theta_1^2 = \min_{i \in [n]} \theta_i^2}$ holds.  
		
		\revvtwo 
	\subsubsection*{Step 2: Establishing an upper bound for \texorpdfstring{$\Delta_\varepsilon$}{F}} 
	\rev Fix $\varepsilon \in (0,1)$ and define $I_\varepsilon' := [-2+\varepsilon, 2-\varepsilon]$. \revv
	Observe that $\omega$ and $\mu_\theta$ both do not have any atoms in $I_\varepsilon'$.  For $\omega$ this is clear, whereas for $\mu_\theta$ this follows from the fact that $\mu_\theta$ is absolutely continuous in $I_\varepsilon'$ for $n \geq n_1$ for some suitably chosen threshold $n_1 \geq 4$; compare to \cite[Corollary 2.4]{Bercovici2022}. Note that the corollary just mentioned requires that the family $\{\mu_1, \dots, \mu_n \}$ is an infinitesimal array of probability measures. Since this can be verified easily by using $\theta \in \mathcal{F}_1$, we omit the proof. Together with the Stieltjes-Perron inversion formula, we can write 
	\begin{align*}
		\Delta_\varepsilon(\mu_\theta, \omega) =	\sup_{x \in I_\varepsilon'} \big\vert \mu_\theta((-2+\varepsilon, x])  - \omega((-2+ \varepsilon, x]) \big\vert = 	\sup_{x \in I_\varepsilon'} \left \vert \lim_{\delta \downarrow 0}  \frac{1}{\pi} \int_{-2+\varepsilon}^x \Im \left(G_\theta(u+i\delta) - G_\omega(u+i\delta)\right) du   \right \vert
	\end{align*}
	for $n \geq n_1.$
	Cauchy's integral theorem yields
	\begin{align*}
		\int_{-2+\varepsilon}^x \Im &\left(G_\theta(u+i\delta) - G_\omega(u+i\delta) \right) du  = 	\int_{-2+\varepsilon}^x \Im \left(G_\theta(u+i) - G_\omega(u+i) \right) du   \\ &+	\rev i \revv \int_{\delta}^1 \Im \left(G_\theta(-2+\varepsilon+iv) - G_\omega(-2+\varepsilon+iv) \right) dv   - \rev i \revv	\int_{\delta}^1 \Im \left(G_\theta(x+iv) - G_\omega(x+iv) \right) dv 
	\end{align*}
	for all $\delta \in (0,1)$ and $x \in I_\varepsilon'$. Note that we can apply the integral theorem since the Cauchy transforms $G_\theta$ and $G_\omega$ are holomorphic in $\mathbb{C}^+$. Finally, for any choice of $a_n \in (0,1)$ and all $n \geq n_1$, the last two calculations imply 
	\begin{align} \label{bound Delta_epsilon}
		\Delta_\varepsilon(\mu_\theta, \omega)  \leq \frac{1}{\pi} \int_{-2+\varepsilon}^{2-\varepsilon} \left \vert  G_\theta(u+i) - G_\omega(u+i)\right \vert  du  + \sup_{x \in I_\varepsilon'}  \frac{2}{\pi} \int_{a_n}^1 \left \vert  G_\theta(x+iv) - G_\omega(x+iv)\right \vert  dv + I_n, 
	\end{align} \revv
	where $I_n$ is given by 
	\begin{align} \label{def I_n}
		I_n := \sup_{x \in I_\varepsilon'} \frac{2}{\pi} \int_{0}^{a_n}  \left \vert  G_\theta(x+iv) - G_\omega(x+iv)\right \vert  dv.
	\end{align}
\rev  Let us remark that in general the term $I_n$ is problematic since Cauchy transforms near the real axis are often not controllable without further information. However, in the setting at hand, $I_n$ is well-defined, which can be seen by means of Theorem \ref{Thm 2.7 Bao, Erdos, Schnelli} as shown in the next step, or  alternatively with the help of the results in \cite{Bercovici2022}.


	
\subsubsection*{Step 3: Bounding the term \texorpdfstring{$I_n$ }{In}} 
The goal of this section is to estimate $I_n$ given in \eqref{def I_n} for the choice $\revtwo a_n := (n \log n)^{-\nicefrac{1}{2}}$.	The general procedure for this is as follows: 
We will divide the weighted sum $S_\theta$ into the sums 
$\smash{\sum_{i=1}^{M_n} \theta_i X_i}$ and $\smash{\sum_{i=M_n+1}^{n} \theta_i X_i}$
for \revtwo some $1 \leq M_n < n$. \revvtwo With some effort, one can prove that the analytic distribution of both partial sums is close to Wigner's semicircle law with variance $\nicefrac{1}{2}$ denoted by $\omega_{\nicefrac{1}{2}}$. Combining this with Theorem \ref{Thm 2.7 Bao, Erdos, Schnelli}, we will be able to bound $I_n$ appropriately. Let us note that the restriction to the interval $I_\varepsilon'$ for fixed $\varepsilon>0$ in the definition of $I_n$ is essential for the above argument; otherwise the application of Theorem \ref{Thm 2.7 Bao, Erdos, Schnelli} would not be possible. \revv

\rev Before we start with the explicit partition of $S_\theta$, let us do some preparatory work. \revv
Define
		\begin{align*} 
			M_n := \max \left\{ l \in [n]: \theta_1^2 + \theta_2^2 + \dots + \theta_l^2 \leq \frac{1}{2}\right\}.
		\end{align*}
As before, let $\lceil \cdot \rceil$ denote the ceiling function, whereas the floor function is denoted by $\lfloor \cdot \rfloor$. We have
		\begin{align} \label{bounds on M_n}
			N_{n,1} := \left \lfloor \frac{n}{2A_\rho^2\log n} \right \rfloor \leq M_n  < N_{n,2} := \left \lceil n - \frac{n}{4A_\rho^2\log n} \right \rceil,
		\end{align}
	which can be seen as follows: Together with  $\theta \in \mathcal{F} \subset \mathcal{F}_1$, one deduces
		\begin{align*}
			\sum_{i=1}^{N_{n,1}} \theta_i^2 \leq N_{n,1} \frac{A_\rho^2 \log n}{n} \leq \frac{1}{2}
		\end{align*}
		proving the left-hand side of \eqref{bounds on M_n}. The other inequality can be verified in the same manner by noting that it is equivalent to $	\smash{\sum_{i = N_{n,2}+1}^{n} \theta_i^2 < \nicefrac{1}{2}}$. 
		\rev From now on, assume that $n \geq n_2:=\max\{n_1, 64e^2A_\rho^4\}$ holds. Then, we get
		\begin{align} \label{lower bounds N_n,i}
			N_{n,1} \geq \frac{n}{4A_\rho^2 \log n}, \qquad N_{n,2} \leq n - \frac{n}{8A_\rho^2 \log n}, \qquad   \frac{A_\rho^2 \log n}{n} < \frac{1}{4}. 
		\end{align} 
		The definition of $M_n$ yields 
		\begin{align}  \label{concentration partial sum of weights 1/2}
			\begin{split}
		0 &\leq \frac{1}{2} - \sum_{i=1}^{M_n} \theta_i^2 < \sum_{i=1}^{M_n+1} \theta_i^2 - \sum_{i=1}^{M_n} \theta_i^2 = \theta^2_{M_n+1} \leq 
				\frac{A_\rho^2 \log n}{n} < \rev \frac{1}{4}, \revv \\
			0 &\leq  \sum_{i=M_n+1}^n \!\! \theta_i^2 - \frac{1}{2}   =1 - \sum_{i=1}^{M_n} \theta_i^2 - \frac{1}{2}  = 	 \frac{1}{2} - \sum_{i=1}^{M_n} \theta_i^2 \leq	\frac{A_\rho^2 \log n}{n} <  \frac{1}{4}.
			\end{split}	
	\end{align}
	In particular, letting
		\begin{align*}
			c_{n,1} := \sqrt{\sum_{i=1}^{M_n} \theta_i^2} , \qquad c_{n,2} := \sqrt{\sum_{i=M_n+1}^{n} \theta_i^2},
		\end{align*}
		we obtain $c_{n,1}, c_{n,2} > \nicefrac{1}{2}$. Together with the fact that $x \mapsto ((\log x)x^{-1})^{\nicefrac{1}{2}}$ is decreasing in $(e, \infty)$ as well as with $\theta \in \mathcal{F}$,
		we deduce
		\begin{align} \label{improvement inner 2 estimates for theta^i max}
			\begin{split}	\max_{i \in \{ 1, \dots, M_n\}} \left \vert \frac{\theta_i}{c_{n,1}} \right \vert\leq 2\max_{i \in [n]} \vert \theta_i \vert &\leq 2A_\rho \sqrt{\frac{\log n}{n}} \leq 2A_\rho \sqrt{\frac{\log M_n}{M_n}}, \\
				\max_{i \in \{ M_n +1, \dots, n\}} \left \vert  \frac{\theta_i }{c_{n,2}} \right \vert &\leq 2A_\rho \sqrt{\frac{\log (n- M_n)}{n-M_n}}.
			\end{split}		
		\end{align}
\revtwo Similarly, one obtains
		\begin{align}  \label{improvement inner 2 estimates for theta^i l_7 norm}
			 \sum_{i=1}^{M_n} \left \vert \frac{\theta_i }{c_{n,1}} \right\vert^7 \leq 2^7 \sum_{i=1}^{n} \vert \theta_i \vert^7 \leq \frac{2^7 \revtwo B_{9, \rho}}{n^{\nicefrac{5}{2}}} \leq  \frac{2^7B_{9, \rho}}{M_n^{\nicefrac{5}{2}}}, \qquad \qquad
				 \sum_{i=M_n+1}^{n} \left \vert \frac{\theta_i }{c_{n,2}} \right\vert^7  \leq \frac{2^7B_{9,\rho}}{(n- M_n)^{\nicefrac{5}{2}}}.
		\end{align} \revvtwo
		
Now, let us divide $S_\theta$ as described before. On the distributional level, we define
		\begin{align*}
			\theta^1 &:=	\left( \theta_1, \dots, \theta_{M_n} \right), \qquad \,\,\,\,\,\,\, \theta^2 := \left( \theta_{M_n+1}, \dots, \theta_n \right), \\	\mu_{\theta^1} &:= \mu_1 \boxplus \dots \boxplus \mu_{M_n}, \qquad \mu_{\theta^2} := \mu_{M_n+1} \ \boxplus \dots \boxplus \mu_n.
		\end{align*}
		Observe that $\mu_{\theta^1}$ is the distribution of the sum $\sum_{i=1}^{M_n} \theta_i X_i$ \revtwo with variance $c_{n,1}^2$, \revvtwo whereas $\mu_{\theta^2}$ is the distribution of $\sum_{i=M_n+1}^{n} \theta_i X_i$ with variance $c_{n,2}^2$. Moreover, the identity $\mu_\theta = \mu_{\theta^1} \boxplus \mu_{\theta^2}$ is valid. 

		 We continue by showing that $\mu_{\theta^{1}}$  and $\mu_{\theta^2}$ are close to Wigner's semicircle law with variance $\nicefrac{1}{2}$ measured with respect to the L\'{e}vy distance \revtwo $d_L$. \revvtwo For this, let us analyze the Kolmogorov distance between $\mu_{\theta^1}$ and $\smash{\omega_{c_{n,1}^2}}$ first.  Since this distance is invariant under dilation, it suffices to study $ \smash{\Delta(D_{\nicefrac{1}{c_{n,1}}}\mu_{\theta^1}, \omega)}$. Considering the fact that  $\smash{ \nicefrac{\theta^1}{c_{n,1}}}$ is a vector in $\mathbb{S}^{M_n-1}$, one \revtwo is \revvtwo tempted to apply \cref{main theorem} to $\smash{D_{\nicefrac{1}{c_{n,1}}} \mu_{\theta^1}}$ with any $n$ appearing in that theorem replaced by $M_n$ and \revtwo with \revvtwo $\rho$ as fixed \revtwo in the first step of this proof. \revvtwo Unfortunately, as  $\nicefrac{\theta^1}{c_{n,1}}$ might not be contained in the subset $\smash{\mathcal{F}_{M_n}  \subset \mathbb{S}^{M_n-1}}$ constructed in \cref{main theorem}, this is not possible. However, a direct comparison of the estimates obtained in \eqref{improvement inner 2 estimates for theta^i max} and \eqref{improvement inner 2 estimates for theta^i l_7 norm} with the \revtwo explicit \revvtwo definition of $\mathcal{F}_{M_n}$ given in \cref{part 1 - construction of F} shows that $\nicefrac{\theta^1}{c_{n,1}}$ is not far from being in $\mathcal{F}_{M_n}$ -- \revtwo at least if we focus on the decay rates and not on the constants. \revvtwo \revtwo It is easy to see that if we replace $A_\rho$ and $B_{7, \rho}$ appearing in the definition of $\mathcal{F}_{M_n}$ by $2A_\rho$ and $2^7B_{9, \rho}$, we can repeat the calculations in the proof of \cref{main theorem} (up to the ones in \cref{part 1 - construction of F}) leading to different constants, but unchanged decay rates. \revvtwo In view of this observation, we get 
		\begin{align*}
			\Delta\left(D_{\nicefrac{1}{c_{n,1}}}\mu_{\theta^1}, \omega\right) \leq \frac{D_{1, \rho, \mu}}{\sqrt{M_n}}
		\end{align*}
for sufficiently large $n \geq n_3$ with $n_3 \in \mathbb{N}$ depending on $\rho$ and $\mu$ and a constant $D_{1, \rho, \mu}>0$ depending on $\rho$, $\beta_3(\mu)$, and $m_4(\mu)$. Note that both $n_3$  and $D_{1, \rho, \mu}$ might be larger than the corresponding analogs given in Theorem \ref{main theorem}. 
		Similarly, due to $\nicefrac{\theta^2}{c_{n,2}} \in \mathbb{S}^{n-M_n-1}$, we obtain
		\begin{align*} 
			\Delta\left(D_{\nicefrac{1}{c_{n,2}}}\mu_{\theta^2}, \omega\right) \leq \frac{D_{2, \rho, \mu}}{\sqrt{n- M_n}}
		\end{align*}
		for all $n \geq n_4$ for some suitably chosen $n_4 \in \mathbb{N}$ and a constant $D_{2, \rho, \mu}>0.$
	Together with Lemma \ref{Diliatation Levy} as well as \eqref{bounds on M_n}, \eqref{lower bounds N_n,i}, and \eqref{concentration partial sum of weights 1/2}, we arrive at 
		\begin{align*}
			d_L \left(\mu_{\theta^1}, \omega_{\nicefrac{1}{2}}\right) &\leq \Delta \left(\mu_{\theta^1}, \omega_{c_{n,1}^2}\right) + d_L \left(\omega_{\nicefrac{1}{2}}, \omega_{c_{n,1}^2}\right) \leq    \revtwo \frac{D_{1, \rho, \mu}}{\sqrt{M_n}} \revvtwo +  2\left \vert \sqrt{\frac{1}{2}} - \sqrt{\sum_{i=1}^{M_n} \theta_i^2} \, \right \vert \\
			& \leq   2A_\rho D_{1, \rho, \mu} \sqrt{\frac{\log n}{n}} +  2 \sqrt{\frac{1}{2} - \sum_{i=1}^{M_n} \theta_i^2}  \leq  2A_\rho(D_{1, \rho, \mu} + 1)\sqrt{	\frac{\log n}{n}}
		\end{align*}
	and
		\begin{align*}
			d_L \left(\mu_{\theta^2}, \omega_{\nicefrac{1}{2}}\right) \leq  \revtwo\frac{D_{2, \rho, \mu}}{\sqrt{n- M_n}} \revvtwo +  2\left \vert \sqrt{\frac{1}{2}} - \sqrt{\sum_{i=M_n+1}^{n} \theta_i^2} \, \right \vert  \leq  3A_\rho(D_{2, \rho, \mu} + 1)\sqrt{	\frac{\log n}{n}}
		\end{align*}
		for all $n \geq \max\{n_2, n_3, n_4\}.$ \revv

		According to Theorem \ref{Thm 2.7 Bao, Erdos, Schnelli},
		applied to $\rev I_\varepsilon'$ and $\eta = 4$, we find constants $b_\varepsilon>0$ and $Z_\varepsilon < \infty$ -- both dependent on $\varepsilon$ but independent of $n$, \revtwo $\rho$, and $\mu$ \revvtwo -- such that for all probability measures $\mu_A, \mu_B$ on $\mathbb{R}$ satisfying the inequality $d_L(\omega_{\nicefrac{1}{2}},\mu_A) + d_L(\omega_{\nicefrac{1}{2}},\mu_B) \leq b _\varepsilon$,
		the estimate 
		\begin{align*}
			\max_{z \in S_{\rev I_\varepsilon'}(0,4)} \left \vert G_\omega(z) - G_{\mu_A \boxplus \mu_B}(z)  \right \vert \leq Z_\varepsilon \left(d_L(\omega_{\nicefrac{1}{2}},\mu_A) + d_L(\omega_{\nicefrac{1}{2}},\mu_B) \right)
		\end{align*}
		holds. By increasing $n$ if necessary, say $\rev n \geq n_5 \geq \max\{n_2, n_3, n_4\}$, we obtain 
		\begin{align*}  
			d_L \left(\mu_{\theta^1}, \omega_{\nicefrac{1}{2}}\right) + d_L \left(\mu_{\theta^2}, \omega_{\nicefrac{1}{2}}\right) \leq  \rev A_\rho(2D_{1, \rho, \mu} + 3D_{2,\rho, \mu} +5) \sqrt{\frac{\log n}{n}}\revv \leq b_\varepsilon.
		\end{align*}
	Hence, it follows
		\begin{align} \label{upper bound G_theta Delta tilde unbounded}
			\max_{z \in S_{\rev I_\varepsilon'}(0,4)}\left \vert G_\omega(z) - G_{\theta}(z)  \right \vert \leq \rev Z_\varepsilon  A_\rho(2D_{1, \rho, \mu} + 3D_{2,\rho, \mu} +5) \sqrt{\frac{\log n}{n}} \revtwo\leq Z_\varepsilon b_\varepsilon, \revvtwo \revv \qquad n \geq \rev n_5.
		\end{align}
		Finally, with $\revtwo a_n := (n\log n)^{-\nicefrac{1}{ 2}} < 1$ as already indicated at the beginning of this \revtwo step, \revvtwo we get
		\begin{align} \label{estimate error term I_n}
			\vert I_n \vert \leq \rev  \frac{2Z_\varepsilon  A_\rho(2D_{1, \rho, \mu} + 3D_{2,\rho, \mu} +5)}{\pi}  \sqrt{\frac{\log n}{n}} a_n  \leq \frac{2Z_\varepsilon  A_\rho(2D_{1, \rho, \mu} + 3D_{2,\rho, \mu} +5)}{\pi}  \revtwo \frac{1}{n}
		\end{align}
		for $n \geq \rev n_5$.
		
\subsubsection*{Step 4: Analysis of the cubic functional equation for \texorpdfstring{$Z_1$}{Z_1}} 

It remains to bound the integrals in \eqref{bound Delta_epsilon}. Remarkably and in contrast to the proof of Theorem \ref{main theorem}, both integrals can be controlled with the help of the cubic functional equation for the subordination function $Z_1$. \revtwo In order to analyze this functional equation, we mainly argue as in \cref{part 2 - 3rd order functional eq,part 3 - roots 3rd oder}, but include the higher order expansions in \eqref{subordination times cauchy expansion more moments}.  

Combining \eqref{functional eq Z_1 } with \eqref{def P}, we have 
\begin{align*}
	P(z, Z_1(z)) = Z_1^3(z) - zZ_1^2(z) + (1-I_3)Z_1(z) - r(z) = 0
\end{align*}
for all $z \in \mathbb{C}^+$. \revtwo Due to $\smash{\theta_1^2 = \min_{i \in [n]} \theta_i^2}$, one obtains
\begin{align} \label{improved estimate I_3}
	\vert I_3 \vert = \theta_1^2 = \theta_1^2 \sum_{i=1}^{n} \theta_i^2 \leq \sum_{i=1}^{n} \theta_i^4 \leq \frac{B_{9, \rho}}{n}.
\end{align}

We continue by estimating the modulus of the coefficient $r(z)$ in a suitably chosen subset of $\mathbb{C}^+$. Define 
\begin{align*} 
\revtwo	B_1 := \left\{ z \in \mathbb{C}^+: \vert \Re z \vert \leq 2- \varepsilon,  3 \geq \Im z \geq  a_n = \frac{1}{\sqrt{n \log n}}\right\} \subset S_{I_\varepsilon'}(0,4). 
\end{align*}
\revtwo 
Recall that we set $r_{n,i}(z) := G_i(Z_i(z))Z_i(z) - 1 = G_\theta(z)Z_i(z) -1$ for all $z \in \mathbb{C}^+$ and $i \in [n]$. Repeating the arguments leading to \eqref{bound $r_n,i$ mit Cauchy und integration by parts} and using \eqref{upper bound G_theta Delta tilde unbounded}, we get
\begin{align} \label{bound $r_{n,i}$ in Kolmogorov epsilon}
	\begin{split}
	\vert r_{n,i} (z) \vert\leq \frac{\vert \theta_i \vert}{\sqrt{\Im z}} \sqrt{\vert G_\theta(z) \vert} \leq A_\rho \sqrt{\frac{\log n}{n}} \frac{1}{\sqrt{\Im z}} \sqrt{1 + Z_\varepsilon b_\varepsilon} \leq A_\rho\sqrt{1 + Z_\varepsilon b_\varepsilon} \frac{(\log n)^{\nicefrac{3}{4}}}{n^{\nicefrac{1}{4}}} < \frac{1}{10}
		\end{split}
\end{align}
for $z \in B_1$, $i \in [n]$, and all sufficiently large $n$, say $n \geq n_6 \geq n_5$. This yields
\begin{align} \label{lower bound Z_i Delta varepsilon} 
	\left \vert \frac{1}{Z_i(z)} \right \vert = \left \vert  \frac{G_\theta(z)}{1 + r_{n,i}(z)}\right \vert \leq \frac{10}{9} \vert G_\theta(z) \vert \leq \frac{10}{9} (1+ Z_\varepsilon b_\varepsilon), \qquad  z \in B_1, i \in [n], n \geq n_6. 
\end{align}
For later reference, set $C_\varepsilon := \frac{10}{9} (1+ Z_\varepsilon b_\varepsilon) >1.$ Moreover, arguing as in \cref{part 2 - 3rd order functional eq}, it follows
\begin{align}  \label{estimate Z_i/Z_1 - 1, Z_i^2/Z_1^2 -1 by constant new}
	\left \vert \frac{Z_1(z)}{Z_i(z)} - 1 \right \vert \leq \frac{10}{9} \left( \vert r_{n,1}(z) \vert + \vert r_{n,i}(z) \vert \right) < \frac{2}{9}
\end{align}
for  $z, i$, and $n$ as above. Combining \eqref{subordination times cauchy expansion} with the inequality $\vert \theta_1 \vert^k \leq \vert \theta_i \vert^k$ holding for all $i \geq 2 $ and $k \in \mathbb{N}$, we obtain
\begin{align} \label{estimate Z_1/Z_i - 1 new}
	\begin{split}
		&	\left \vert  \frac{Z_1(z)}{Z_i(z)} -1  \right \vert  \\ 
		&	\leq \frac{10}{9} \left( \frac{ \theta_i^2}{\vert Z_i^2(z) \vert} + \frac{\vert \theta_i \vert^3 \beta_3(\mu)}{\vert Z_i^3(z) \vert} + \frac{\theta_i^4m_4(\mu)}{\vert Z_i^3(z)\vert \Im z} +\frac{ \theta_1^2}{\vert Z_1^2(z) \vert} + \frac{\vert \theta_1 \vert^3 \beta_3(\mu)}{\vert Z_1^3(z) \vert} + \frac{\theta_1^4m_4(\mu)}{\vert Z_1^3(z)\vert \Im z}\right) \\ &\leq \frac{10}{9} \left(\frac{\theta_i^2}{\vert Z_1^2(z)\vert}\left(  \left \vert \frac{Z_1^2(z)}{Z_i^2(z)} \right \vert \!+1 \! \right) +  \frac{ \vert \theta_i \vert^3\beta_3(\mu)}{\vert Z_1^3(z) \vert} \left( \left \vert \frac{Z_1^3(z)}{Z_i^3(z)} \right \vert \!+\!1 \right) +\frac{ \theta_i^4 m_4(\mu)}{\vert Z_1^3(z) \vert \Im z} \left( \left \vert \frac{Z_1^3(z)}{Z_i^3(z)} \right \vert \!+ \! 1 \right) \right)\\
		&\leq 4 \left(\frac{\theta_i^2}{\vert Z_1^2(z)\vert} + \frac{ \vert \theta_i \vert^3\beta_3(\mu)}{\vert Z_1^3(z) \vert} + \frac{ \theta_i^4 m_4(\mu)}{\vert Z_1^3(z) \vert \Im z} \right)
	\end{split}
\end{align}
as well as 
\begin{align} \label{estimate Z_1^2/Z_i^2 - 1 new}
	\begin{split}
		\left \vert \frac{Z_1^2(z)}{Z_i^2(z)} - 1 \right \vert&  \leq \frac{100}{81} \left( 2\vert r_{n,1}(z) \vert + 2\vert r_{n,i}(z) \vert + \left \vert r_{n,1}^2(z) \right\vert + \left \vert r_{n,i}^2(z) \right\vert\right) \\ 
		& \leq 8 \left(\frac{\theta_i^2}{\vert Z_1^2(z) \vert} +  \frac{\vert \theta_i \vert^3\beta_3(\mu)}{\vert Z_1^2(z) \vert \Im z}\right) + 4 \left( \frac{\theta_i^4}{\vert Z_1^4(z) \vert} + \frac{\theta_i^6\beta_3(\mu)^2}{ \vert Z_1^4(z) \vert (\Im z)^2} \right)
	\end{split}
\end{align}
for all $z \in B_1$, $i \in [n]$, and $n \geq n_6$. Let us remark that the last two inequalities will serve as replacements for \eqref{estimate $Z_1/Z_i -1$} and \eqref{estimate $Z_1^2/Z_i^2 -1$}. 

Now, we can bound $I_1(z), I_2(z), I_4(z)$, and $I_5$ appearing in $r(z)$ for any $z \in B_1.$ We start with $I_1(z)$. In analogy to \eqref{r_{n,i}(z)Z_i(z)  - ...}, we can write
\begin{align}\label{second estimate I_1}
	r_{n,i}(z)Z_i(z) - \frac{\theta_i^2(1 + r_{n,i}(z))}{Z_i(z)} - \frac{\theta_i^3 m_3(\mu)(1 + r_{n,i}(z))}{Z_i^2(z)}  = \frac{s^*_i(z)}{Z_i^2(z)}, \qquad z \in \mathbb{C}^+, i \in [n]
\end{align}
with 
\begin{align} \label{s_i^*}
	\begin{split}
		s_i^*(z) & :=    \frac{\theta_i^4 m_4(\mu) - \theta_i^4}{Z_i(z)} + \frac{\theta_i^5 m_5(\mu) - 2\theta_i^5m_3(\mu)}{Z_i^2(z)}  - \frac{\theta_i^6m_3(\mu)^2}{Z_i^3(z)}  + \frac{1}{Z_i^2(z)} \int_{\mathbb{R}} \frac{u^6}{Z_i(z) - u} \mu_i(du)  \\ & \,\,\,\, \,\,\,\,\,\,\,\,\, - \left( \frac{\theta_i^2}{Z_i^2(z)} + \frac{\theta_i^3m_3(\mu)}{Z_i^3(z)} \right) \int_{\mathbb{R}} \frac{u^4}{Z_i(z) - u} \mu_i(du).
	\end{split}
\end{align}
Note that the derivation of $s^*_i(z)$ differs from the one of $s_i(z)$ defined in \eqref{def s_i(z)} as follows: Previously, we expanded the term $r_{n,i}(z)$ in \eqref{second estimate I_1} up to the fourth order with the help of \eqref{subordination times cauchy expansion}. Now, we expand the first $r_{n,i}(z)$ in \eqref{second estimate I_1} up to the sixth order (by making use of \eqref{subordination times cauchy expansion more moments}), whereas the remaining two terms $r_{n,i}(z)$ are only expanded up to the fourth order as before. Arguing as done for \eqref{estimate I_1}, i.e.\@ with the help of \eqref{bound $r_{n,i}$ in Kolmogorov epsilon}, \eqref{first estimate I_1}, \eqref{estimate Z_i/Z_1 - 1, Z_i^2/Z_1^2 -1 by constant new}, and \eqref{second estimate I_1}, we can prove that $\vert I_1(z) \vert \leq \nicefrac{20}{9} \sum_{i=2}^n \vert s^*_i(z) \vert$
holds for all $z \in B_1$, $n \geq n_6.$ 
Combining \eqref{lower bound Z_i Delta varepsilon} with $\theta \in \mathcal{F}$, we get
\begin{align*}
	\sum_{i=1}^n \frac{\vert \theta_i \vert^k }{\vert Z_i^l(z) \vert} \leq \frac{C_\varepsilon^l B_{9, \rho}}{n^{\frac{k-2}{2}}} \leq \frac{C_\varepsilon^4B_{9,\rho}}{n}
\end{align*}
for $z \in B_1$, $n \geq n_6$, and $k,l \in \mathbb{N}$ with $1 \leq l \leq 4  \leq k \leq 9$. Similarly, we deduce
\begin{align*}
	&\sum_{i=2}^n \left \vert \frac{1}{Z_i^2(z)}\int_{\mathbb{R}} \frac{u^6}{Z_i(z) - u} \mu_i(du) \right \vert \leq  \frac{C_\varepsilon^2 B_{9, \rho}m_6(\mu)}{n^2} \frac{1}{a_n} = C_\varepsilon^2 B_{9, \rho} m_6(\mu) \frac{\sqrt{\log n}}{n^{\nicefrac{3}{2}}}, \\
	&\sum_{i=2}^n\left \vert \frac{\theta_i^2}{Z_i^2(z)} \int_{\mathbb{R}} \frac{u^4}{Z_i(z) - u} \mu_i(du) \right \vert \leq C_\varepsilon^2 B_{9, \rho} m_4(\mu) \frac{\sqrt{\log n}}{n^{\nicefrac{3}{2}}},\\
	&\sum_{i=2}^n \left \vert \frac{\theta_i^3m_3(\mu)}{Z_i^3(z)} \int_{\mathbb{R}} \frac{u^4}{Z_i(z) - u} \mu_i(du) \right \vert \leq C_\varepsilon^3 B_{9, \rho}\beta_3(\mu)m_4(\mu)  \frac{\sqrt{\log n}}{n^2}
\end{align*}
for $z$ and $n$ as above. 
Hence, we obtain
\begin{align*}
	\vert I_1(z) \vert \leq  \frac{20}{9} \sum_{i=2}^{n} \vert s_i^*(z) \vert \leq \frac{C_{1, \varepsilon} B_{9, \rho}\beta_3(\mu)^2m_4(\mu)\beta_5(\mu)m_6(\mu)}{n}, \qquad z \in B_1, n \geq n_6
\end{align*}
with $C_{1, \varepsilon}>0$ being some appropriately chosen constant depending on $\varepsilon$ via $Z_\varepsilon$ and $b_\varepsilon$. 
As can be seen with the help of \eqref{estimate Z_1/Z_i - 1 new}, the term $\vert I_2(z) \vert$ admits the estimate
\begin{align} \label{estimate I_2 higher exp}
	\begin{split}
		\vert I_2(z)\vert & \leq  \left \vert Z_1^2(z) \right\vert \sum_{i=2}^{n} \frac{\theta_i^2}{\vert Z_1(z) \vert}\left \vert  \frac{Z_1(z)}{Z_i(z)} -1 \right \vert  \leq 4 \sum_{i=1}^{n}  \left(\frac{\theta_i^4}{\vert Z_1(z)\vert} + \frac{ \vert \theta_i \vert^5\beta_3(\mu)}{\vert Z_1^2(z) \vert} + \frac{ \theta_i^6 m_4(\mu)}{\vert Z_1^2(z) \vert \Im z} \right) \\ &\leq   \frac{4C_\varepsilon B_{9, \rho}}{n} + \frac{4C_\varepsilon^2B_{9, \rho}\beta_3(\mu)}{n^{\nicefrac{3}{2}}} + 4C_\varepsilon^2B_{9,\rho}m_4(\mu)  \frac{\sqrt{\log n}}{n^{\nicefrac{3}{2}}} \leq \frac{C_{2, \varepsilon} B_{9, \rho}\beta_3(\mu)m_4(\mu)}{n}
	\end{split}
\end{align} 
for all $z \in B_1, n \geq n_6$, and a constant $C_{2,\varepsilon}>0.$
Moreover, by \eqref{estimate Z_1^2/Z_i^2 - 1 new}, it follows 
\begin{align}  \label{estimate I_4 higher exp}
	\begin{split}
		\vert I_4(z) \vert & \leq  \vert m_3(\mu) \vert \sum_{i=2}^n \left \vert \theta_i \right \vert^3 \left \vert \frac{Z_1^2(z)}{Z_i^2(z)} -1\right \vert  \\ & \leq 
		8 \beta_3(\mu) \sum_{i=1}^n  \left(\frac{\vert \theta_i \vert^5}{\vert Z_1^2(z) \vert} +  \frac{ \theta_i^6\beta_3(\mu)}{\vert Z_1^2(z) \vert\Im z}\right) + 4\beta_3(\mu) \sum_{i=1}^{n} \left( \frac{\vert \theta_i\vert^7}{\vert Z_1^4(z) \vert} + \frac{\vert \theta_i\vert^9\beta_3(\mu)^2}{ \vert Z_1^4(z) \vert (\Im z)^2} \right)  \\ & \leq  
		\frac{C_{3, \varepsilon} B_{9, \rho} \beta_3(\mu)^3}{n}
	\end{split}
\end{align}
for $z, n$ as before and $C_{3, \varepsilon}>0.$  \revvtwo  Using the fact that $\theta \in \mathcal{F}_3$ holds, we get
\begin{align} \label{I_5 inner}
	\vert I_5 \vert \leq  \left \vert m_3(\mu) \sum_{i=1}^n \theta_i^3 \right \vert  + \beta_3(\mu) \vert \theta_1\vert^3 \leq \frac{\beta_3(\mu)\left(23\log\frac{6}{\rho}\right)^{\nicefrac{3}{2}}}{n} +   \beta_3(\mu) A_\rho^3 \left(\frac{\log n}{n}\right)^{\nicefrac{3}{2}}.
\end{align} \revtwo 
We arrive at 
\begin{align*}
	\vert r(z) \vert = \vert I_1(z) + I_2(z) + I_4(z) + I_5 \vert  \leq \frac{D_{\varepsilon, \rho, \mu}}{n}, \qquad z \in B_1, n \geq n_6
\end{align*}
for a suitably chosen constant $D_{\varepsilon, \rho, \mu} \geq B_{9,\rho}$ depending on $\varepsilon$, $\rho$, and the first six absolute moments of $\mu$. 
\revvtwo

	Let us proceed by analyzing the roots of $P= P(z, \omega)$ in $\revtwo B_1$. 
	 By obvious modifications, we can apply Rouch\'{e}'s theorem as done in Section \ref{part 3 - roots 3rd oder} and obtain the following two results: First, we get that for any $z \in \revtwo B_1$ and $\smash{n \geq n_7:= \max\{n_6, \revtwo 75D_{\varepsilon,\rho, \mu}\}}$ the polynomial $P(z, \omega)$ has a root $\omega_1 = \omega_1(z)$ with 
		\begin{align*}
	\vert \omega_1(z) \vert < \revtwo \frac{12D_{\varepsilon, \rho, \mu}}{n}, \revvtwo \qquad \vert \omega_j(z) \vert \geq \revtwo \frac{12D_{\varepsilon, \rho, \mu}}{n}, \revvtwo \qquad j=2,3.
		\end{align*}
		Here, $\omega_j = \omega_j(z)$, $j=2,3,$ denote the remaining roots of $P$.
	Second, \revtwo for $n \geq n_7$, \revvtwo one deduces that $\revtwo \omega_1: B_1 \rightarrow \mathbb{C}, z \mapsto \omega_1(z)$ is continuous. In order to prove that $Z_1 \neq \omega_1$ holds in $\revtwo B_1$ for sufficiently large $n$, we argue by contradiction: Let $\smash{n \geq n_8 := \max\{n_7, 12^4D^4_{\revtwo \varepsilon, \revvtwo \rho, \mu}\}}$. Assume that there exists $z \in \revtwo B_1$ with $Z_1(z) = \omega_1(z)$. Then, we obtain 
	\begin{align*} \revtwo
		\frac{1}{\sqrt{n \log n}} = a_n \leq \Im z \leq \vert Z_1(z) \vert = \vert \omega_1(z) \vert < \frac{12D_{\varepsilon,\rho, \mu}}{n}, 
	\end{align*}
	which is a contradiction to $n \geq n_8.$
	
	Now, let us study the remaining roots $\omega_2, \omega_3$. We start by showing that $\omega_2 \neq  \omega_3$ holds in $\revtwo B_1$ for large $n$. Suppose that there exists $z \in \revtwo B_1$ with $\omega_2(z) = \omega_3(z)$. Then, we have
	\begin{align*}
		z = -\omega_1 \pm 2\sqrt{1-I_3 + \omega_1^2}.
	\end{align*}
	 The bounds on $\vert \omega_1 \vert$ and $ I_3$ imply
		\begin{align*}
			\left \vert \Re \sqrt{1-I_3 + \omega_1^2} \right \vert \geq 1 - \revtwo  \sqrt{\frac{B_{9, \rho}}{n}} - \frac{12D_{\varepsilon, \rho, \mu}}{n} \geq 1- \frac{13D_{\varepsilon, \rho, \mu}}{\sqrt{n}}  >  0, \qquad n \geq n_8.
		\end{align*}
	It follows 
		\begin{align*}
			\vert \Re z\vert \geq 2	\left \vert \Re \sqrt{1-I_3 + \omega_1^2} \right \vert -\vert \omega_1\vert \geq 2 - \revtwo \frac{38D_{\varepsilon, \rho, \mu}}{\sqrt{n}} \revvtwo > 2-\varepsilon
		\end{align*}
		for all $\smash{n \geq n_{9}:= \max\{ n_8, \revtwo (38D_{\varepsilon, \rho, \mu})^{2}\varepsilon^{-2}\}}$.
	Since this is a contradiction to our choice $z \in \revtwo B_1$, we obtain $\omega_2 \neq \omega_3$ in $\revtwo B_1$ for all $n \geq n_9$. 
As calculated in \cref{part 3 - roots 3rd oder} (\revtwo or more precisely in \eqref{w_2, w_3}\revvtwo), we get 	
		\begin{align*} 
			\omega_j = \frac{1}{2}\left( z + (-1)^{j-1}\sqrt{z^2-4+r_2(z)}  \right) - \frac{\omega_1}{2}, \qquad j=2,3,
		\end{align*}
		with $r_2(z) := 4I_3 + (2z - 3\omega_1)\omega_1$ for all $z$ and $n$ as above. It is easy to see that $r_2(z)$ admits the estimate 
		\begin{align*}
			\vert r_2(z) \vert \leq \revtwo \frac{91D_{\varepsilon, \rho, \mu}}{n} < \frac{1}{10}, \qquad z \in B_1, n \geq n_{9}.
		\end{align*}
		Observe that for all $z \in \revtwo B_1$ and $n \geq \revtwo  n_{9} \geq 91D_{\varepsilon, \rho, \mu} \varepsilon^{-1}$, we have
		\begin{align*}
			\Re \left(z^2 - 4 + r_2(z)\right) <  (2-\varepsilon)^2 -4 +\vert  r_2(z) \vert \leq -4\varepsilon + \varepsilon^2 + \revtwo \frac{91D_{\varepsilon, \rho, \mu}}{n} < - \varepsilon + \frac{91D_{\varepsilon, \rho, \mu}}{n} \leq 0.
		\end{align*} 
			Hence, it follows $z^2 - 4 + r_2(z) \in \nolinebreak \mathbb{C} \setminus [0,\infty)$ for $z$ and $n$ as before. The continuity of \revtwo $\omega_1\vert_{B_1}$ and of \revvtwo the complex square root function implies that $Z_1$ does not jump between $\omega_2$ and $\omega_3$ in a sufficiently small neighborhood of any point in $\revtwo B_1$ for $n \geq \revtwo n_{9}$, i.e.\@ we get a statement as in \eqref{no jumps}.  Since $\revtwo B_1$ is compact, we obtain either $Z_1 = \omega_2$ or $Z_1 = \omega_3$ in $\revtwo B_1$ for $n$ as above. Lastly, we can prove that $Z_1(2i) \neq \omega_2(2i)$ holds for $n \geq \revtwo n_{10}$ for some suitably chosen threshold $\revtwo n_{10} \geq n_9$ arising from a continuity argument. Finally, we get $Z_1 = \omega_3$ in $\revtwo B_1$ for all $n \geq \revtwo n_{10}.$

\subsubsection*{Step 5: Bounding the integrals in \texorpdfstring{\eqref{bound Delta_epsilon}}{(4.1)}}
	In the following, we proceed as done in \cref{part 7 - bounding integral wrt im }. 	
		Together with $S := \nicefrac{1}{G_\omega}$ and the formula for $\omega_3$ given \revtwo in the previous step, \revvtwo we obtain
		\begin{align*} 
		\frac{1}{Z_1(z)}  - \frac{1}{S(z)}  = \frac{1}{2S(z)Z_1(z)}\left( \omega_1(z) - \frac{r_2(z)}{\sqrt{z^2-4} + \sqrt{z^2-4+r_2(z)}}\right)
	\end{align*}
	for all $\revtwo z \in B_1, n \geq n_{10}$; compare to \eqref{difference 1/Z_1 - 1/S, cubic}.
	
	 We have to differentiate two cases in order to be able to handle the sum of the square roots appearing above. For this purpose, define $\smash{E := 50D_{\revtwo \varepsilon, \revvtwo \rho, \mu}(\log n)^{\nicefrac{1}{2}}n^{-\nicefrac{1}{2}} >0}$ \revtwo and let $n \geq n_{10}.$ \revvtwo
		Then, for any $z \in \revtwo B_1$ with $\Re z \geq E >0$, we have 
		\begin{align*}
			2\Re z \Im z + \Im r_2(z) \geq 2E a_n - \vert r_2(z) \vert = \revtwo  \frac{100D_{\varepsilon, \rho, \mu}}{n}  - \frac{91D_{\varepsilon, \rho, \mu}}{n} >0.
		\end{align*}
		For $\Re z \leq -E < 0$, we derive $2 \Re z \Im z + \Im r_2(z)<0$. The identities in \eqref{sign real parts} yield that the real parts of $\smash{\sqrt{z^2-4}}$ and $\smash{\sqrt{z^2- 4 +r_2(z)}}$ have the same sign for $z \in \revtwo B _1, \revvtwo \vert \Re z \vert \geq E.$ In particular, it  follows
		\begin{align*} 
			\left \vert \sqrt{z^2-4} + \sqrt{z^2-4+r_2(z)} \right \vert \geq \left\vert \sqrt{z^2-4} \right \vert, \qquad \revtwo z \in B_1,  \vert \Re z \vert \geq E, n \geq  n_{10}.
		\end{align*}
		\revtwo By increasing $n$ if necessary to $n \geq n_{11} \geq n_{10}$, we can assume that $E < \sqrt{2}$ holds. \revvtwo Hence, with the help of \eqref{square root real and im formula} and \eqref{im sq z2-4 larger 1}, one obtains
		\begin{align*}
			\left \vert \sqrt{z^2-4} + \sqrt{z^2-4+r_2(z)} \right \vert  \geq \Im \sqrt{z^2-4} \geq 1, \qquad \revtwo z \in B_1, \revvtwo \vert \Re z \vert < E, n \geq \revtwo n_{11}.
		\end{align*}
	Combining the last observations with $\vert S(z) \vert \geq 1$ holding for all $z \in \mathbb{C}^+$, \eqref{lower bound Z_i Delta varepsilon}, and \eqref{estimate z^2 -4}, we get
		\begin{align*}
			\left \vert \frac{1}{Z_1(z)} - \frac{1}{S(z)} \right \vert &\leq 
			\begin{cases}
				\revtwo \frac{C_\varepsilon}{2}  \revvtwo	\left( \vert \omega_1(z) \vert + \frac{\vert r_2(z) \vert}{\sqrt{\Im z}} \right) & z \in \revtwo B_1, \revvtwo \vert \Re z \vert \geq E \\
			\revtwo \frac{C_\varepsilon}{2}  \revvtwo \left( \vert \omega_1(z) \vert + \vert r_2(z) \vert \right) & z \in \revtwo B_1, \revvtwo \vert \Re z \vert < E
			\end{cases}
		\end{align*} 
		for all $n \geq \revtwo n_{11}.$ Integration yields
		\begin{align*}
			\frac{1}{\pi} \int_{-2+\varepsilon}^{2-\varepsilon}  \left \vert  \frac{1}{Z_1(u+i)} -\frac{1}{S(u+i)}\right \vert  du  \leq \revtwo \frac{C_\varepsilon}{2\pi} \revvtwo \int_{-2+\varepsilon}^{2-\varepsilon}	 \vert \omega_1(u+i) \vert + \vert r_2(u+i) \vert du \leq \revtwo \frac{206C_\varepsilon}{\pi} \frac{ D_{\varepsilon, \rho, \mu}}{n} 
		\end{align*}
		as well as 
		\begin{align*}
			\sup_{x \in I_\varepsilon'} \frac{2}{\pi}\int_{a_n}^1 \left \vert \frac{1}{Z_1(x+iv)} - \frac{1}{S(x+iv)} \right \vert dv &\leq \revtwo \frac{194C_\varepsilon}{\pi} \frac{D_{\varepsilon, \rho, \mu}}{n}, \qquad n \geq n_{11}.
		\end{align*}
	Using \eqref{difference G_mu_theta - G_nu, cubic} and \eqref{subordination times cauchy expansion}, we have
		\begin{align*}
			\frac{1}{\pi} \int_{-2+\varepsilon}^{2-\varepsilon}  \left \vert  G_\theta(u+i)-\frac{1}{Z_1(u+i)}\right \vert  du &= 	 \frac{1}{\pi} \int_{-2+\varepsilon}^{2-\varepsilon}  \left \vert \frac{r_{n,1}(u+i)}{Z_1(u+i)}\right \vert  du \leq \frac{4}{\pi}\theta_1^2 \leq \frac{4}{\pi} \revtwo \frac{B_{9,\rho}}{n}
		\end{align*}
			and 
		\begin{align*} 
			\sup_{x \in I_\varepsilon'} \frac{2}{\pi}\int_{a_n}^1 \left \vert G_\theta(x+iv) - \frac{1}{Z_1(x+iv)} \right \vert dv &\leq  \frac{\revtwo 2C_\varepsilon^3}{\pi}\left( \theta_1^2 + \beta_3(\mu)\vert \theta_1 \vert^3\int_{a_n}^{1} \frac{1}{v} dv \right) \\   & \revtwo \leq \frac{2C_\varepsilon^3}{\pi} \left( \frac{B_{9, \rho}}{n} + \beta_3(\mu)A_\rho B_{9, \rho} \left(\frac{\log n}{n} \right)^{\nicefrac{3}{2}} \right) \\ & \revtwo \leq \frac{2C_\varepsilon^3}{\pi}\frac{B_{9, \rho} + \beta_3(\mu)A_\rho B_{9, \rho}}{n}
		\end{align*}
		for $n \geq \revtwo n_{11}$.
		
		\subsubsection*{Step 6: Final conclusion} 
		 We summarize: Let $\varepsilon, \rho \in (0,1)$, define $\mathcal{F} \subset \mathbb{S}^{n-1}$ as in \revtwo the first step of this proof, \revvtwo and set $n_{\varepsilon, \rho, \mu} := \lceil \revtwo n_{11} \rceil$. Choose $\theta \in \mathcal{F}$ and let $n \geq n_{\varepsilon, \rho, \mu}$. 
	Combining the estimates obtained in the last \revtwo step \revvtwo with \eqref{bound Delta_epsilon} and \eqref{estimate error term I_n}, it follows
		\begin{align*}
			\Delta_\varepsilon(\mu_\theta, \omega) \leq \revtwo \frac{C_{\varepsilon, \rho, \mu}}{n}
		\end{align*}  
		for some suitably chosen constant $C_{\varepsilon, \rho, \mu}>0$ depending on \revtwo $\varepsilon$ through $Z_\varepsilon$ and $b_\varepsilon$ from Theorem \ref{Thm 2.7 Bao, Erdos, Schnelli}, on $\rho$,  and on $\mu$ via its first six absolute moments. \revvtwo\\

\revtwo 
Let us briefly comment on the moment assumptions made in \cref{main Kolmogorov epsilon unbounded} as well as on possible variations of our proof.
\begin{rem} \label{main Kolmogorov unbounded moment assumption}
\begin{enumerate} 
	\item[(a)] In our proof, the assumption that the underlying probability measure $\mu$ has finite sixth moment is necessary in order to allow for a reasonable application of the higher order expansions in \eqref{subordination times cauchy expansion more moments}. Note that the use of these expansions plays a significant role in the derivation of a bound of order $n^{-1}$ for $\Delta_\varepsilon(\mu_\theta, \omega)$.
	\item[(b)]  In the more general case that $\mu$ has finite fourth moment, the quantity  $\Delta_\varepsilon(\mu_\theta, \omega)$ is of order $(\log n)^{\nicefrac{1}{4}}n^{-\nicefrac{3}{4}}$ for all $\theta$ in some suitably chosen subset of $\mathbb{S}^{n-1}$ and all sufficiently large $n$. The proof of this statement follows closely the one given above with any application of the higher order expansions in \eqref{subordination times cauchy expansion more moments} replaced by the expansions in \eqref{subordination times cauchy expansion}.
\end{enumerate}
\end{rem}
\revvtwo

The following remark considers the special choice $\theta=(n^{-\nicefrac{1}{2}}, \dots, n^{-\nicefrac{1}{2}})$ and shows that in this case \revtwo $\Delta_\varepsilon(\mu_\theta, \omega)$ does not decay faster than $n^{-\nicefrac{1}{2}}$ in general. \revvtwo
\begin{rem} \label {counterexample binomial}
	Choose $\varepsilon \in (0,1)$ and let  $\mu$ be a probability measure on $\mathbb{R}$ with mean zero, unit variance, and \revtwo finite third absolute moment $\beta_3(\mu)$. \revvtwo Letting  $\mu_n$ denote the normalized $n$-fold convolution of $\mu$, i.e.\@ 
	\begin{align*}
\mu_n := D_{\nicefrac{1}{\sqrt{n}}}\mu  \boxplus \dots \boxplus D_{\nicefrac{1}{\sqrt{n}}} \mu,
	\end{align*}
	and using \revtwo \cite[Corollary 2.2]{Chistyakov2013}, \revvtwo we obtain
	\begin{align*}
		\Delta_\varepsilon(\mu_n, \omega) \leq \Delta(\mu_n, \omega) \leq \revtwo \frac{  c\beta_3(\mu)}{\sqrt{n}}
	\end{align*}
	for some absolute constant $c>0.$ \rev An upper bound of order $n^{-\nicefrac{1}{2}}$ is optimal for both $\Delta(\mu_n, \omega)$ and $\Delta_\varepsilon(\mu_n, \omega)$ among all probability measures satisfying the moment constraints given above. The optimality for the Kolmogorov distance has been proven by Chistyakov and Götze \cite[Proposition 2.5]{ Chistyakov2008b} by using the standard example of binomial measures. Making use of some of their calculations, we obtain the optimality for $\Delta_\varepsilon(\mu_n, \omega)$. For completeness, we sketch the idea: Let $\nu$ be a probability measure with
	\begin{align*}
	\nu\left( \left\{ \sqrt{\frac{q}{p}}\right\} \right) = p, \qquad \nu \left(\left\{ - \sqrt{\frac{p}{q}} \right\} \right) = q	
	\end{align*}
	for $0<p<1$, $q=1-p$, and $q \neq p$. As introduced above, the normalized $n$-fold convolution of $\nu$ is denoted by $\nu_n$. 
Having calculated the Cauchy transform of $\nu_n$, the Stieltjes-Perron inversion formula leads to a formula for the density of $\nu_n$ holding in some interval containing $[-2+\varepsilon, 2-\varepsilon]$ for sufficiently large $n$; compare to (6.21) and the equation after (6.22) in  \cite{Chistyakov2008b}. Using this formula as well as the mean value theorem, one can show that 
\begin{align*}
	\Delta_\varepsilon(\nu_n, \omega) \geq \big
	\vert	\nu_n((-2+\varepsilon, 0]) - \omega((-2+\varepsilon, 0]) \big\vert \geq \frac{C_{p, \varepsilon}}{\sqrt{n}}
\end{align*}
holds for sufficiently large $n$ and a constant $C_{p, \varepsilon}>0$. \revv
\end{rem}

\section{\texorpdfstring{Rate of convergence with respect to $\tilde{\Delta}$ and the special case of vanishing third moment: Proofs of Theorem \ref{main Delta tilde} and \cref{main Delta vanishing third moment}}{Proofs of Theorem 1.4 and Corollary 1.5}} \label {Section: Delta tilde} 

\revtwo In this section we \revvtwo prove \cref{main Delta tilde} and \cref{main Delta vanishing third moment}. Recall that  \cref{main Delta tilde} establishes a rate of convergence of the distribution \revtwo of a weighted sum \revvtwo to Wigner's semicircle law with respect to the quantity $\tilde{\Delta}$ defined in \eqref{def Delta tilde}, whereas \cref{main Delta vanishing third moment} provides the same rate in terms of the Kolmogorov distance in the special case that the underlying probability measure \revtwo has vanishing third moment. \revvtwo

\subsection{\texorpdfstring{Proof of Theorem \ref{main Delta tilde}}{Proof of Theorem 1.4}}
The proof is a combination of the proofs given for \revtwo \cref{main theorem,main Kolmogorov epsilon unbounded} \revvtwo and \revtwo is structured as follows: \revvtwo  
In the first step, we define the set $\mathcal{F}$. \revtwo The second step is dedicated to the analysis of the cubic functional equation for $Z_1$. \revvtwo Lastly, we bound the integral appearing in the definition of $\tilde{\Delta}(\mu_\theta, \omega)$ leading to the claimed rate of convergence. \revv

\subsubsection*{Step 1:  Construction of $\mathcal{F}$}
	Fix $\rho \in (0,1)$, $n \geq 4$. For $A_\rho$ and $B_9$ as before, we define $\tilde{B}_{9, \rho} := \nolinebreak B_9(\log \nicefrac{6}{\rho})^{\nicefrac{9}{2}}$ and set
\begin{align*}
	\mathcal{H}_1 & := \Bigg\{ \theta \in \mathbb{S}^{n-1}:   \max_{i \in [n]} \vert \theta_i \vert \leq A_\rho\sqrt{\frac{\log n}{n}} \Bigg\}, \qquad
	\mathcal{H}_2 := \left\{ \theta \in \mathbb{S}^{n-1}: \sum_{i=1}^n \vert \theta_i \vert^{9} \leq \frac{\tilde{B}_{9, \rho}}{n^{\nicefrac{7}{2}}} \right\},\\
	\mathcal{H}_3 & :=  \left \{ \theta \in \mathbb{S}^{n-1}: \left \vert \sum_{i=1}^n\theta_i^3 \right \vert \leq \frac{\left(23\log\frac{12}{\rho}\right)^{\nicefrac{3}{2}}}{n}   \right\}, \qquad \mathcal{H} := \mathcal{H}_1 \cap \mathcal{H}_2 \cap \mathcal{H}_3.
\end{align*}
Together with the results in \cref{section: concentration inequalities}, we immediately get $\sigma_{n-1}(\mathcal{H}^{\mathsf{c}}) \leq \nicefrac{\rho}{2}$. According to \cref{main theorem} applied to $\nicefrac{\rho}{2}$, there exist a set $\mathcal{G} \subset \mathbb{S}^{n-1}$ with $\sigma_{n-1}(\mathcal{G}^{\mathsf{c}}) \leq \nicefrac{\rho}{2}$ and constants $n_{\nicefrac{\rho}{2}, \mu} \in \mathbb{N}, C_{\nicefrac{\rho}{2}, \mu}>0$ such that 
\begin{align} \label{anwendung main theorem}
	\Delta(\mu_\theta, \omega) \leq \frac{C_{\nicefrac{\rho}{2}, \mu}}{\sqrt{n}}
\end{align}
holds for all $\theta \in \mathcal{G}$ and $n \geq n_{\nicefrac{\rho}{2}, \mu}$.
Letting $\mathcal{F} := \mathcal{G} \cap \mathcal{H} \subset \mathbb{S}^{n-1}$, it follows  $\sigma_{n-1}(\mathcal{F}^{\mathsf{c}}) \leq \rho$. Fix $\theta \in \mathcal{F}$ and assume that $\smash{\theta_1^2 = \min_{i \in [n]} \theta_i^2}$ holds. 
	
	\subsubsection*{Step 2: Analysis of the cubic functional equation for \texorpdfstring{$Z_1$}{Z1}} \rev Knowing that $P(z, Z_1(z)) = 0$ holds for all $z \in \mathbb{C}^+$, \revtwo it remains to modify the bounds on the coefficients $I_3$ and $r(z)$ appearing in the cubic polynomial $P$. For this purpose, we follow the approach presented in the fourth step of the proof of \cref{main Kolmogorov epsilon unbounded} including the higher order expansions in \eqref{subordination times cauchy expansion more moments}. 
	 
	 Arguing as in \eqref{improved estimate I_3}, the coefficient $I_3$ admits the estimate $\vert 	I_3  \vert \leq \nicefrac{\revtwo \tilde{B}_{9, \rho}}{n}.$ Let us turn to $r(z)$.
	Define $\smash{D_{\rho, \mu}:= 20A_\rho(\pi C_{\nicefrac{\rho}{2}, \mu})^{\nicefrac{1}{2}}}$ and set 
	\begin{align} \label{def D_1'}
		D'_1 :=\left \{z \in \mathbb{C}^+:  \Im z \geq D_{\rho, \mu} \frac{\sqrt{\log n}}{n^{\nicefrac{3}{4}}} \right\}.
	\end{align}
Combining the calculations done for \eqref{bound $r_n,i$ mit Cauchy und integration by parts} with integration by parts and \eqref{anwendung main theorem}, we obtain 
	\begin{align} \label{estimate r_{n,i} für Delta tilde}
		\begin{split}
		\vert r_{n,i}(z) \vert \leq \frac{\vert \theta_i \vert}{\sqrt{\Im z}}  \sqrt{\vert G_\theta(z) \vert} &\leq A_\rho \sqrt{\frac{\log n}{n}} \frac{1}{\sqrt{\Im z}} \sqrt{ 1 + \frac{\pi C_{\nicefrac{\rho}{2}, \mu}}{\sqrt{n}} \frac{1}{\Im z}}  \\ &\leq A_\rho \sqrt{\frac{\log n}{n}} \frac{1}{\sqrt{\Im z}} + A_\rho \sqrt{\frac{\log n}{n}} \frac{\sqrt{\pi C_{\nicefrac{\rho}{2}, \mu}}}{n^{\nicefrac{1}{4}}} \frac{1}{\Im z} \\ &\leq \frac{A_\rho}{\sqrt{D_{\rho, \mu}}} \frac{(\log n)^{\nicefrac{1}{4}}}{n^{\nicefrac{1}{8}}} + \frac{1}{20} < \frac{1}{10}
			\end{split}
	\end{align}
	for all  $z \in D_1'$, $i \in [n]$, and sufficiently large $n$, say $n \geq n_1 \geq n_{\nicefrac{\rho}{2}, \mu}$. This yields
	\begin{align} \label{lower bound Z_i Delta tilde}
		\left \vert \frac{1}{Z_i(z)} \right \vert \leq \frac{10}{9} \vert G_\theta(z) \vert \leq \frac{10}{9} \left( 1 + \frac{\pi C_{\nicefrac{\rho}{2},\mu}}{D_{\rho, \mu}} \frac{n^{\nicefrac{1}{4}}}{\sqrt{\log n}}\right) \leq D_{0, \rho, \mu} \frac{n^{\nicefrac{1}{4}}}{\sqrt{\log n}}
 	\end{align}
 	for $z,i$, and $n$ as before and some suitably chosen constant $D_{0, \rho, \mu} >1.$ Moreover, the estimates in \eqref{estimate Z_i/Z_1 - 1, Z_i^2/Z_1^2 -1 by constant new}, \eqref{estimate Z_1/Z_i - 1 new}, and \eqref{estimate Z_1^2/Z_i^2 - 1 new} remain valid in $D_1'$ for $n \geq \revtwo n_1$. 
 	
	As derived in the fourth step of the proof of \cref{main Kolmogorov epsilon unbounded}, \revvtwo we have 
	$\vert I_1(z) \vert \leq \nicefrac{20}{9} \sum_{i=2}^n \vert s_i^*(z) \vert$ for all $z \in D_1'$ and $n \geq n_1$, where $s_i^*(z)$ is defined in \eqref{s_i^*}. \revv
The inequality in \eqref{lower bound Z_i Delta tilde} implies
	\begin{align} \label{def D 1 rho mu}
		\sum_{i=1}^n \frac{\vert \theta_i \vert^k }{\vert Z_i^l(z) \vert} \leq \frac{D_{1, \rho, \mu}}{  n^{\nicefrac{3}{4}} }, \qquad \revtwo D_{1, \rho, \mu} :=\tilde{B} _{9,\rho}D_{0, \rho, \mu}^4
	\end{align}
	for  all \rev $z \in D_1', n \geq \revtwo n_1$, and $k,l \in \mathbb{N}$ with $k-l=3, 1 \leq l \leq 4  \leq k \leq 7$. Using the estimates  \revv
	\begin{align*}
		&\sum_{i=2}^n \left \vert \frac{1}{Z_i^2(z)}\int_{\mathbb{R}} \frac{u^6}{Z_i(z) - u} \mu_i(du) \right \vert  \leq \frac{D_{1, \rho, \mu}m_6(\mu)}{D_{\rho, \mu}} \frac{1}{n^{\nicefrac{3}{4}}},\\
		&\sum_{i=2}^n\left \vert \frac{\theta_i^2}{Z_i^2(z)} \int_{\mathbb{R}} \frac{u^4}{Z_i(z) - u} \mu_i(du) \right \vert \leq \frac{D_{1, \rho, \mu}m_4(\mu)}{D_{\rho, \mu}} \frac{1}{n^{\nicefrac{3}{4}}},\\
		&\sum_{i=2}^n \left \vert \frac{\theta_i^3m_3(\mu)}{Z_i^3(z)} \int_{\mathbb{R}} \frac{u^4}{Z_i(z) - u} \mu_i(du) \right \vert \leq \frac{D_{1, \rho, \mu}\beta_3(\mu)m_4(\mu)}{D_{\rho, \mu}} \frac{1}{n}
	\end{align*}
 holding for $z$ and $n$ as above, we conclude
		\begin{align*}
			\vert I_1(z) \vert  \leq \frac{20}{9}\left( 6D_{1,\rho,\mu}\beta_3(\mu)^2m_4(\mu)\beta_5(\mu) + \frac{3D_{1,\rho,\mu}\beta_3(\mu)m_4(\mu)m_6(\mu)}{D_{\rho, \mu}} \right) \frac{1}{\rev n^{\nicefrac{3}{4}}}, \,\,\,\,\,\,\,\,\, \rev  z \in D_1', n \geq \revtwo n_1.
		\end{align*}
 In order to bound the term $ \vert I_2(z) \vert$, we proceed as in \eqref{estimate I_2 higher exp} and obtain 
			\begin{align*} 
			\vert I_2(z)\vert \leq 4 \sum_{i=1}^{n}  \left(\frac{\theta_i^4}{\vert Z_1(z)\vert} + \frac{ \vert \theta_i \vert^5\beta_3(\mu)}{\vert Z_1^2(z) \vert} + \frac{ \theta_i^6 m_4(\mu)}{\vert Z_1^2(z) \vert \Im z} \right) \leq   \frac{8D_{1, \rho, \mu}\beta_3(\mu)}{n^{\nicefrac{3}{4}}} + \frac{4D_{1,\rho, \mu}m_4(\mu)}{D_{\rho, \mu }}\frac{1}{n^{\nicefrac{3}{4}}} 
		\end{align*} 
	for all $z \in D_1', n \geq \revtwo n_1.$ Similarly, as in \eqref{estimate I_4 higher exp}, it follows
		\begin{align*} 
			\vert I_4(z) \vert & \leq 8 \beta_3(\mu) \sum_{i=1}^n  \left(\frac{\vert \theta_i \vert ^5}{\vert Z_1^2(z) \vert} +  \frac{ \theta_i^6\beta_3(\mu)}{\vert Z_1^2(z) \vert \Im z}\right) + 4\beta_3(\mu) \sum_{i=1}^{n} \left( \frac{\vert \theta_i \vert^7}{\vert Z_1^4(z) \vert} + \frac{\vert \theta_i\vert^9\beta_3(\mu)^2}{ \vert Z_1^4(z) \vert (\Im z)^2} \right) \\ &\leq  
			\frac{12D_{1,\rho, \mu}\beta_3(\mu)}{n^{\nicefrac{3}{4}}} + \frac{8D_{1,\rho, \mu}\beta_3(\mu)^2}{D_{\rho, \mu }}\frac{1}{n^{\nicefrac{3}{4}}}  + \frac{4D_{1,\rho, \mu}\beta_3(\mu)^3}{D_{\rho, \mu}^2} \frac{1}{n}
		\end{align*}
		for $z$ and $n$ as before. The upper bound on $\vert I_5 \vert$ is essentially given by the one in \eqref{I_5 inner} -- just replace the term $\log(\nicefrac{6}{\rho})$ appearing on the right-hand side of \eqref{I_5 inner} by $\log(\nicefrac{12}{\rho})$. Note that this is necessary due to the definition of the set $\revtwo \mathcal{H}_3.$
Finally, we conclude
		\begin{align*}
			\vert r(z) \vert \leq \frac{D_{2,\rho, \mu}}{\rev n^{\nicefrac{3}{4}}}, \qquad z \in \rev D_1', n \geq \revtwo n_1
		\end{align*}
		for some suitably chosen constant \revtwo  $\smash{D_{2, \rho, \mu} \geq \max\{\nicefrac{D_{\rho, \mu}}{50}, \tilde{B}_{9, \rho}\}}$. \revvtwo 
		
		\rev We continue by analyzing the roots of $P$. \revv Let 
		\begin{align*}
		\rev	D_2' \revv := \left \{ z \in \mathbb{C}^+:  \vert \Re z \vert \leq 2 - \rev \varepsilon'_n, \revv \,3 \geq \Im z \geq  \revtwo 50D_{2,\rho, \mu}\frac{\sqrt{\log n}}{n^{\nicefrac{3}{4}}} \revv \right\} \subset \rev D_1' \revv, \qquad  \rev  \varepsilon'_n :=\frac{50D_{2,\rho, \mu}}{\sqrt{n}}.
		\end{align*}
		\rev By increasing $n$ to $\smash{n \geq \revtwo n_2 > \max\{n_1, 50^2D_{2, \rho, \mu}^2\}}$, we may assume that $\varepsilon'_n<1$ holds. \revv  Rouch\'{e}'s theorem implies that for any $z \in \rev D_2'$, $\rev n \geq \revtwo n_2$, there exists a root $\omega_1 = \omega_1(z)$ with 
		\begin{align*}
			\vert \omega_1(z) \vert < \frac{12D_{2,\rho, \mu}}{\rev n^{\nicefrac{3}{4}}}, \qquad  \vert \omega_j(z) \vert \geq  \frac{12D_{2,\rho, \mu}}{\rev n^{\nicefrac{3}{4}}}, \qquad j=2,3.
		\end{align*}
		Again, $\omega_j = \omega_j(z)$, $j=2,3$, denote the remaining roots of $P$. By contradiction, we obtain $Z_1(z) \neq \omega_1(z)$ for all $z \in \rev D_2'$ and $\revtwo n \geq n_2.$ \rev Now, assume that there exists some $z \in D_2'$ with $\omega_2(z) = \omega_3(z)$. Then, arguing as done in the previous proofs, we get $\smash{z = - \omega_1 \pm 2\sqrt{1-I_3+\omega_1^2}}$ leading to
			\begin{align*}
			\vert\Re z \vert \geq 2	\left\vert \Re \sqrt{1-I_3 + \omega_1^2} \right\vert - \vert \omega_1 \vert \geq 2-2\left(\sqrt{\frac{ \tilde{B}_{9, \rho}}{n}} + \frac{12D_{2,\rho, \mu}}{n^{\nicefrac{3}{4}}} \right) - \frac{12D_{2, \rho, \mu}}{n^{\nicefrac{3}{4}}} \geq 2- \frac{38D_{2, \rho, \mu}}{\sqrt{n}} > 2- \varepsilon'_n
		\end{align*}
		for $n \geq \revtwo n_2$. Since this contradicts our choice of $z \in D_2'$, we must have 
		$\omega_2 \neq \omega_3$ in $\rev  D_2'$ for $n \geq \revtwo n_2$. Note that the calculations in the argument above force us to choose $\varepsilon'_n$ to be of order \smash{$n^{-\nicefrac{1}{2}}$}, which in the end limits the rate of convergence for \smash{$\tilde{\Delta}(\mu_\theta, \omega)$} to \smash{$n^{-\nicefrac{3}{4}}$}. The roots $\omega_j, j=2,3,$ are of the form as given in \eqref{w_2, w_3}.
		The error term $r_2(z) = 4I_3 + (2z - 3\omega_1)\omega_1$ can be bounded by
		\begin{align*}
			\vert r_2(z) \vert \leq \rev \frac{94D_{2,\rho, \mu}}{n^{\nicefrac{3}{4}}} < \rev \frac{3}{10}, \qquad z \in D_2', n \geq \revtwo n_2.
		\end{align*}
		\revtwo It is easy to verify that $\omega_1: D_2' \rightarrow \mathbb{C}$, $z \mapsto \omega_1(z)$ is continuous for $n$ as before. \revvtwo Moreover, \revv for any $\smash{n \geq \revtwo n_3 := \max\{ n_2 , (12.5D_{2,\rho, \mu} + 0.47)^4\}}$ and all $z \in \rev D_2'$, we have $z^2-4+r_2(z) \in \mathbb{C}\setminus [0, \infty)$. The continuity of the complex square root function combined with the compactness of $D_2'$ implies that \revv $Z_1$ equals either $\omega_2$ or $\omega_3$ in $\rev D_2'$ for $n \geq \revtwo n_3.$ By showing that $\vert \omega_2(2i) \vert <1$ holds for sufficiently large $n$, say $\revtwo n \geq  n_4 \geq n_3$, we obtain $Z_1(2i) \neq \omega_2(2i)$ and thus $Z_1 = \omega_3$ in $\rev D_2'$ for $n \geq \revtwo n_4.$
		
		\subsubsection*{Step 3: Final conclusion} \rev It remains to bound the integral appearing in the definition of $\tilde{\Delta}(\mu_\theta, \omega);$ compare to \eqref{def Delta tilde}. For this, we proceed as in \cref{part 7 - bounding integral wrt im }. \revv Recall that we defined $S := \nicefrac{1}{G_\omega}$. Using the formula for $\omega_3$ \revtwo in \eqref{w_2, w_3}, \revvtwo we obtain
		\begin{align*}
			\left \vert \frac{1}{Z_1(z)} - \frac{1}{S(z)}  \right \vert \leq \frac{1}{2\vert S(z) \vert \vert Z_1(z) \vert} \left(  \vert \omega_1(z) \vert + \frac{ \vert r_2(z) \vert }{\big \vert \sqrt{z^2-4} + \sqrt{z^2-4+r_2(z)} \big \vert}\right), \qquad z \in \rev D_2', n \geq \revtwo n_4; 
		\end{align*}
		compare to \eqref{difference 1/Z_1 - 1/S, cubic}.
	Arguing as done several times before, we get
		\begin{align*}
			\left \vert \sqrt{z^2-4} + \sqrt{z^2-4+r_2(z)} \right \vert \geq 
			\begin{dcases}
			\left \vert 	\sqrt{z^2-4} \right \vert  & \rev z \in D_2', \vert \Re z \vert \geq 1 \revv \\
				 \hfil 1 & \rev  z \in D_2', \vert \Re z \vert < 1
			\end{dcases}
		\end{align*}
		for $n \geq  \revtwo n_4.$
		
		\revtwo Moreover, we can show that $\vert Z_1\vert > \nicefrac{1}{10}$ holds in $D_2'$ for sufficiently large $n$. More precisely, by using the upper bounds on $\vert r(z) \vert$ and $I_3$ as well as \eqref{lower bound Z_i Delta tilde}, it follows 
		\begin{align*}
			\left \vert I_3 + \frac{r(z)}{Z_1(z)} \right \vert \leq \frac{\revtwo \tilde{B}_{9, \rho}}{n} + \frac{\revtwo D_{2,\rho, \mu}D_{0, \rho, \mu}}{\sqrt{n \log n}} \leq \frac{2}{50}
		\end{align*}
		for all $z \in D_2'$ and large $n$, say $n \geq n_{5} \geq n_4$. \revvtwo Now, the assumption  $\vert Z_1(z) \vert \leq \nicefrac{1}{10}$ for some $z \in D_2'$ leads to a contradiction via 
		\begin{align*}
			\frac{1}{10} \geq \vert Z_1(z) \vert > \frac{1}{4}\vert Z_1(z)\vert \vert Z_1(z) - z\vert \geq \frac{1}{4} - \frac{1}{4} \left\vert I_3 + \frac{r(z)}{Z_1(z)} \right \vert > \frac{1}{5}, \qquad \revtwo n \geq n_5.
		\end{align*}
		
	Together with \rev $\vert Z_1 \vert > \nicefrac{1}{10}$ in $D_2'$, $\vert S \vert \geq 1$ in $\mathbb{C}^+$, and  \eqref{estimate z^2 -4}, \revv we arrive at 
		\begin{align*}
			\left \vert \frac{1}{Z_1(z)} - \frac{1}{S(z)} \right \vert &\leq 
			\begin{cases}
				5\left( \vert \omega_1(z) \vert + \frac{\vert r_2(z) \vert}{\sqrt{\Im z}} \right) & \rev z \in D_2', \vert \Re z \vert \geq 1 \revv \\
				5\left( \vert \omega_1(z) \vert + \vert r_2(z) \vert \right) & \rev z \in D_2', \vert \Re z \vert < 1
			\end{cases} 
		\end{align*} 
		for $n \geq \revtwo n_{5}.$ The equation in \eqref{difference G_mu_theta - G_nu, cubic} yields
		\begin{align*}
			\left \vert G_\theta(z) - \frac{1}{Z_1(z)}\right \vert = \left \vert  \frac{r_{n,1}(z)}{Z_1(z)}\right \vert \leq 1000\left( \theta_1^2 + \frac{ \beta_3(\mu) \vert \theta_1 \vert^3}{\Im z}\right), \qquad z \in \rev D_2', n \revtwo \geq n_{5}.
		\end{align*}
		Let 
		\begin{align} \label{choices Bai tilde Delta}
			a := \revtwo 50D_{2,\rho, \mu} \frac{\sqrt{\log n}}{n^{\nicefrac{3}{4}}}, \revvtwo \qquad \varepsilon := 2 \varepsilon'_n = \frac{100D_{2, \rho, \mu}}{\sqrt{n}}, \qquad I_\varepsilon :=  [-2+\nicefrac{\varepsilon}{2}, 2 - \nicefrac{\varepsilon}{2}].
		\end{align}
		By integration of the last two inequalities, 
		we deduce
		\begin{align*}
			\sup_{u \in I_\varepsilon} \int_{a}^1 \vert G_\theta(u+iv) - G_\omega(u+iv)  \vert dv \rev \leq \frac{1000D_{2,\rho, \mu} + \revtwo 1000\big(\tilde{B}_{9, \rho} + \beta_3(\mu) A_\rho \tilde{B}_{9, \rho}\big) \revvtwo}{n^{\nicefrac{3}{4}}}\
		\end{align*}
		for all $n \geq \revtwo n_{5}.$ In particular, it follows
		\begin{align*}
			\tilde{\Delta}(\mu_\theta, \omega) = \sup_{u \in I_\varepsilon} \int_{a}^1 \vert G_\theta(u+iv) - G_\omega(u+iv)  \vert dv + a + \varepsilon^{\nicefrac{3}{2}} \leq \revtwo D_{3, \rho, \mu}\frac{\sqrt{\log n}}{n^{\nicefrac{3}{4}}}
		\end{align*} 
		for a suitably chosen constant $D_{3, \rho, \mu} >0$ and all $n \geq n_{\rho, \mu} := \lceil \revtwo n_{5} \rceil.$ \revtwo Note that $D_{3, \rho, \mu}$ depends on $\rho$ and on $\mu$ through its first six absolute moments. \revvtwo  \\

In analogy to Remark \ref{counterexample binomial}, we briefly comment on the optimal bound on  $\tilde{\Delta}(\mu_\theta, \omega)$ in the case of the usual normalization with the vector $\theta = (n^{-\nicefrac{1}{2}}, \dots, n^{-\nicefrac{1}{2}})$. 
	\rev
	\begin{rem}	\label{counterexample binomial 2} 
		In the following, we use the notation introduced in Remark \ref{counterexample binomial}. \revtwo Combining some of the arguments in the proofs of \cite[Theorem 2.4]{Chistyakov2008a} and \cite[Theorem 2.1]{Chistyakov2013}, it follows 
		\begin{align*}
		 \tilde{\Delta}(\mu_n, \omega) \leq  \frac{c \beta_3(\mu)}{\sqrt{n}}
		\end{align*} 
	 for an absolute constant $c>0$ \revvtwo and any probability measure $\mu$ satisfying the same moment constraints as given in the above-mentioned remark. Using the binomial measure $\nu_n$ from Remark \ref{counterexample binomial}, one can prove \revtwo that a bound of order $n^{-\nicefrac{1}{2}}$ is sharp for $\tilde{\Delta}(\mu_n, \omega)$. \revvtwo Again, we just sketch the idea of the proof and refer to \cite[Proposition 2.5]{Chistyakov2008b} for the details: Having derived a formula for the Cauchy transform $G_n$ of $\nu_n$, one obtains
		\begin{align*}
	 \sup_{u \in [-2+\nicefrac{\varepsilon}{2}, 2-\nicefrac{\varepsilon}{2}]} \int_{a}^1  \vert G_n(u+iv) - G_\omega(u+iv)  \vert dv \geq	\int_{a}^{2^{-4}} \!\!\! \vert G_n(iv) - G_\omega(iv) \vert dv \geq \frac{C_p}{\sqrt{n}}
		\end{align*}
	for any $\varepsilon \in [0,1]$, $a \in (0,2^{-5})$, some constant $C_p>0$, and all sufficiently large $n$. Recalling the definition of $\smash{\tilde{\Delta}}$ given in \eqref{def Delta tilde}, this already proves the claim.
		 \revv
			\end{rem}

\subsection{Proof of \texorpdfstring{\cref{main Delta vanishing third moment}}{Corollary 1.5}}		
 Since \cref{main Delta vanishing third moment} can be derived easily from \cref{main Delta tilde} and the calculations done in Sections \ref{part 4 - 2nd order functional eq}, \ref{part 5 - roots of 2nd}, and \ref{part 8 - bounding integrals wrt re}, we just give a brief overview of the proof.

 For fixed $\rho \in (0,1)$, define $\mathcal{F}$ as in the first step of the proof of \cref{main Delta tilde}. Let $\theta \in \mathcal{F}$ and assume that $\theta_1^2 = \min_{i \in [n]} \theta_i^2$ holds. Observe that $a$ and $\varepsilon$ as given in \eqref{choices Bai tilde Delta} satisfy the conditions of Proposition \ref{Bai - Götze's Version} for sufficiently large $n$. Together with \cref{main Delta tilde}, we obtain
			\begin{align*}
				\Delta(\mu_\theta, \omega) 
				 \leq 12C_\gamma \tilde{C}_{\rho, \mu} \revtwo \frac{\sqrt{\log n}}{n^{\nicefrac{3}{4}}} \revvtwo + C_\gamma \int_{-\infty}^{\infty} \vert G_\theta(u+i) - G_\omega(u+i) \vert du
			\end{align*}
		for large $n$ with $\tilde{C}_{\rho, \mu}>0$ taken from \cref{main Delta tilde} and $C_\gamma$ as in Proposition \ref{Bai - Götze's Version}. In order to control the above integral, we use the quadratic functional equation for the subordination function $Z_1$. \revtwo According to \eqref{quadratic functional eq.} and \eqref{def Q}, we have 
		\begin{align*}
			Q(z, Z_1(z)) = Z_1^2(z) - zZ_1(z) + 1 -q(z) = 0, \qquad q(z) = M_1(z) + M_2(z) + M_3
		\end{align*}
		for all $z \in \mathbb{C}^+$. \revvtwo As can be seen by combining the arguments in \cref{part 4 - 2nd order functional eq} with \eqref{estimate r_{n,i} für Delta tilde}, the summand $M_1(z)$ appearing in $q(z)$ admits the estimate
			\begin{align*}
			\vert M_1(z) \vert \leq \frac{110}{81} \sum_{i=2}^{n} \left\vert \left( 1 - \frac{\theta_i^2}{Z_i^2(z)} \right)\int_{\mathbb{R}} \frac{u^3}{Z_i(z) - u} \mu_i(du) - \frac{\theta_i^4}{Z_i^2(z)}  \right\vert 
			\end{align*}
for sufficiently large $n$ and all $z \in D_1'$; see \eqref{def D_1'} for the definition of $D_1'$. Due to $m_3(\mu) = 0$, we have 
\begin{align*}
	\int_{\mathbb{R}} \frac{u^3}{Z_i(z) - u} \mu_i(du) = \frac{1}{Z_i(z)} \int_{\mathbb{R}} \frac{u^4}{Z_i(z) - u} \mu_i(du), \qquad z \in \mathbb{C}^+.
\end{align*}
Together with \eqref{lower bound Z_i Delta tilde}, this yields
\begin{align*}
	\vert M_1(z) \vert \leq \frac{110}{81} \frac{m_4(\mu)}{(\Im z)^2} \left( 2\sum_{i=1}^{n}  \theta_i^4 + \sum_{i=1}^{n}\frac{\theta_i^6}{\vert Z_i^2(z) \vert}\right) \leq \frac{1}{(\Im z)^2}  \frac{\revtwo 5 D_{1, \rho, \mu} m_4(\mu)}{n}
\end{align*}
for all $z \in D_1'$ and large $n$, where the constant $\revtwo D_{1, \rho, \mu}$ is defined in \eqref{def D 1 rho mu}. Together with \eqref{estimate Z_1/Z_i - 1 new}, one can prove that $\vert M_2(z) \vert$ is of order $(\Im z)^{-2}n^{-1}$ for $z$ and $n$ as before. Moreover, reasoning as in \eqref{improved estimate I_3}, the last summand $\vert M_3 \vert $ admits an upper bound of the same order for $z \in \mathbb{C}^+$ with $\Im z \leq 3.$ Hence, we get 
\begin{align*}
	\vert q(z) \vert = \vert  M_1(z) + M_2(z) + M_3 \vert \leq \frac{1}{(\Im z)^2}  \frac{C_{\rho, \mu}'}{n}, \qquad z \in D_1', \Im z \leq 3,
\end{align*}
for large $n$ and some suitably chosen constant $C_{\rho, \mu}'>0$ depending on $\mu$ via \revtwo $\beta_3(\mu)$ and $m_4(\mu)$. \revvtwo By increasing $n$, we can assume that $\vert q(u+i) \vert < \nicefrac{1}{10}$ holds for all $u \in \mathbb{R}$. Then, arguing as done in Section \ref{part 5 - roots of 2nd}, one obtains $Z_1 = \tilde{\omega}_2$ in $\mathbb{C}_1$; see \eqref{tilde omega_1 and tilde omega_2 - roots 2nd order} for a formula for $\tilde{\omega}_2$. Finally, the arguments used in \cref{part 8 - bounding integrals wrt re} imply
\begin{align*}
 \int_{-\infty}^{\infty} \vert G_\theta(u+i) - G_\omega(u+i) \vert du \leq \frac{47C_{\rho, \mu}' +  \revtwo 58(\tilde{B}_{9, \rho} + \beta_3(\mu)A_\rho \tilde{B}_{9, \rho})}{n}
\end{align*}
for sufficiently large $n$.
\revv
	
	\section{\texorpdfstring{Weighted sums and superconvergence: Proof of Theorem \ref{compact superconvergence result}}{Proof of Theorem 1.6}} \label{Section: Superconvergence}

	In this section we prove the superconvergence result stated in Theorem  \ref{compact superconvergence result}. \revtwo Recall that the underlying probability measure $\mu$ in that result is assumed to have compact support in $[-L,L]$ for some $L>0$.\revvtwo
	
	As already indicated in the introduction, the proof follows closely Kargin's proof of \cite[Theorem 1]{Kargin2007}. The main idea is as follows: We start by analyzing the $K$-transform $K_\theta$ of $\mu_\theta$ and prove that it is invertible. We will observe that the domain of the inverse contains a set of the form $\{ x \in \mathbb{R}: \vert x \vert > M \}$ for some $M>0.$ Then, the Cauchy transform $G_\theta$ of $\mu_\theta$ can be continued analytically to the above-mentioned set via the inverse of $K_\theta.$ A simple application of the Stieltjes-Perron inversion formula will imply the claim. 
	
\revtwo In the following, let $K_i$ denote \revvtwo  the $K$-transform of $\revtwo \mu_i = D_{\theta_i} \mu$ for $i \in [n]$. We define 
	\begin{align*}
		\varphi_\theta(z) := \sum_{i=1}^n K_i(z) - \frac{n-1}{z}  
	\end{align*}
	for any $\theta \in \mathbb{S}^{n-1}$ and all $z \in \mathbb{C}$ at which the sum exists. Note that we have $K_\theta(z) = \varphi_{\theta}(z)$ for all sufficiently small $z \in \mathbb{C}.$
	\begin{lem} \label{Lemma 7 Kargin with derivative}
		The function $\varphi_\theta$ is meromorphic  in $\{ z \in \mathbb{C}: \vert z \vert <  (6L \max_{i \in [n]} \vert  \theta_i \vert)^{-1} \}$ with a simple pole at $0$ and Laurent series given by 
		\begin{align*}
			\varphi_\theta(z) =\frac{1}{z} + \sum_{i=1}^n  \sum_{m=1}^\infty \kappa_{m+1}(\mu_i)z^m=  \frac{1}{z} + \sum_{m=1}^\infty \kappa_{m+1}(\mu_\theta)z^m    , \qquad 0<\vert z \vert < \left(6L\max_{i \in [n]} \vert \theta_i \vert\right)^{-1}.
		\end{align*}
		Moreover, we have
		\begin{align*}
			\left \vert \varphi_\theta (z) - \frac{1}{z} - z \right \vert \leq  128L^4\vert z \vert^3 \sum_{i=1}^n\theta_i  ^4  +   \left \vert m_3(\mu)\sum_{i=1}^n \theta_i^3  \right \vert \vert z \vert^2
		\end{align*}
		as well as 
		\begin{align*}
			\left \vert 
			\varphi_\theta'(z) + \frac{1}{z^2} - 1 \right \vert \leq 384L^4\vert z \vert^2 \sum_{i=1}^n \theta_i ^4 + 2\vert z \vert  \left \vert  m_3(\mu) \sum_{i=1}^n \theta_i^3 \right \vert
		\end{align*}
		for all $0 < \vert z \vert <  (6L \max_{i \in [n]} \vert  \theta_i \vert)^{-1}$.
		\begin{proof} The first assertion follows immediately from the corresponding properties of the summands $K_i$, $ i \in [n].$ The claim concerning the Laurent series can be proven by  using the Laurent series expansion of each $K_i$, rearranging the resulting sum, and recalling that $\kappa_m(\mu_\theta) = \sum_{i=1}^n \kappa_m(\mu_i)$ holds for all $m \in \mathbb{N}.$  \revtwo Following the ideas introduced by Kargin in \cite[Lemma 7]{Kargin2007a}, we obtain 
			\begin{align} \label{estimates K_i as in Kargin}
				\begin{split}
			\left \vert K_i(z) - \frac{1}{z} - \theta_i^2z - m_3(\mu)\theta_i^3z^2 \right \vert & \leq \sum_{m=3}^{\infty}  \left \vert \kappa_{m+1}(\mu_i) z^m\right \vert  \leq \sum_{m=3}^{\infty} \frac{1}{m} (2L \vert \theta_i \vert) (4L \vert \theta_i\vert)^m \left \vert z\right \vert^m \\ & = 128L^4 \theta_i^4  \vert z \vert^3\sum_{m=3}^{\infty} \frac{(4L \vert \theta_i \vert )^{m-3}}{m} \left  \vert z\right \vert^{m-3} < 128L^4\theta_i^4\vert z \vert^3
							\end{split}
			\end{align} \revvtwo 
			for all $0 < \vert z \vert < (6L\max_{j \in [n]}\vert \theta_j \vert)^{-1}$ and all $i \in [n].$ Hence, it follows
			\begin{align*}
				\left \vert  \varphi_\theta(z) - \frac{1}{z} - z - \left(m_3(\mu)\sum_{i=1}^n \theta_i^3 \right) z^2\right \vert  \leq \sum_{i=1}^n  \left \vert K_i(z) - \frac{1}{z} - \theta_i^2z - m_3(\mu)\theta_i^3z^2\right \vert\leq  128L^4\vert z \vert^3 \sum_{i=1}^n \theta_i^4 
			\end{align*}
			for all $z$ as before. Similarly, we get
			\begin{align*}
				\left \vert K_i'(z) + \frac{1}{z^2} - \theta_i^2 - 2\theta_i^3 m_3(\mu)z \right \vert =\left\vert \sum_{m=3}^\infty m\kappa_{m+1}(\mu_i)z^{m-1} \right\vert 
				\leq 128L^4 \theta_i^4\vert z \vert^2 \sum_{m=3}^\infty \left(\frac{2}{3}\right)^{m-3} = 384 L^4 \theta_i^4\vert z \vert^2
			\end{align*}
			for all $0 < \vert z \vert < (6L\max_{j \in [n]} \vert \theta_j\vert)^{-1}$ and $i \in [n]$, which implies the last claim after summation. 
		\end{proof}
	\end{lem}
	Using the last lemma, we obtain the following result concerning the critical points of $\varphi_{\theta}$:
	
	\begin{lem} \label{no zeros for derivative}
		Consider $\theta \in \mathbb{S}^{n-1}$ with $\max_{i \in [n]} \vert \theta_i \vert < (6L)^{-1}$ and let 
		\begin{align*}
			r_\theta := 384L^4 \sum_{i=1}^n  \theta_i^4 + 3\left \vert  m_3(\mu) \sum_{i=1}^n \theta_i^3 \right \vert >0.
		\end{align*}
		Then, we have $\varphi_\theta'(z) \neq 0$ for all $z \in \mathbb{C}$ with $\vert z \vert \leq 1 - r_\theta. $
		\begin{proof}
			It suffices to consider $r_\theta <1$. On the circle  $\vert z \vert = 1- r_\theta$, we have $\vert z \vert^{-2} > 1$ as well as  $\vert z - 1 \vert \vert z + 1 \vert > r_\theta$; compare to the proof of \cite[Lemma 6]{Kargin2007}. Due to $1 - r_\theta <(6L\max_{i\in [n]} \vert \theta_i \vert)^{-1}$,	we can apply Lemma \ref{Lemma 7 Kargin with derivative} to all $z \in \mathbb{C}$ with $\vert z \vert = 1-r_\theta$ and arrive at
			\begin{align*}
				\left \vert \varphi_{\theta}'(z) + \frac{1}{z^2} - 1 \right \vert \leq 384L^4  (1-r_\theta)^2 \sum_{i=1}^n  \theta_i^4 + 2(1-r_\theta)  \left \vert  m_3(\mu) \sum_{i=1}^n \theta_i^3 \right \vert \leq r_\theta < \left\vert 1-\frac{1}{z^2}   \right \vert. 
			\end{align*}
			We know that $\varphi_{\theta}'$ and $z \mapsto 1-z^{-2}$ both are meromorphic in $\vert z \vert  < (6L\max_{i \in [n]} \vert \theta_i \vert)^{-1}$ with a pole at $0$ of order $2$. It is clear that $\smash{z \mapsto 1-z^{-2}}$ has no zeros nor poles on $\vert z \vert = 1- r_\theta$.  Moreover, the estimate 
			\begin{align*}
				\vert \varphi_{\theta}'(z) \vert \geq \left \vert  \frac{1}{z^2} -1 \right \vert - r_\theta >0, \qquad \vert z \vert = 1- r_\theta,
			\end{align*}
			shows that $\varphi_{\theta}'$ has no zero on $\vert z \vert = 1- r_\theta$.
			Rouch\'{e}'s theorem for meromorphic functions (see Theorem 9.2.3 in \cite{Hille2012}) implies the claim.
		\end{proof}
	\end{lem}
	
	In the next lemma we will invert $\varphi_{\theta}$ in a punctured real interval around $0.$ 
	\begin{lem}
		Consider $\theta \in \mathbb{S}^{n-1}$ with $\max_{i \in [n]} \vert \theta_i \vert < (6L)^{-1}$ and  $r_\theta \leq \nicefrac{1}{2}$. Then, $\varphi_{\theta}$ maps the set $I_\theta:= [-1 + r_\theta, 0) \cup (0, 1-r_\theta]$ bijectively onto a set which contains
		\begin{align*}
			J_\theta:= (-\infty, -2-  2r_\theta] \cup [2 + 2r_\theta, \infty). 
		\end{align*}
		In particular, $\varphi_\theta$ has a differentiable inverse defined in $J_\theta$ and $\varphi_\theta^{-1}(J_\theta) \cup \{0\}$ is a closed interval.
		\begin{proof}
			By assumption, $\varphi_\theta$ is holomorphic in $ I_\theta$ and the Laurent series expansion holds in that set. Since the coefficients in the Laurent series of $\varphi_\theta$ are real, the restriction of $\varphi_\theta$ to $I_\theta$ is a real-valued function. By continuity of $\varphi_{\theta}'$ in $I_\theta$, we know that $\varphi_{\theta}'((0, 1-r_\theta]) \subset \mathbb{R}$ is connected. Thus, together with Lemma \ref{no zeros for derivative}, we must have either $\varphi_{\theta}' > 0$ or $\varphi_{\theta}' <0$ in $(0,1-r_\theta]$. Arguing analogously for the interval $[-1+r_\theta, 0)$, we derive that $\varphi_{\theta}$ is strictly monotone (and thus injective) in the intervals $(0, 1-r_\theta]$ and $[-1+r_\theta, 0),$ respectively. Lemma \ref{Lemma 7 Kargin with derivative} yields 
			\begin{align} \label{estimate varphi_theta - reciprocal}
				\left \vert \varphi_{\theta}(x) - \frac{1}{x} \right \vert < 128L^4\sum_{i=1}^n \theta_i^4 + 1-r_\theta + \left \vert m_3(\mu) \sum_{i=1}^n \theta_i^3 \right \vert = 1-r_\theta + \frac{r_\theta}{3} < \infty, \qquad \vert x \vert \leq 1-r_\theta.
			\end{align}  Thus, we obtain $\varphi_{\theta}(0\pm) = \pm \infty$ and conclude
			\begin{align*}
				\varphi_\theta(I_\theta) = (-\infty, \varphi_\theta(-1 + r_\theta) ] \cup [\varphi_\theta(1-r_\theta), \infty). 
			\end{align*}
			Note that $\varphi_{\theta}(0\pm) = \pm \infty$  also implies that $\varphi_{\theta}$ is strictly decreasing in the intervals $[-1+r_\theta,0)$ and $(0, 1-r_\theta]$. 
			Together with \eqref{estimate varphi_theta - reciprocal} and $(1-r_\theta)^{-1} \leq 1 + 2r_\theta$, we obtain
			\begin{align*}
				\varphi_\theta(1 - r_\theta) \leq \frac{1}{1-r_\theta} + \left \vert \varphi_\theta(1-r_\theta) - \frac{1}{1-r_\theta}\right \vert \leq 2 + r_\theta + \frac{r_\theta}{3} < 2 + 2r_\theta
			\end{align*}
			as well as 
			\begin{align*}
				\varphi_{\theta}(-1+r_\theta)  \geq  \frac{1}{-1+r_\theta} - \left \vert \varphi_\theta(-1+r_\theta) - \frac{1}{-1+r_\theta} \right \vert \geq -2 -r_\theta - \frac{r_\theta}{3} > -2 -2r_\theta.
			\end{align*}
			The last claim of the lemma can be proven as follows: We know that $\varphi_\theta^{-1}$ is strictly decreasing in $[2+2r_\theta, \infty)$ and $(-\infty, -2-2r_\theta]$.  Hence, we have
			\begin{align*}
				\varphi_\theta^{-1}(J_\theta) \cup \{0\} = [\varphi_\theta^{-1}(-2-2r_\theta), \varphi_\theta^{-1}(2+2r_\theta)].
			\end{align*}
		
		\end{proof}
	\end{lem}
	
	Since some of the following lemmas can be proven exactly as Lemma 8, Lemma 9, and Lemma 1 in \cite{Kargin2007}, we just state the results without providing (detailed) proofs. We start with an easy application of the inverse function theorem for holomorphic functions, which allows to invert $\varphi_{\theta}$ not only in $I_\theta$, but in complex neighborhoods of $I_\theta \cup \{0\}$.
	\begin{lem}  \label{Ktheta is biholomorphic}
		Consider $\theta \in \mathbb{S}^{n-1}$ with $\max_{i \in [n]} \vert \theta_i \vert < (6L)^{-1}$ and  $r_\theta \leq \nicefrac{1}{2}$. Let $z \in [-1 + r_\theta, 1 - r_\theta]$. Then, there exists a neighborhood $U_z$ of $z$ and a neighborhood $W_\omega$ of $\omega = \varphi_\theta(z)$ such that $\varphi_\theta: U_z \rightarrow W_\omega$ is bijective and the inverse $\varphi_\theta^{-1}: W_\omega \rightarrow U_z$ is holomorphic. 
		\begin{proof}
			See \cite[Lemma 8]{Kargin2007}.
		\end{proof}
	\end{lem}
	
	The next lemma shows that the local inverses established in the last lemma give rise to an analytic function defined (globally) in a neighborhood of $J_\theta \cup \{\infty\}$ acting as the inverse of $\varphi_{\theta}$ in $I_\theta \cup \{0\}$.
	\begin{lem} \label{global inverse}
		Consider $\theta \in \mathbb{S}^{n-1}$ with $\max_{i \in [n]} \vert \theta_i \vert < (6L)^{-1}$ and  $r_\theta \leq \nicefrac{1}{2}$. By analytic continuation, we can define a (unique) function $\varphi_{\theta}^{-1}$ which is holomorphic in a complex open neighborhood $S$ of $J_\theta \cup \{ \infty\}$ and is the inverse of $\varphi_{\theta}$ in a complex open neighborhood $R$ of $I_\theta \cup \{0\}$, i.e.\@ we have  $\varphi_{\theta}^{-1}(\varphi_{\theta}(z)) = z$ for all $z \in R$ and $\varphi_\theta(\varphi_{\theta}^{-1}(z)) = z$
		for all $z \in S$.
		\begin{proof}
			See \cite[Lemma 9]{Kargin2007}.
		\end{proof}
	\end{lem}
	
	In the next lemma we explain how the analytic continuation  $\varphi_{\theta}^{-1}$  is related to the Cauchy transform $G_\theta$ of $\mu_\theta.$
	\begin{lem}
		Consider $\theta \in \mathbb{S}^{n-1}$ with $\max_{i \in [n]} \vert \theta_i \vert < (6L)^{-1}$ and  $r_\theta \leq \nicefrac{1}{2}$. The inverse $\varphi_\theta^{-1}$ from Lemma \ref{global inverse} is the unique analytic continuation of $G_\theta$ to $S$.
		\begin{proof}
			We firstly prove that $\varphi_\theta^{-1}$ equals $G_\theta$ in a neighborhood of $\infty$. Let  $L_\theta>0$ be chosen in such a way that $\supp \mu_\theta \subset [-L_\theta, L_\theta]$ holds. Define  $M_\theta := \max\{ L_\theta, L \max_{i \in [n]} \vert \theta_i \vert\} \geq 1.$ We have  $\varphi_\theta = K_\theta$  as well as $G_\theta(K_\theta(z)) = z$ in $U:=\{z \in \mathbb{C}: 0 <\vert z \vert < (6M_\theta)^{-1} \}$. Moreover, with $R$ being chosen as in the last lemma, it is easy to see that $U \cap R \neq \emptyset$ holds true.  Now, define $V := K_\theta(U \cap R) = \varphi_\theta(U  \cap R) \subset \{z \in \mathbb{C}: \vert z\vert > 4L_\theta\}$, $V \neq \emptyset$. Since $K_\theta$ is non-constant in $U$ (which follows from Lemma \ref{no zeros for derivative}), we derive that $V$ is an open neighborhood of $\infty$ by the open mapping theorem. Finally, by using the inversion properties shown in Lemma \ref{global inverse}, we obtain $\varphi_\theta^{-1} = G_\theta$ in $V$. Now, recall that if two analytic functions agree on some open set, then they are analytic continuations of each other. Applying this to our situation, we obtain: $G_\theta$ can be continued analytically to $S$ via $\varphi_{\theta}^{-1}$. The proof of Lemma \ref{global inverse} shows that $S$ is simply connected in the extended complex plane, which implies uniqueness of the analytic continuation. 
		\end{proof}
	\end{lem}
	
	Note that the analytic continuation of $G_\theta$ to $S$ is holomorphic in each $z \in \mathbb{R}$ with $\vert z \vert > 2 + 2r_\theta.$ With that knowledge, the next lemma will end the proof of Theorem \ref{compact superconvergence result}. The proof of the lemma heavily relies on the Stieltjes-Perron inversion formula.
	
	\begin{lem}
		Let $G_\nu$ be the Cauchy transform of a compactly supported probability measure $\nu$. Assume that $G_\nu$ can be continued analytically to the set $\{ z \in \mathbb{R}: \vert z\vert >M\}$ for some $M>0.$ Then, we have $\supp \nu \subset [-M, M]$.
		\begin{proof} We briefly sketch the proof. The details can be carried out as in the proof of \cite[Lemma 1]{Kargin2007}. 
			In the following, we will denote the analytic continuation of the Cauchy transform of $\nu$ simply by $G_\nu$. Then, $G_\nu$ is holomorphic in $\{z \in \mathbb{C}: \vert z \vert >M\}$, including at the point at infinity. In particular, it follows $\lim_{\varepsilon \downarrow 0} \Im G_\nu(z+i\varepsilon) =0$ for any $z \in \mathbb{R}$ with $\vert z\vert > M.$ Applying the Stieltjes-Perron inversion formula, we obtain $\nu([a,b]) =0$, whenever $\nu(\{a\}) = 0 = \nu(\{b\})$ and $[a, b] \subset \{ x \in \mathbb{R}: \vert x \vert > M\}$ hold. The set of all $\nu-$atoms is at most countable. Using this, we can easily prove that $\{ x \in \mathbb{R}: \nu(\{ x\}) = 0\}$ is dense in $\{ x \in \mathbb{R}: \vert x \vert > M \}$. Hence, we can cover $\{ x \in \mathbb{R}: \vert x \vert > M \}$ with at most countably many disjoint intervals with endpoints in $\{ x \in \mathbb{R}: \nu(\{ x\}) = 0\}$, which implies the claim. 
		\end{proof}
	\end{lem} 
	
	We immediately get the following result:
	\begin{kor} \label{corollary support verbesserung}
		Let $\theta \in \mathbb{S}^{n-1}$ with $\max_{i \in [n]} \vert \theta_i \vert < (6L)^{-1}$ and  $r_\theta \leq \nicefrac{1}{2}$ as before. Then,  we have $\supp \mu_{\theta} \subset \left[ -2 -2r_\theta, 2 + 2r_\theta\right]$.
	\end{kor}
	
	Now, we are able to prove Theorem \ref{compact superconvergence result}:
	\begin{proof}[Proof of Theorem \ref{compact superconvergence result}]
		Choose $\rho \in (0,1)$ and let 
		\begin{align*}
			\mathcal{F}_1 &:= \left\{ \theta \in \mathbb{S}^{n-1}: \max_{i \in [n]} \vert \theta_i \vert \leq A_\rho\sqrt{\frac{\log n}{n}} \right\}, \qquad 
			\mathcal{F}_2 := \left\{ \theta \in \mathbb{S}^{n-1}: \sum_{i=1}^{n} \theta_i^4 \leq  \frac{\left(11\log\frac{3}{\rho}\right)^2}{n} \right\} \\ 
			\mathcal{F}_3&:= \left \{ \theta \in \mathbb{S}^{n-1}: \left \vert \sum_{i=1}^n\theta_i^3 \right \vert \leq \frac{\left(23\log\frac{6}{\rho}\right)^{\nicefrac{3}{2}}}{n}   \right\}, \qquad \mathcal{F}:= \mathcal{F}_1 \cap \mathcal{F}_2 \cap \mathcal{F}_3
		\end{align*}
		with $A_\rho$ defined as in Section \ref{part 1 - construction of F}. We obtain $\sigma_{n-1}(\mathcal{F}) \geq 1- \rho$ for $n \geq 4$. Now, fix $\theta \in \mathcal{F}$. Evidently, we can estimate $\max_{i \in [n]} \vert \theta_i\vert \leq (6L)^{-1}$  for all $n \geq (6LA_\rho)^4$. Moreover, there exists $M_{\rho, \mu} \in \mathbb{N}$ such that
		\begin{align*}
			r_\theta = 384L^4 \sum_{i=1}^n  \theta_i ^4 + 3 \left \vert  m_3(\mu) \sum_{i=1}^n \theta_i^3 \right \vert \leq 384L^4\frac{\left(11\log \frac{3}{\rho}\right)^2}{n} + 3\left\vert m_3(\mu)\right \vert \frac{\left( 23\log \frac{6}{\rho} \right)^{\nicefrac{3}{2}}}{n}  \leq  \frac{1}{2}
		\end{align*}
		holds for all $n \geq M_{\rho, \mu}.$ Define $n_{\rho, \mu} :=\left \lceil \max\left\{(6LA_\rho)^4, M_{\rho, \mu}\right\} \right\rceil$.
		Letting 
		\begin{align*}
			C_\rho := 2\max\left\{ \left(11\log \frac{3}{\rho}\right)^2, \left( 23\log \frac{6}{\rho} \right)^{\nicefrac{3}{2}} \right\},
		\end{align*}
		the claim follows from Corollary \ref{corollary support verbesserung}. 
	\end{proof}

\section{\texorpdfstring{Berry-Esseen type estimate in the non-id free CLT: Proof of Theorem \ref{Berry esseen bounded non id}}{Proof of Theorem 1.7}} \label{sec: Berry esseen bounded non id}

In this section we leave the setting of weighted sums and analyze sums of \revtwo free not necessarily identically distributed bounded self-adjoint \revvtwo random variables. With a few modifications of the method used in the proofs before, we can show that the distribution of such a sum \revtwo (after normalization) \revvtwo converges to Wigner's semicircle law with a rate of order $L_n$ measured with respect to the Kolmogorov distance.

However, before we are able to end up with a rate of order $L_n$, we need to  establish the following preliminary result:  

\begin{prop} \label{preliminary rate with power 2/3}
Under the conditions of Theorem \ref{Berry esseen bounded non id}, we have 
$\Delta(\mu_{\boxplus n}, \omega) < CL_n^{\nicefrac{2}{3}}$ for some constant $C>0.$
\end{prop}

\begin{proof} [Proof of Proposition \ref{preliminary rate with power 2/3}]
\revtwo  Let us outline the structure of the proof: In the first step, we use the machinery of $K$-transforms in order to derive a lower bound for the modulus of the subordination functions belonging to the convolution $\mu_{\boxplus n}$. \revtwo In comparison to the proofs given for weighted sums, this step is new. Moreover, it is responsible for the additional assumption of compact supports.  \revvtwo In the second and third step, we analyze a cubic and \revtwo a \revvtwo quadratic functional equation for the \revtwo subordination \revvtwo functions. An application of Proposition \ref{Bai - Götze's Version} will end the proof in the fourth step. \revvtwo

Before we begin with the first step, we fix some notation: Let $\mu_i$ denote the distribution of the normalized random variable $\nicefrac{X_i}{B_n}$, i.e.\@ we have $\mu_i = D_{B_n^{-1}} \nu_i$. Note that $\mu_{\boxplus n} = \mu_1 \boxplus \dots \boxplus \mu_n$ holds.  The subordination functions with respect to $\mu_{\boxplus n}$ will be denoted by $Z_1, \dots, Z_n$. Let $G_1, \dots, G_n, G_{\boxplus n}$ denote the Cauchy transforms of $\mu_1, \dots, \mu_n, \mu_{\boxplus n}$ and define $F_i := \nicefrac{1}{G_i}$ for $i \in [n].$ \rev Lastly, the $K$-transform of $\mu_i$ is given by $K_i$. \revv In the following, we may assume that $\sigma_1^2 = \min_{i \in [n]} \sigma_i^2$ holds. 

\subsubsection*{\rev Step 1: Lower bound for \texorpdfstring{$Z_i$}{Z1}} \rev 
\revtwo We construct a suitable lower bound for the modulus of the subordination functions $Z_i$ as follows: We start by deriving an upper bound on the absolute value of $G_{\boxplus n}(z)$ holding for certain $z \in \mathbb{C}^+$. From this, we conclude that $G_{\boxplus n}(z)$ lies in the domain of analyticity of the $K$-transform $K_i$ for any $i \in [n]$. Making use of the fact that the $K$-transform is the functional inverse of the Cauchy transform and combining this with the properties of subordination, we will be able to establish the required lower bound for the absolute value of $Z_i(z)$ for certain $z \in \mathbb{C}^+$; compare to \eqref{lower bound Z_i - bounded non id} below. \revvtwo

From \cite[Theorem 2.6]{Chistyakov2008a}, we know that $\Delta(\mu_{\boxplus n}, \omega) \leq c \sqrt{L_n}$ holds for some suitably chosen constant $c>1.$  Suppose that 
\begin{align*}
L_n < \min\big\{ 2^{-12}, 0.06(c\pi)^{-\nicefrac{3}{2}}, 5(c \pi)^{-3}\big\}
\end{align*}
is satisfied. By integration by parts, we have
\begin{align*}
	\vert G_{\boxplus n} (z) \vert \leq 1 + \frac{\pi c \sqrt{L_n}}{\Im z}  < 1 + \frac{1}{3L_n^{\nicefrac{1}{6}}} \leq  \frac{2}{3} \frac{1}{L_n^{\nicefrac{1}{6}}}  \leq \frac{2}{3} \sqrt{\frac{B_n}{\max_{i \in [n]} T_i}} < \frac{1}{6}\frac{B_n}{\max_{i \in [n]} T_i}
\end{align*}
for all $z \in \mathbb{C}^+$ with $\Im z >	3c \pi L_n^{\nicefrac{2}{3}}$. 

\revtwo As explained in \cref{section: basics in FPT}, $K_i$ is analytic in the punctured disk $\{ z \in \mathbb{C} : 0 < \vert z \vert < B_n(6T_i)^{-1}\}$, $i \in [n]$. In particular, using $G_{\boxplus n} \neq 0$ in $\mathbb{C}^+$, we conclude that $G_{\boxplus n}(z)$ lies in the domain of analyticity of $K_i$ for all $z \in \mathbb{C}^+$ with $\smash{\Im z >	3c \pi L_n^{\nicefrac{2}{3}}}$ and all $i \in [n]$.

\revvtwo
Now, set 
\begin{align*}
 D_1 :=\left\{ z \in \mathbb{C}^+: \Im z > 3c \pi L_n^{\nicefrac{2}{3}} \right\}, \quad  E_n := \left\{ z \in \mathbb{C}^+: \Im z \revtwo > \revvtwo 7L_n^{\nicefrac{1}{3}}  \right\}. 
\end{align*}
\revtwo Let us prove that  $K_i(G_{\boxplus n}(z)) = Z_i(z)$ holds for all $z \in D_1$, $i \in [n]$. For fixed $i \in [n]$, define $f_i: D_1 \rightarrow \mathbb{C}$ by $f_i(z) := K_i(G_{\boxplus n}(z)) - Z_i(z)$. According to the observations made above, $f_i$ is holomorphic in $D_1$. Note that 
\begin{align*}
\vert Z_i(z) \vert \geq \Im z > 7L_n^{\nicefrac{1}{3}} \geq 7B_n^{-1}\max_{j \in [n]}T_j\geq 7B_n^{-1}T_i
\end{align*}
holds for any $z \in E_n$. Using the inversion relation between $K_i$ and $G_i$  as stated in \cref{section: basics in FPT}, we get $K_i(G_{i}(Z_i(z))) = Z_i(z)$ for all $z \in E_n$. \cref{subordination functions} implies $G_{\boxplus n}(z) = G_i(Z_i(z))$ for all $z \in \mathbb{C}^+$. Clearly, this leads to $f_i = 0$ in $E_n.$ Together with the identity theorem for holomorphic functions, we obtain $f_i = 0$ in $D_1$. Here, the identity theorem is applicable since $\smash{D_1}$ is open and connected with $\smash{E_n \subset D_1}$ and since $\smash{E_n}$ contains an accumulation point.

Analogously to \eqref{estimates K_i as in Kargin}, i.e.\@ with the help of the Laurent series expansion of $K_i,$ we derive \revvtwo 
\begin{align*}
	\left \vert K_i(z) - \frac{1}{z} - \frac{\sigma_i^2}{B_n^2}z \right \vert < 32 \vert z \vert^2\frac{T_i^3}{B_n^3}
\end{align*}
for all $z \in \mathbb{C}^+$ with $\vert z \vert < B_n(6 \max_{j \in [n]} T_j)^{-1}$ and $i \in [n]$ leading to 
\begin{align*}
	\left \vert K_i(G_{\boxplus n}(z)) - \frac{1}{G_{\boxplus n}(z)} - \frac{\sigma_i^2}{B_n^2}G_{\boxplus n}(z)\right \vert < \frac{15\left(\max_{j \in [n]} T_j\right)^2}{B_n^2} \leq 15L_n^{\nicefrac{2}{3}}, \qquad z \in D_1, i \in [n].
\end{align*}
\revtwo Using the inequality 
$\smash{\sigma_i^2B_n^{-2} \leq L_n^{\nicefrac{2}{3}}}$, we conclude \revvtwo
\begin{align} \label{lower bound Z_i - bounded non id}
	\begin{split}
		\vert Z_i(z) \vert = \vert K_i (G_{\boxplus n} (z)) \vert & \revtwo\geq  \left \vert \frac{1}{G_{\boxplus n} (z)} \right \vert - \left \vert   \frac{\sigma_i^2}{B_n^2}G_{\boxplus n}(z) \right \vert -  \left \vert K_i(G_{\boxplus n} (z)) - \frac{1}{G_{\boxplus n}(z)} - \frac{\sigma_i^2}{B_n^2}G_{\boxplus n}(z)\right \vert  \\&\revtwo  \geq \frac{3}{2}L_n^{\nicefrac{1}{6}} - \frac{2}{3}L_n^{\nicefrac{1}{2}} - 15L_n^{\nicefrac{2}{3}} > L_n^{\nicefrac{1}{6}}  > 4L_n^{\nicefrac{1}{3}} \geq \frac{4T_i}{B_n}
	\end{split}
\end{align}
for all $z \in D_1$ and \revtwo $i \in [n].$ \revvtwo

\subsubsection*{\rev Step 2: Analysis of the  cubic functional equation for \texorpdfstring{$Z_1$}{Z1}}  

As in the case of weighted sums, compare to \cref{part 2 - 3rd order functional eq}, we need to  prove that the subordination function $Z_1$ satisfies a cubic functional equation. 
According to Theorem \ref{subordination functions}, we have
\begin{align*} 
	Z_1(z) - z + \frac{1}{Z_1(z)} = \frac{I_1(z) + I_2(z) + I_4(z) + I_5}{Z_1^2(z)} + \frac{I_3}{Z_1(z)}
\end{align*}
for any $z \in \mathbb{C}^+$ with 
\begin{align*}
	&	I_1(z) := Z_1^2(z) \left( \sum_{i=2}^n F_i(Z_i(z)) - Z_i(z) + \frac{m_2(\mu_i)}{Z_i(z)}  + \frac{m_3(\mu_i)}{Z_i^2(z)} \right), \,\,\, I_2(z) := Z_1^2(z) \left(\sum_{i=2}^n \frac{m_2(\mu_i)}{Z_1(z)} - \frac{m_2(\mu_i)}{Z_i(z)} \right), \\
	& \qquad \qquad \quad	I_3 := m_2(\mu_1), \qquad
	I_4(z) :=  Z_1^2(z) \left(\sum_{i=2}^n \frac{m_3(\mu_i)}{Z_1^2(z)} - \frac{m_3(\mu_i)}{Z_i^2(z)} \right), \qquad 
	I_5 := - \sum_{i=2}^n m_3(\mu_i).
\end{align*}
Define $r(z):= I_1(z) + I_2(z) + I_4(z) + I_5$ as well as $P(z, \omega) := \omega^3 - z\omega^2 +(1 - I_3)\omega - r(z)$. 

We continue by bounding the coefficients $I_3$ and $r(z)$.
 Using the fact that $\sigma_1^2 \leq \sigma_i^2$ holds for all $i \in [n]$, we get
\begin{align*}
	\sigma_1^2B_n^4= \sigma_1^2 \left(\sum_{i=1}^{n} \sigma_i^2 \right)^2 \leq \left( \sum_{i=1}^{n} \sigma_i^{3} \right)^2  \revtwo \leq \left(\sum_{i=1}^{n} \beta_3(\nu_i) \right)^2.
\end{align*}
Clearly, this implies  $\smash{I_3 = m_2(\mu_1) = B_n^{-2}\sigma_1^2\leq L_n^2}.$ \rev Now, let us estimate $r(z)$. For this purpose, define 
$\smash{r_{n,i}(z) := Z_i(z)G_i(Z_i(z)) -1}$ for $z \in \mathbb{C}^+$ and $i \in [n]$. \revv
 Due to \eqref{lower bound Z_i - bounded non id}, we can use the power series expansion for $G_i$ and obtain 
\begin{align} \label{power series expansion r_{n,i}}
	r_{n,i}(z) = Z_i(z)G_i(Z_i(z)) - 1 = Z_i(z) \sum_{k=0}^{\infty} \frac {m_k(\mu_i)}{Z_i^{k+1}(z)} - 1 = \sum_{k=2}^{\infty} \frac{m_k(\mu_i)}{Z_i^k(z)}
\end{align}
for all $z \in D_1$ and $i \in [n]$.
This implies $	\vert r_{n,i}(z) \vert
\leq \sum_{k=2}^{\infty} 4^{-k} < \nicefrac{1}{10}$
for $z$ and $i$ as before. Note that
\begin{align*}
	\left \vert \frac{Z_1(z)}{Z_i(z)} - 1\right \vert \leq \frac{10}{9} \left( \vert r_{n,1}(z) \vert + \vert r_{n,i}(z) \vert \right) < \frac{2}{9}, \qquad z \in D_1, i \in [n]
\end{align*}
holds. We get 
\begin{align*}
	\vert r_{n,i}(z) \vert 	\leq \frac{m_2(\mu_i)}{\vert Z_i^2(z) \vert}\sum_{k=2}^\infty \left(\frac{T_i}{B_n}\right)^{k-2} \frac{1}{\big \vert Z_i^{k-2}(z) \big \vert}
	\leq \frac{4}{3} \frac{m_2(\mu_i)}{\vert Z_i^2(z) \vert} <  2\frac{m_2(\mu_i)}{ \vert Z_1^2(z) \vert}, \qquad  z \in D_1, i \in [n],
\end{align*}
which in turn yields
\begin{align} \label{estimate Z_1/Z_i -1 with more than constant - bounded non id}
	\left \vert \frac{Z_1(z)}{Z_i(z)} - 1\right \vert \leq \frac{10}{9} \left( \frac{4}{3} \frac{m_2(\mu_1)}{\vert Z_1^2(z) \vert} + 2 \frac{m_2(\mu_i)}{\vert Z_1^2(z) \vert}\right) < 4\frac{m_2(\mu_i)}{\vert Z_1^2(z) \vert} 
\end{align}
for all $z$ and $i$ as before. 
 It is clear that \revtwo the identity in \eqref{relation F_i - Id and $r_{n,i}$} remains valid \revvtwo in the setting at hand. Hence, together with \eqref{power series expansion r_{n,i}}, we arrive at
\begin{align*} 
	\left \vert  F_i(Z_i(z)) \!- \! Z_i(z) \!  + \! \frac{m_2(\mu_i)}{Z_i(z)} \!+\! \frac{m_3(\mu_i)}{Z_i^2(z)}  \right \vert & \leq  \frac{10}{9}
	\left \vert r_{n,i}(z)Z_i(z) \!-\! \frac{m_2(\mu_i)(1 + r_{n,i}(z))}{Z_i(z)} \!-\! \frac{m_3(\mu_i)(1 + r_{n,i}(z))}{Z_i^2(z)} \right \vert \\
	& \leq \frac{10}{9} \frac{\vert s_i(z) \vert}{ \vert Z_i^2(z) \vert}
\end{align*}
with 
\begin{align*} 
	s_{i}(z) := \sum_{k=4}^{\infty}  \frac{m_k(\mu_i)}{Z_i^{k-3}(z)} - \sum_{k=2}^\infty \frac{m_2(\mu_i) m_k(\mu_i)}{Z_i^{k-1}(z)} - \sum_{k=2}^{\infty} \frac{m_3(\mu_i) m_k(\mu_i)}{Z_i^k(z)}
\end{align*}
for all $z \in D_1, i \in [n].$
It follows
\begin{align*}
	\sum_{i=2}^{n}\vert s_i(z) \vert &
	\leq  \sum_{k=4}^{\infty}  \sum_{i=1}^{n} \frac{\vert m_k(\mu_i) \vert}{\vert Z_i^{k-3}(z) \vert} + \sum_{k=2}^\infty  \sum_{i=1}^{n} \frac{m_2(\mu_i) \vert m_k(\mu_i) \vert}{ \vert Z_i^{k-1}(z) \vert} + \sum_{k=2}^{\infty}  \sum_{i=1}^{n} \frac{\vert m_3(\mu_i) \vert \vert m_k(\mu_i) \vert}{ \vert Z_i^k(z) \vert} \\ 
	& \leq \sum_{k=4}^{\infty} \frac{L_n^{\nicefrac{k}{3}}}{4^{k-3}L_n^{\nicefrac{(k-3)}{3}}} +   \sum_{k=2}^{\infty}  \frac{L_n^{\nicefrac{(k+2)}{3}}}{4^{k-1}L_n^{\nicefrac{(k-1)}{3}}} + \sum_{k=2}^{\infty}  \frac{L_n^{\nicefrac{(k+3)}{3}}}{4^k L_n^{\nicefrac{k}{3}}}  \leq \frac{3}{4} L_n, \qquad z \in D_1. 
\end{align*}
We deduce
\begin{align*}
	\vert I_1(z) \vert & \leq \left \vert Z_1^2(z)  \right \vert \sum_{i=2}^n \left \vert  F_i(Z_i(z)) \!- \! Z_i(z) \!  + \! \frac{m_2(\mu_i)}{Z_i(z)} \!+\! \frac{m_3(\mu_i)}{Z_i^2(z)}  \right \vert \leq \frac{10}{9} \sum_{i=2}^n \frac{\vert Z_1^2(z) \vert}{ \vert Z_i^2(z) \vert} \vert s_i(z) \vert < \frac{5}{4}L_n, \qquad z \in D_1.
\end{align*}
Together with \eqref{estimate Z_1/Z_i -1 with more than constant - bounded non id}, we get
\begin{align*}
	\vert I_2(z) \vert & \leq \vert Z_1(z) \vert \sum_{i=2}^n m_2(\mu_i) \left \vert 1 - \frac{Z_1(z)}{Z_i(z)} \right \vert  \leq 4 \sum_{i=1}^n \frac{m_2(\mu_i)^2}{\vert Z_1(z) \vert} < L_n
\end{align*}
for all $z \in D_1.$ Similarly, using 
\begin{align*}
	\left \vert \frac{Z_1^2(z)}{Z_i^2(z)} - 1 \right \vert = \left \vert \frac{Z_1(z)}{Z_i(z)} -1 \right \vert \left \vert \frac{Z_1(z)}{Z_i(z)}  + 1 \right \vert \leq \frac{20}{9}\left \vert \frac{Z_1(z)}{Z_i(z)} -1 \right \vert, \qquad z \in D_1,
\end{align*}
we obtain $\vert I_4(z) \vert  < \nicefrac{5L_n}{9}$ 
for all $z \in D_1.$
It follows $\vert r(z) \vert <  4L_n$ for $z \in D_1$. 

\rev
We continue by analyzing the roots of $P$. Since we will mostly argue as done in \cref{part 3 - roots 3rd oder}, we omit some of the details. Define 
\begin{align*}
D_2 := \big\{ z \in \mathbb{C}^+: \vert \Re z \vert \leq 2 - \varepsilon_n, 3 \geq \Im z \geq 4\pi c L_n^{\nicefrac{2}{3}} \big\} \subset D_1, \qquad \varepsilon_n := 24\pi cL_n^{\nicefrac{2}{3}}.
\end{align*}
Let $\omega_j = \omega_j(z)$, $j=1,2,3,$ denote the roots of $P(z, \omega).$
By Rouch\'{e}'s theorem, we can prove that for any $z \in D_2$ there exists a root, say $\omega_1 = \omega_1(z)$, such that
\begin{align*}
\vert \omega_1(z) \vert < 8L_n, \qquad \vert \omega_j(z) \vert \geq 8L_n, \qquad j=2,3
\end{align*}
hold. Making use of contradiction arguments, we obtain $Z_1 \neq \omega_1$ and $\omega_2 \neq \omega_3$ in $D_2.$ The roots $\omega_2, \omega_3$ are of the form 
\begin{align*}
	\omega_j =  \frac{1}{2}\left( z + (-1)^{j-1}\sqrt{z^2-4+r_2(z)}  \right) - \frac{\omega_1}{2},  \qquad j=2,3
\end{align*}
for  $r_2(z):= 4I_3+ (2z - 3\omega_1)\omega_1$. The error term $r_2(z)$ admits the estimate 
\begin{align*}
	\vert r_2(z) \vert \leq 4L_n^2+ 16\sqrt{13}L_n + 192L_n^2 < 60L_n
\end{align*}
for all $z \in D_2.$ Note that the restricted functions $\omega_1, r_2: D_2 \rightarrow \mathbb{C}$ are continuous. The assumption $\smash{L_n < 0.06(c\pi)^{-\nicefrac{3}{2}}}$ guarantees that the complex square root function is continuous at $z^2 - 4 +r_2(z)$ for any $z \in D_2$. Combining the last two observations, we obtain the following statement:
\begin{align*} 
	\forall z_0 \in D_2: \exists r_0 = r_0(z_0) >0: \exists! j \in \{ 2,3\}: \forall z \in D_2 \cap \{ \revtwo s \revvtwo \in \mathbb{C}: \vert \revtwo s \revvtwo -z_0 \vert < r_0\}: Z_1(z) = \omega_j(z);
\end{align*}
compare to \eqref{no jumps}. A simple covering argument based on the compactness of $D_2$ shows that we have either $Z_1 = \omega_2$ or $Z_1 = \omega_3$ in $D_2$. It remains to prove $Z_1(2i) \neq \omega_2(2i)$. By continuity, we find $\delta_2>0$ such that 
\begin{align*}
\left \vert  \sqrt{-8 + r_2(2i)} - \sqrt{-8}  \right \vert \leq \frac{1}{2}
\end{align*}
is valid whenever $L_n < \delta_2$ holds (which we assume from now on). 
Together with the previously established bound on $\vert \omega_1\vert$, it follows $\vert \omega_2(2i) \vert < 1$. Due to $\vert Z_1(2i) \vert \geq 2$, we must have $Z_1 = \omega_3$ in $D_2$.
\revv 

\subsubsection*{\rev Step 3: Analysis of the quadratic functional equation for \texorpdfstring{$Z_1$}{Z1}} \rev As in \cref{part 4 - 2nd order functional eq}, the quadratic functional equation for $Z_1$ can be derived from \cref{subordination functions}. In more detail, we have \revv
\begin{align*} 
	Z_1(z) - z + \frac{1}{Z_1(z)} = \frac{M_1(z) + M_2(z) + M_3}{Z_1(z)} = \frac{q(z)}{Z_1(z)}, \qquad z \in \mathbb{C}^+,
\end{align*}
with 
\begin{align*}
	M_1(z) :=  Z_1(z) \left( \sum_{i=2}^n F_i(Z_i(z))\,\, - \right. & \left.  Z_i(z) + \frac{m_2(\mu_i)}{Z_i(z)}\right), \qquad
	M_2(z) :=  Z_1(z) \left( \sum_{i=2}^n \frac{m_2(\mu_i)}{Z_1(z)} - \frac{m_2(\mu_i)}{Z_i(z)}\right),  \\
	& \!\!\!\!\!\!\!\!\!\!\!\!\!\!  M_3 := m_2(\mu_1), \qquad q(z) := M_1(z) + M_2(z) + M_3.
\end{align*}
Define $Q(z, \omega) := \omega^2 - z \omega +1 - q(z)$ and let
\begin{align*}
 	t_i(z) := \sum_{k=3}^{\infty}  \frac{m_k(\mu_i)}{Z_i^{k-3}(z)} - \sum_{k=2}^\infty \frac{m_2(\mu_i) m_k(\mu_i)}{Z_i^{k-1}(z)} 
\end{align*}
for $z \in D_1, i \in [n].$
Copying the arguments in the estimation of $\vert I_1(z)\vert$, we obtain
\begin{align*}
	\vert M_1(z) \vert \leq \vert Z_1(z) \vert \sum_{i=2}^n \left \vert  F_i(Z_i(z)) -  Z_i(z)  +  \frac{m_2(\mu_i)}{Z_i(z)}  \right \vert  \leq  \frac{10}{9} \frac{1}{\Im z} \sum_{i=2}^n \frac{\vert Z_1(z) \vert}{ \vert Z_i(z) \vert} \vert t_i(z) \vert  < \frac{3L_n}{\Im z}, \qquad z \in D_1.
\end{align*}
Moreover, using \eqref{estimate Z_1/Z_i -1 with more than constant - bounded non id}, it follows
\begin{align*}
	\left \vert M_2(z) \right \vert & \leq  \sum_{i=2}^{n} m_2(\mu_i) \left \vert 1 - \frac{Z_1(z)}{Z_i(z)} \right \vert< \frac{4}{\Im z} \sum_{i=1}^{n} \frac{m_2(\mu_i)^2}{\vert Z_1(z) \vert} < 
	\frac{L_n}{\Im z}
\end{align*}
for all  $z \in D_1$.
Lastly, for any $z \in D_3 := \{ z \in D_1 : \Im z \leq 3\}$, we \revtwo get \revvtwo $\vert M_3 \vert \leq L_n^2 \leq \nicefrac{3L_n}{\Im z}$. In particular, we have $\vert q(z) \vert <  \nicefrac{7L_n}{\Im z} < \nicefrac{1}{10}$ for $z \in D_3$ as well as $\vert q(u+i) \vert < 7L_n < \nicefrac{1}{10}$ for all $u \in \mathbb{R}.$

Solving the equation $Q(z, \omega) = 0$ for fixed $z \in \mathbb{C}^+$, we obtain two roots $\tilde{\omega}_j, j=1,2$, of the form 
\begin{align*}  
	\tilde{\omega}_j = \tilde{\omega}_j(z) = \frac{1}{2} \left( z + (-1)^{j} \sqrt{q_1(z)} \right), \qquad q_1(z) := z^2 - 4 + 4q(z) \neq 0, \qquad z \in \mathbb{C}^+.
\end{align*}
As in Section \ref{part 5 - roots of 2nd}, we can prove that $Z_1(z) = \tilde{\omega}_2(z)$ holds for any
$z \in \mathbb{C}_1  = \{ u+i: u \in \mathbb{R}\} \subset D_3$.

\subsubsection*{\rev Step 4: Bounding the integrals in Proposition \texorpdfstring{\ref{Bai - Götze's Version}}{2.5}} 
We set 
\begin{align*}
\tau \in (1, 3], \qquad \gamma <1, \qquad a := 4 \pi c L_n^{\nicefrac{2}{3}}, \qquad  \varepsilon := 2\varepsilon_n = 48 \pi c L_n^{\nicefrac{2}{3}}.
\end{align*}
 It is easy to verify that all conditions in Proposition \ref{Bai - Götze's Version} are satisfied for the measure $\mu_{\boxplus n}$ as well as for the parameters defined above. Let $\nu$ denote the probability measure on $\mathbb{R}$ determined by the formula $G_\nu(z) = \nicefrac{1}{Z_1(z)}$ for $z \in \mathbb{C}^+$; compare to Lemma \ref{reciprocal subordination function as Cauchy transform}. Moreover, define $I_\varepsilon := [-2 + \nicefrac{\varepsilon}{2}, 2 - \nicefrac{\varepsilon}{2}]$ as well as  $S(z) := \nicefrac{1}{G_\omega(z)} = \frac{1}{2}\left( z + \sqrt{z^2-4} \right)$, $z \in \mathbb{C}^+$.
It suffices to bound the integrals 
\begin{align} \label{integrals wrt re - bounded non id}
\!\!\!\!\!	\int_{-\infty}^{\infty} \!\!\!\! \vert G_{\boxplus n}(u+i) - G_\omega(u+i) \vert du \leq\!  \int_{-\infty}^{\infty}  \!\!\! \! \vert G_{\boxplus n}(u+i) - G_\nu(u+i) \vert du + \! \int_{-\infty}^{\infty} \!\!\!\! \vert G_{\nu}(u+i) - G_\omega(u+i) \vert du 
\end{align}
and
\begin{align} \label{integrals wrt im - bounded non id}
\!\!\! \int_a^1\!\! \vert G_{\boxplus n}(u\! +\! iv) - G_\omega(u+\! iv) \vert dv \leq \! \! 	\int_a^1 \! \!  \vert G_{\boxplus n}(u\! +\! iv) - G_\nu(u\! +\! iv) \vert dv + \! 	\int_a^1 \!\! \vert G_\nu(u\! +\! iv) - G_\omega(u\! +\! iv) \vert  dv  \!
\end{align}
for all $u \in I_\varepsilon.$ 

We start with the integrals in  \eqref{integrals wrt im - bounded non id}. \rev  As done for \eqref{bounds Z_1 by 10 and 1/10}, we can prove that $\vert Z_1 \vert > \nicefrac{1}{10}$ holds in $D_2.$ Recalling that $Z_1$ is equal to $\omega_3$ in $D_2$ and copying the arguments used for \eqref{estimate difference 1/Z_1 - 1/S for big real parts} and \eqref{estimate difference 1/Z_1 - 1/S for small real parts}, we obtain \revv
\begin{align} \label{G_nu - G_omega bounded non id}
\left\vert G_\nu(z) - G_\omega(z) \right \vert = 	\left \vert \frac{1}{Z_1(z)} - \frac{1}{S(z)} \right \vert  \leq  \begin{cases}
		5 \left(  \vert \omega_1(z) \vert + \frac{\vert r_2(z) \vert}{\sqrt{\Im z}}\right) & z \in D_2, \vert \Re z \vert \geq F_n \\ 
		\hfil 5 \left( \vert \omega_1(z) \vert  + \vert r_2(z) \vert\right) & z \in D_2, \vert \Re z \vert < F_n
	\end{cases}
\end{align}
for $F_n$ given by $F_n :=8L_n^{\nicefrac{1}{3}}(\pi c)^{-1} < \sqrt{2}$.
After integration, we get
\begin{align*}
	\sup_{u \in I_\varepsilon} \int_{a}^1 \vert G_\nu(u+iv) - G_\omega(u+iv) \vert dv \leq 640L_n.
\end{align*}
In order to handle the other integral appearing on the right-hand side in \eqref{integrals wrt im - bounded non id}, observe that
\begin{align} \label{difference G_boxplus n - Gnu - Nebenrechnung vor Integration}
	\begin{split}
		\left \vert  G_{\boxplus n}(z)- G_\nu(z) \right \vert = \left \vert \frac{r_{n,1}(z)}{Z_1(z)}\right \vert 
		\leq 2\frac{m_2(\mu_1)}{\vert Z_1^3(z) \vert} 
		\leq 2000 L_n^2, \qquad z \in D_2
	\end{split}
\end{align}
holds. 
This yields
\begin{align*} 
	\begin{split}
		\sup_{u \in I_\varepsilon}\int_{a}^1 \left \vert 	G_{\boxplus n}(u+iv) - G_\nu(u+iv) \right \vert dv  \leq 2000L_n^2 < L_n.
	\end{split}
\end{align*}

Now, let us proceed with the integrals in \eqref{integrals wrt re - bounded non id}. As in Section \ref{part 8 - bounding integrals wrt re} and together with $Z_1 = \tilde{\omega}_2$ holding in $\mathbb{C}_1$, we can prove that
\begin{align}   \label{G_nu - G_omega 2 -bounded non id}
\left \vert \frac{1}{Z_1(u+i)} - \frac{1}{S(u +i)} \right \vert 
	 \leq 
	 \begin{dcases}
	 \frac{4 \vert q(u+i) \vert}{ \sqrt{1 +((\vert u \vert -4)_+)^2} \sqrt{\max\{1, (u^2 -5)_+\}}} & \vert u \vert \geq \frac{1}{5} \\[10pt]
	 \hfil 2\vert q(u+i) \vert & u \in \mathbb{R}
 \end{dcases}
\end{align} 
is valid.
Making use of \eqref{estimate denominator quadratic case 0.01(1+u^2)^2} as well as of the previously established bound on $\vert q \vert$ and integrating the last inequality by splitting the domain of integration into appropriate subdomains, we obtain
\begin{align*} 
	\int_{-\infty}^{\infty} \left \vert G_\nu(u+i) - G_\omega(u+i) \right \vert du \leq  \rev 329 \revv L_n. 
\end{align*}
It remains to bound the other integral on the right-hand side in \eqref{integrals wrt re - bounded non id}.
We have
\begin{align} \label{difference G_boxplus n - G_nu evaluated at u+i}
	\left \vert G_{\boxplus n}(u+i) - G_\nu(u+i) \right \vert = \left \vert \frac{r_{n,1}(u+i)}{Z_1(u+i)} \right \vert \leq \frac{2L_n^2}{\vert Z_1^3(u+i)\vert} \leq  \frac{2L_n^2}{\vert Z_1^2(u+i) \vert} \leq 2L_n^2, \qquad u \in \mathbb{R}.
\end{align}
Moreover, analogously to \eqref{bounds for abs Z_1 for integral wrt Re of 1/Z_1 - 1/S}, we can prove that $\vert Z_1(u+i) \vert \geq 	\nicefrac{1}{5}\left( 1 + (\vert u \vert -4)_+\right)$ holds true for all $u \in \mathbb{R}$ with $\vert u \vert \geq  \nicefrac{1}{5}$.
This implies
\begin{align*} 
	\int_{-\infty}^{\infty}  \vert G_{\boxplus n}(u+i) - G_\nu(u+i) \vert du \leq \rev 116 \revv L_n^2 < L_n.
\end{align*}
Provided that  $L_n < C_0 := \min\{ 2^{-12}, 0.06(c\pi)^{-\nicefrac{3}{2}}, 5(c \pi)^{-3}, \delta_2\}$ holds, Proposition \ref{Bai - Götze's Version} yields
\begin{align*}
	\Delta(\mu_{\boxplus n}, \omega) \leq  \rev C_\gamma \revv \left( \rev 971 \revv L_n  + \rev 144 \revv c L_n^{\nicefrac{2}{3}} + (48\pi c)^{\nicefrac{3}{2}}L_n \right) < C_1L_n^{\nicefrac{2}{3}}
\end{align*}
for \rev $C_\gamma$ as in Proposition \ref{Bai - Götze's Version} and some numerical constant $C_1>1$. \revv Letting $C:= C_1C_0^{-\nicefrac{2}{3}} > C_1$, we arrive at $\smash{\Delta(\mu_{\boxplus n}, \omega) < CL_n^{\nicefrac{2}{3}}}$ as claimed.
\end{proof}

Let us continue with the proof of Theorem \ref{Berry esseen bounded non id}. \rev Basically, we have to repeat all steps carried out in the proof of \cref{preliminary rate with power 2/3}. Since we can copy most arguments, we just give a brief overview. Moreover, we use the notation (including the definitions of the constants $C$ and $\delta_2$) introduced above. \revv 
 \begin{proof} [Proof of Theorem \ref{Berry esseen bounded non id}] 
Fix $Z >\max \{12\pi C, 70\}$ and, without loss of generality, assume that 
\begin{align*}
L_n < C_0':=  \min \left\{ \frac{5}{Z^{\nicefrac{3}{2}}}, \frac{1}{2^{12}}, \delta_2 \right\} < \frac{1}{Z}
\end{align*} 
as well as $\sigma_1^2 = \min_{i \in [n]} \sigma_i^2$ hold.

\subsubsection*{\rev Step 1: Lower bound for $Z_i$}
Combining integration by parts with Proposition \ref{preliminary rate with power 2/3}, we obtain 
\begin{align*}
	\vert G_{\boxplus n} (z) \vert 
	\leq 1+ \frac{1}{12 L_n^{\nicefrac{1}{3}}}< \frac{1}{6L_n^{\nicefrac{1}{3}}} \leq \frac{B_n}{6 \max_{i \in [n]} T_i}
\end{align*}
for all $z \in \mathbb{C}^+$ with $\Im z \geq 12\pi C L_n$. It follows
\begin{align} \label{lower bound Z_i new - bounded non id}
\begin{split}
	\vert Z_i(z) \vert &  = \vert K_i (G_{\boxplus n} (z)) \vert \geq \left \vert \frac{1}{G_{\boxplus n}(z)}\right \vert - \frac{\sigma_i^2}{B_n^2} \vert G_{\boxplus n} (z) \vert - 	\left \vert K_i(G_{\boxplus n}(z)) - \frac{1}{G_{\boxplus n}(z)} - \frac{\sigma_i^2}{B_n^2}G_{\boxplus n}(z)\right \vert \\ &\geq 6L_n^{\nicefrac{1}{3}} - \frac{1}{6}L_n^{\nicefrac{1}{3}} - L_n^{\nicefrac{1}{3}} > 4L_n^{\nicefrac{1}{3}}\geq \frac{4T_i}{B_n}
\end{split}
\end{align}
for all $z \in \mathbb{C}^+$ satisfying $\Im z > 12\pi CL_n$ and all $i \in [n]$. \rev Before we continue with the next step, let us briefly compare the inequalities in  \eqref{lower bound Z_i - bounded non id}  and \eqref{lower bound Z_i new - bounded non id}: Note that \eqref{lower bound Z_i - bounded non id} contains the stronger inequality $\smash{\vert Z_i(z) \vert > L_n^{\nicefrac{1}{6}}}$, which, however, was not used in the proof of Proposition \ref{preliminary rate with power 2/3}. Moreover, observe that the estimate in \eqref{lower bound Z_i new - bounded non id} holds for a larger subset in $\mathbb{C}^+$. More precisely, the lower bound on $\Im z$ was decreased from being of order $\smash{L_n^{\nicefrac{2}{3}}}$ to the order $L_n$. 

\subsubsection*{Step 2: Analysis of the cubic and quadratic functional equation for $Z_1$} 
Arguing as in the proof of Proposition \ref{preliminary rate with power 2/3}, we have $P(z, Z_1(z)) = 0 = Q(z, Z_1(z))$ for all $z \in \mathbb{C}^+$. It remains to bound the coefficients and analyze the roots of both polynomials.
Due to the observation made at the end of the first step, we can proceed exactly as in Proposition \ref{preliminary rate with power 2/3} as long as $L_n$ is sufficiently small -- which is guaranteed by $L_n < C_0'$ -- and we modify all steps in which we bound $\Im z$ from below. Define 
\begin{align*}
	 a := ZL_n, \qquad \varepsilon := 2ZL_n^{\nicefrac{2}{3}}, \qquad D_2 := \left\{ z \in \mathbb{C}^+: \vert \Re z \vert \leq 2- \nicefrac{\varepsilon}{2}, 3 \geq \Im z \geq a\right\}. 
\end{align*} 
We obtain
\begin{align*}
\vert r(z) \vert < 4L_n, \qquad \vert \omega_1(z) \vert < 8L_n, \qquad \vert r_2(z) \vert < 60L_n, \qquad Z_1(z) = \omega_3(z), \qquad \vert Z_1(z) \vert > \frac{1}{10}
\end{align*}
for all $z \in D_2$
as well as
\begin{align*}
\vert q(u+i) \vert <  7L_n < \frac{1}{10}, \qquad Z_1(u+i) = \tilde{\omega}_2(u+i)
\end{align*}
for all $u \in \mathbb{R}$. \revv

\subsubsection*{\rev Step 3: Bounding the integrals in Proposition \ref{Bai - Götze's Version}}
Replacing $F_n$ in Proposition \ref{preliminary rate with power 2/3} by $F_n':= 40Z^{-1}< \sqrt{2}$, the estimate in \eqref{G_nu - G_omega bounded non id} remains valid. Moreover, it is clear that the inequalities in \eqref{difference G_boxplus n - Gnu - Nebenrechnung vor Integration}, \eqref{G_nu - G_omega 2 -bounded non id}, and \eqref{difference G_boxplus n - G_nu evaluated at u+i} also hold in our new setting leading to 
\begin{align*}
	\sup_{u \in I_\varepsilon} \int_{a}^1 \vert G_{\boxplus n}(u+iv) - G_\omega(u+iv) \vert dv \leq \rev 641 \revv L_n, \qquad
	\int_{-\infty}^{\infty}  \vert G_{\boxplus n}(u+i) - G_\omega(u+i) \vert du \leq \rev 330 \revv L_n.
\end{align*}
Now, Proposition \ref{Bai - Götze's Version} \rev applied to $\tau \in (1, 3]$, $\gamma <1$, and $a, \varepsilon$ as defined above \revv ends the proof: We have 
\begin{align*}
	\Delta(\mu_{\boxplus n}, \omega) \leq \rev C_\gamma \revv \left( \rev 971  \revv L_n + \rev 12 \revv ZL_n +  \rev 3 \revv Z^{\nicefrac{3}{2}}L_n\right) \leq c_0L_n
\end{align*}
for \rev $C_\gamma$ taken from Proposition \ref{Bai - Götze's Version} \revv and some numerical constant $c_0>0$.
\end{proof}


\end{document}